\documentclass{scrartcl}

\usepackage{pdfsync}
\usepackage{amsmath, amsthm, amssymb}
\usepackage{mathrsfs}
\usepackage{url}
\usepackage{amssymb,amsfonts}
\usepackage{color}

\usepackage{shuffle}
\usepackage[usenames,dvipsnames,table]{xcolor}
\usepackage{graphicx}
\usepackage[font={small,it}]{caption}
\usepackage[flushleft]{threeparttable}
\usepackage{longtable}
\usepackage{enumitem}
\usepackage{ulem}
\usepackage[makeroom]{cancel}

\usepackage{silence}
\usepackage{empheq}
\definecolor{myblue}{rgb}{1, 1, 1}
\newcommand*\mybluebox[1]{%
\colorbox{myblue}{\hspace{1em}#1\hspace{1em}}}

\makeatletter
\def\uwave{\bgroup \markoverwith{\lower3.5\p@\hbox{\sixly \textcolor{red}{\char58}}}\ULon}
\font\sixly=lasy6 % does not re-load if already loaded, so no memory problem.
\makeatother

\usepackage{mdframed}

\usepackage[top=3cm, bottom=3cm, left=2.5cm, right=3cm]{geometry}

\newtheorem{theorem}{Theorem}
\theoremstyle{plain}

\newtheorem{corollary}[theorem]{Corollary}

\newtheorem{example}[theorem]{Example}

\newtheorem{lemma}[theorem]{Lemma}

\newtheorem{proposition}[theorem]{Proposition}
\newtheorem{remark}[theorem]{Remark}

\theoremstyle{definition} 

\newtheorem{definition}[theorem]{Definition}
\newtheorem*{tata}{Generalization}
  {\begin{mdframed}[backgroundcolor=lightgray]\begin{tata}}%
  {\end{tata}\end{mdframed}}

\newcommand{\R}{\mathbb{R}}
\newcommand{\T}{\mathbb{T}}

\newcommand{\N}{\mathbb{N}}
\newcommand{\Z}{\mathbb{Z}}
\newcommand{\E}{\mathbb{E}}

\newcommand{\1}[1]{1_{\{#1\}}}

\newcommand{\mute}[1]{}

\newcommand*\laplacian{\mathop{}\!\mathbin\bigtriangleup}

\newcommand{\supp}{\operatorname{supp}}
\newcommand{\spann}{\operatorname{span}}

\renewcommand{\1}{\mathbf{1}}

\renewcommand{\symbol}[1]{#1}
\definecolor{pinkish}{RGB}{255, 79, 255}
\newcommand{\cotang}[1]{#1}
\newcommand{\vareps}{\varepsilon}

\newcommand{\Sym}{\operatorname{Sym}}
\newcommand{\p}{\mathsf{p}} % heat kernel, since we cannot use p (which is a point on M)
\newcommand{\proj}{\operatorname{proj}} % the projection operator

\begin{document}

\title{The parabolic Anderson model on Riemann surfaces}
%\author{Antoine Dahlqvist, Joscha Diehl, Bruce Driver}
\author{Antoine Dahlqvist\thanks{University of Cambridge; the author is responsible for the  first  part of the Appendix} \and Joscha Diehl\thanks{Max-Planck Institute Leipzig} \and Bruce Driver\thanks{University of California San Diego; the author is responsible for the second part of the Appendix}}

%\thanks{ Part of this research .. }
\maketitle

%\tableofcontents

\renewcommand{\1}{\mathbf{1}}
\newcommand{\I}{\mathcal{I}}
\newcommand{\J}{\mathcal{J}}
\newcommand{\D}{\mathcal{D}}
\newcommand{\monos}{\mathcal{T}}
\newcommand{\polys}{\bar{T}}

\newcommand{\mcD}{\mathcal{D}} % test functions
\newcommand{\msD}{\mathscr{D}} % modelled distributions

\newcommand{\mcR}{\mathcal{R}} % reconstruction

\newcommand{\mcT}{\mathcal{T}} % regularity structures
\newcommand{\mcV}{\mathcal{V}}
\newcommand{\mcW}{\mathcal{W}}
\newcommand{\mcG}{\mathcal{G}}

\newcommand{\K}{\mathcal{K}} % a compact set
\newcommand{\dV}{d\operatorname{Vol}} % volume form
\newcommand{\density}{d\operatorname{|Vol|}} % Riemannian density
\newcommand{\U}{\mathcal{U}} % open set
\newcommand{\dom}{\operatorname{dom}} % domain

\begin{abstract}
We show well-posedness for the parabolic Anderson model on $2$-dimensional closed Riemannian manifolds. To this end we extend the notion of regularity structures to curved space, and explicitly construct the minimal structure required for this equation. A central ingredient is the appropriate re-interpretation of the polynomial model, which we build up to any order.
\end{abstract}

\section{Introduction}

%% We study the parabolic Anderson model on 
%% on $2$ dimensional close Riemannian manifolds,
%% which is formally written as
%% \begin{align*}
%%   \partial_t u = \laplacian u + u \xi.
%% \end{align*}
%% Here $\xi$ is (space) white noise on $M$ and $\laplacian$ denotes the Laplac-Beltrami operator.
%% 
%% The additive noise equation
%% \begin{align*}
%%   \partial_t \nu = \laplacian \nu + \xi,
%% \end{align*}
%% is classically solved via convolution with the heat kernel and $\nu$ is H\"older continuous in space
%% with exponent strictly smaller than $1$.
%% This is not regular enough to define the product $\nu \xi$ in an analytic way \dots
%% 
%% two approaches
%% 
%% Hairer's theory of regularity structures \cite{bib:hairer}
%% 
%% the theory of paracontrolled distributions by Gubinelli/Imkeller/Perkowski \cite{bib:gubinelliImkellerPerkowski}

The last few years have seen an explosion of literature on singular stochastic partial differential equations (singular SPDEs).
The simplest instance of such an equation is the parabolic Anderson model in two dimensions, formally written as
\begin{align}
  \label{eq:introPAM}
  \partial_t u = \laplacian u + u \xi. \tag{PAM}
\end{align}
Here $u: [0,T] \times D \to \R$ is looked for, where $D$ is some $2$ dimensional domain,
and $\xi$ is (time-independent) white noise on the domain $D$.
This equation is formally ill-posed (or ``singular''), since $u$ is not expected to be regular
enough for the product $u \xi$ to be well-defined analytically.
The standard tool of stochastic calculus, the It\=o integral, is also of no use here,
since the white-noise is constant in time.

With the breakthrough results of Hairer \cite{bib:hairer}
and Gubinelli, Imkeller and Perkowski \cite{bib:gubinelliImkellerPerkowski}
a large class of such equations has become amenable to analysis.
Let us sketch the approach of \cite{bib:hairer}, since this is the one we shall use in this work.
\begin{itemize}
  \item assume that $u$ ``looks like`` the solution $\nu$ to the additive-noise equation %$\nu_t = \int_0^t P_s \xi ds$ to the additive-noise equation
    \begin{align}
      \partial_t \nu = \laplacian \nu + \xi,\label{Heat  with noise}
    \end{align}
    which is classically well-defined via convolution with the heat semigroup $P_t$
  \item under this assumption, if we somehow \emph{define} $\nu \cdot \xi$, then the framework defines $u \cdot \xi$ automatically
  \item close the fixpoint argument, i.e.
    \begin{enumerate}
      \item $u$ ''looks like`` $\nu$
      %\item $u \xi$ is well-defined
      \item $w := P_t u_0 + \int_0^t P_{t-s} [ u_s \xi ] ds$
      \item then $w$ ''looks like`` $\nu$
    \end{enumerate}
\end{itemize}
It then only remains to define the missing ingredient ''$\nu \cdot \xi$``.
This can be done probabilistically and is actually the only place in this theory that is not deterministic.
Using this procedure, it is shown in \cite{bib:hairer} that \eqref{eq:introPAM} possesses a unique solution for $D = \T^2$,
the two dimensional torus. %In \cite{bib:HL15} this has been extended to $D = \R^2$.

In this work we show that the theory can be adapted to work for $D = M$, a $2$-dimensional closed Riemannian manifold.
The theory of regularity structures is intrinsically a local theory
(as opposed to the theory of paracontrolled distributions, which, at least at first sight is global in spirit).
It is hence natural to expect that it can be applied to general geometries.
It turns out that the implementation of this heuristic is not straightforward.

At least two hurdles need indeed to be bypassed. On the one hand,  at the core
of Euclidean regularity structures stands the space of polynomials, encoding
classical Taylor expansions at any point.  The operation of re-expansion from a
point to another leads to  a morphism from $(\R^d,+)$ to a space of unipotent
matrices.  On a manifold, one would need  to look for such  a space,
encoding Taylor expansion and  enjoying  a similar structure. 
On the other hand,
as usual for fixpoint arguments of (S)PDEs,
one needs to estimate the improvement of the heat kernel
in adequate spaces, which is a global operation (Schauder estimates).

%the fixed point argument used to solve the SPDE at hand in
%the Euclidean setting relies on estimates for the solution map of (\ref{Heat
%with noise}), which is a global object. 

To solve the first issue,  we  show that the space  of polynomials on the tangent
space of  the manifolds is a suitable candidate for a canonical regularity
structure, that allows to encode H\"older functions.  This choice enforces a
modified definition of a regularity structure. In particular one has to abandon
the idea of one fixed vector space and work with vector bundles instead.
For our definition of a model, there is no unipotent structure anymore and re-expansions are only
approximately compatible.  Within this new framework, when considering the
parabolic Anderson model on a surface, we give a weak version of a  Schauder
estimate with elementary tools and heat kernel estimates.

This exposition does not demand any previous knowledge of regularity structures on the Reader.
In this sense it is self-contained, apart from a reference
to the reconstruction theorem of Hairer in our Theorem \ref{thm:reconstructionOnOpenSet-new} and in the construction
of the Gaussian model in Section \ref{sec:gaussian}.
Its proof using wavelet analysis is of no use reproducing here.
We believe that the validity of that reconstruction theorem, which we use in coordinates, is easily believed.

We follow a very hands-on approach. Instead of trying to set up a general theory of regularity structures on manifolds,
we work with the smallest structure that is necessary to solve PAM.
We show the Schauder estimates explicitly.
Apart from introducing for the first time regularity structures on manifolds,
we believe our work also has a pedagogical value. Since everyting is laid
out explicitly and covers the flat case $M = \T^2$,
it can serve as a gentle introduction to the general theory.

%The Schauder estimate is Taylor made for this particular equation
In future work we will investigate the algebraic foundation
necessary for studying general equations, without having to build the regularity structure ``by hand''.
For general equations a new proof of the Schauder estimates has also to be found.

During the writing of the present article, a different approached has been
put forward in \cite{bib:IB2016a}, where the notion of paracontrolled
products using semi-groups is developed on general metric spaces.
The advantage of the paracontrolled approach is that it requires less machinery.
On\ the downside, the class of equations that can be covered is currently strictly smaller
than in the setting of regularity structures.
Let us point to \cite{bib:IB2016b} though, which pushes the framework to more general equations.

The outline of this paper is as follows.
After presenting notational conventions,
we give in Section \ref{sec:holder} the notion
of distributions on manifold we shall use in this work.
Moreover we introduce H\"older spaces on manifolds.
In Section \ref{sec:rs} we introduce the notion of regularity structure, model
and modelled distribution on a manifold.
We show how these objects behave nicely under diffeomorphisms and use this fact to show
the reconstruction theorem.
In Section \ref{sec:poly} we give the simplest non-trivial example of a regularity structure
on a manifold; the regularity structure for ``linear polynomials''.
This forms the basis for the regularity structure for PAM, which is constructed
in Section \ref{sec:pam}.
As input it takes the product $\nu \xi$ alluded to before. This is constructed
in Section \ref{sec:gaussian} via renormalization.
Section \ref{sec:schauder} gives the Schauder estimate for modelled distributions
in the setting of PAM
and finally Section \ref{sec:fixpoint} solves the corresponding fixpoint equation.
In Section \ref{sec:higherOrderPolynomials}
we show how the construction of Section \ref{sec:poly}
can be extended to ``polynomials'' of arbitrary order.

\newcommand{\deltaf}{\delta/4} % a fraction of delta
\newcommand{\balloftestfunctions}{\mathcal{B}}

\subsection{Notation}
\label{ss:notation}
In all what follows $M$ will be a $d$-dimensional closed Riemannian manifold.
When we specialize to the parabolic Anderson model (PAM), the dimension will be $d=2$.
Denote by $\delta > 0$ the radius of injectivity of $M$.

For a function $\varphi$ supported in $B_{\delta}( 0 ) \subset \R^d,$ we define for $\lambda \in (0,1], p \in M,$
the ``scaled test function''
\begin{align}
  \label{eq:varphiplambda}
  \varphi_p^\lambda( \cdot ) := \lambda^{-d} \varphi( \lambda^{-1} \exp_p^{-1}( \cdot ) ),
\end{align}
extended to all of $M$ by setting it to zero outside of $\exp_p( B_\delta(0) )$.
%Note that for $M = \R^d$ this is just $\varphi^\lambda_p$ in the notation of Hairer.

For $\tau \in \mcG$, $\mcG$ a graded normed vector bundle with grading $A$ we
denote by $||\tau||_a$ the size of component in the $a$-th level, $a \in A$.

%\begin{remark}
%  {\color{red} we should now use $\vareps_\K$ \dots}
%
%  Test functions will usually be supported in a radius of $\deltaf$ around a point of consideration $x$.
%  Another point $y$ under consideration will usually have at distance at most $d(x,y) \le \deltaf$.
%  %
%  A distribution on $B_{\delta}(x)$ hence can act on a test function around $x$ \underline{and} on a test function around $y$.
%\end{remark}

The differential of a smooth enough function $f: M \to \R$ at a point $p$ will be denoted $d|_p f \in T^*_p M$.
Similar for higher order derivatives (see Section \ref{sec:higherOrderPolynomials})
$\nabla^\ell|_p f \in \left( T^*_p M \right)^{\otimes \ell}$.
For the action on vectors $W \in \left( T_p M \right)^\ell,$ we shall write
either $\langle \nabla^\ell|_p f, W \rangle$
or $\nabla_{W} f$.

For $\eta, r > 0$ denote
\begin{align}
  \label{eq:balloftestfunctions}
  \balloftestfunctions^{r,\eta}  
  :=
  \{ \varphi \in C^r(\R^d) : \supp \varphi \subset B_{\eta}(0), ||\varphi||_{C^r(\R^d)} \le 1 \},
\end{align}
where $B_{\eta}(0) := \{ x \in \R^d : |x| < \eta \}$.
Here $r$ will be depend  the situation, and will always be large enough so that the distributions
under consideration can act on $\varphi$.

We shall use $p,q$ for points in $M$
and $x,y,z$ to denote points in $\R^d$. For $x \in \R^d$, $\varphi: \R^d \to \R$ we write
\begin{align}
  \label{eq:varphixlambda}
  \varphi^\lambda_x := \lambda^{-d} \varphi( \lambda^{-1}(\cdot - x) ),
\end{align}
which is
consistent with the notation introduced above when considering $\R^d$ as
Riemannian manifold with the standard metric.

For $\gamma \in \R$ we denote by $\lceil \gamma \rceil$ the smallest integer strictly larger than $\gamma$.

For a pairing of a distribution $T$ with a test function we write $\langle T, \varphi \rangle$.

For two quantities $f,g$ we write $f \lesssim g$ if
there exists a constant $C> 0$ such that $f \le C g$.
To make explicit the dependence of $C$ on a quantity $h$, we sometimes write $f \lesssim_h g$.

%\section{Distributions on manifolds and change of coordinates}
%\label{sec:fDistribution}
\section{H\"{o}lder spaces}
\label{sec:holder}

\begin{definition}
  A \textbf{distribution} %in the sense of Friedlander (F-distribution)}
  on $M$ is a bounded, linear functional
  on $C^\infty_c( M )$ ($=C^\infty(M)$, if $M$ is compact).
\end{definition}

Given a density $\lambda$ on $M$,
$\langle T_\lambda, \varphi \rangle := \int_M \varphi d\lambda$ defines
a distribution.
Distributions are hence ``generalized densities``.
Compare
\cite[Section 2.8]{bib:friedlander}
and \cite[Section 1.3]{bib:waldmann}.

There is another definition of distibutions as ``generalized functions'', see \cite[Section 1.8]{bib:hormander}.
They are equivalent when there is a canonical way to turn a function into a density and vice versa.
This is the case when there is a reference density, like on a Riemannian manifold.
%and their difference amounts only to slightly different transformation rules.
\begin{remark}
  \label{rem:fDistributionFunction}
  On a Riemannian manifold $M$,
  denote the standard density by $\density$.
  We can lift a function $f \in C(M)$ to a density $f \density$.
  Then, for
  $f \in C^\infty_c(M)$, $T_f$ defined as
  \begin{align*}
    \langle T_f, \varphi \rangle
    :=
    \int_M f(z) \varphi(z) \density,
  \end{align*}
  is a distribution.
\end{remark}

\begin{definition}[Push-forward]
  Let $(\Psi, \U)$ be a coordinate chart on $M$.
  If $\varphi \in C^\infty_c( \Psi(\U) )$ and $T$ is a distribution on $M$ we can define the push-forward $\Psi_* T \in \mcD'(\Psi(\U))$
  via
  \begin{align*}
    \langle \Psi_* T, \varphi \rangle
    :=
    \langle T, \Psi^* \varphi \rangle
    :=
    \langle T, \varphi \circ \Psi \rangle.
  \end{align*}
\end{definition}

\begin{remark}
  This push-forward is compatible with the pull-back of densities.
  Indeed, for
  $f \in C(M)$ we get the distribution $T_f := f \density$, by Remark \ref{rem:fDistributionFunction}.
  This
  %$d$-form
  density
  pulls back under $\Psi^{-1}$
  as (compare \cite[Proposition 16.38]{bib:lee})
  \begin{align*}
    f \density
    \mapsto
    f \circ \Psi^{-1} \sqrt{\det g} dy^1 \wedge \dots\wedge dy^d,
  \end{align*}
  where $y^i$ are standard coordinates on $\R^d$, and $g$ is the Riemannian metric in the coordinates $\Psi$.
  Hence
  \begin{align*}
    \langle \Psi_* T_f, \varphi \rangle
    &=
    \langle T_f, \varphi \circ \Psi \rangle \\
    &=
    \int_M \varphi( \Psi(x) ) f(x) \density \\
    &=
    \int_{\R^d} \varphi( z ) f( \Psi^{-1}(z) ) \sqrt{\det g(z)} dz \\
    &=
    \langle T_{ (\Psi^{-1})^* f }, \varphi \rangle,
  \end{align*}
  where the last line is the pairing of a distribution with a test function on $\R^d$ and $T_h$ is the canonical identification
  of a locally integrable density $h$ on $\R^d$ with a distribution.
\end{remark}

Recall the following definition of H\"older spaces in Euclidean space.
\begin{definition}
  For $\gamma \le 0$ denote by
  $C^\gamma(\R^d)$ the space of distributions $T \in \mcD'(\R^d)$
  with
  \begin{align*}
    ||T||_{C^\gamma(\R^d)}
    :=
    \sup_{x \in \R^d} \sup_{\lambda \in (0,1]} \sup_{\varphi \in \balloftestfunctions^{r,1}}
      \lambda^{-\gamma} |\langle T, \varphi^\lambda_x \rangle| < \infty,
  \end{align*}
  here $r := \lceil |\gamma| \rceil$,
  $\varphi^\lambda_x$ is defined in \eqref{eq:varphixlambda}
  and the set of test functions $\balloftestfunctions^{r,1}$ is defined in \eqref{eq:balloftestfunctions}.
  %Here the supremum regarding $\varphi$ is over all
  %$\varphi \in C^r(\R^d)$, $r$ the smallest integer larger than $|\gamma|$,
  %with $||\varphi||_{C^r(\R^d)} \le 1$ and
  %$\supp \varphi^\lambda_x \subset U \cap B_\lambda(x)$.

  For $\gamma > 0$ we keep the classical definition, i.e.
  \begin{align*}
    ||T||_{C^\gamma(\R^d)}
    :=
    \sum_{|\ell| \le n}
    ||D^\ell T||_{\infty;\R^d}
    +
    \sup_{|\ell|=n; x,y \in \R^d} \frac{|D^\ell T(x) - D^\ell T(y)|}{|x-y|^{s}},
  \end{align*}
  where $\gamma = n + s$, $n \in \N$, $s \in (0,1]$.
\end{definition}

\begin{remark}
  \label{rem:independentOfLambda0}
  For $\gamma < 0$ the norm is independent of the arbitrary upper bound $1$ for the supremum over $\lambda$ as well as the support of $\varphi$.
  For every $\lambda_0, \varepsilon_0 > 0$
  \begin{align*}
    ||T||_{C^\gamma(\R^d)}
    \lesssim_{\lambda_0, \varepsilon_0}
    \sup_{x \in \R^d} \sup_{\lambda \in (0,\lambda_0]} \sup_{\varphi \in \balloftestfunctions^{r,\varepsilon_0}}
      \lambda^{-\gamma} |\langle T, \varphi^\lambda_x \rangle|,
  \end{align*}
  where $r := \lceil |\gamma| \rceil$.
\end{remark}

\begin{remark}
  \label{rem:221}
  Every time that a condition like
  \begin{align*}
    |\langle T, \varphi^\lambda_x \rangle| \lesssim \lambda^\gamma, 
  \end{align*}
  appears,
  uniformly over
  $\supp \varphi \subset B_{\varepsilon}(0),$ with 
  $||\varphi||_{C^r(B_\varepsilon(0))} \le 1,$
  %$\supp \varphi^\lambda_x \subset U \cap B_\lambda(x)$,
  one can equivalently demand
  \begin{align*}
    |\langle T, \varphi \rangle| \lesssim \lambda^\gamma, 
  \end{align*}
  uniformly over
  $\supp \varphi \subset B_{\lambda \varepsilon_0}(x)$, with 
  $||D^k \varphi||_\infty \lesssim \lambda^{-d - k}$, for $k=0,\dots,r$.
  %[Compare Remark 2.21 in \cite{bib:hairer}]
\end{remark}

We need a reformulation similar to this remark, but for Schwartz test functions.
\begin{lemma}
  \label{lem:rem221ForS}
  Let $\gamma \le 0$ and $T \in C^\gamma(\R^d)$. Then $T \in S'(\R^d)$ (and not just $T \in \mcD'(\R^d)$).
  Define for $r := \lceil |\gamma| \rceil$, $\varphi \in S(\R^d), x_0 \in \R^d, \lambda \in (0,1]$
  \begin{align*}
    C(\varphi,\lambda,x_0,N,r)
    :=
    \sup_{|k| \le r}
    \sup_{x \in \R^d}
      |D^k \varphi(x)| \lambda^{d+k} \left( 1 + \lambda^{-N} |x-x_0|^{N} \right).
  \end{align*}
  Then, for $N > d$,
  \begin{align*}
    |\langle T, \varphi \rangle| \lesssim_N C(\varphi,\lambda,x_0,N) ||T||_{C^\gamma(\R^d)} \lambda^\gamma.
  \end{align*}

  %Assume moreover that $\varphi \in S(\R^d), x_0 \in \R^d$ with
  %\begin{align*}
  %  D^k \varphi(x) \le C_\varphi \lambda^{-d-k} \frac{1}{1+ \lambda^{-N} |x-x_0|^{N}},
  %\end{align*}
  %all $k \le r$, some $N > d$ and some $x_0 \in \R^d$.
  %Then
  %\begin{align*}
  %  \langle T, \varphi \rangle \lesssim_N C_\varphi ||T||_{C^\gamma(\R^d)} \lambda^\gamma.
  %\end{align*}
\end{lemma}
\begin{remark}
  \label{rem:rem221ForS}
  Note that if $\varphi \in S(\R^d)$, then
  $\varphi^\lambda_{x_0} := \lambda^{-d} \varphi(\lambda^{-1}( \cdot - x_0) )$ satisfies
  for $\lambda \in (0,1]$, $r > 0$, $N \in \N$, $x_0 \in \R^d$
  %the condition of the Lemma.
  \begin{align*}
    C(\varphi^\lambda_{x_0},\lambda,x_0,N,r)
    \le
    \sup_{|k| \le r}
    \sup_{x \in \R^d}
      |D^k \varphi(x)| \left( 1 + |x|^{N} \right).
  \end{align*}
\end{remark}
\begin{proof}
  %{\color{green}
  %%Set
  %%\begin{align*}
  %%  \phi_{\lambda (\deltaf) z}(\cdot) := \tilde \phi_z( (\lambda\deltaf )^{-1} \cdot ).
  %%\end{align*}
  %%Then
  %%  $\sum_{r \in \lambda (\deltaf) \Z^d} \phi_r \equiv 1$ and
  %%  $\supp \phi_r \subset B_{\lambda\deltaf}(r)$.
  %Let $\tilde \lambda := \lambda \deltaf$.
  %Define
  %\begin{align*}
  %  \varphi_{z,\tilde \lambda} := \phi_z( \tilde \lambda^{-1} \cdot ) \varphi \qquad z \in \Z^d,
  %\end{align*}
  %so that
  %\begin{align*}
  %  \sum_{z \in \Z^d} \varphi_z = \varphi.
  %\end{align*}
  %Then $\supp \varphi_z \subset B_{\tilde \lambda}(\tilde \lambda z) \cap B_{\deltaf}(0)$ and
  %\begin{align*}
  %  ||D^k \varphi_z||_\infty
  %  &\lesssim 
  %  \lambda^{-d-k} \frac{1}{1 + |z|^{N}},
  %\end{align*}
  %}
  Let $\phi_z$, $z \in \Z^d$, be a partition of unity of $\R^d$ such that
  $\supp \phi_z \subset B_1(z)$ and $\sup_{z \in \Z^d} ||\phi_z||_{C^r} < \infty$.
  %
  %Then
  %\begin{align*}
  %  \sum_{z\in \Z} \phi_z(\lambda^{-1}\cdot) \equiv 1,
  %\end{align*}

  Define
  \begin{align*}
    \varphi_{z,\lambda}(\cdot) := \phi_z(\lambda^{-1} \cdot) \varphi(\cdot).
  \end{align*}
  Then $\sum_{z \in \Z^d} \varphi_{z,\lambda} = \varphi$.
  Write for short $C_\varphi := C(\varphi,\lambda,x_0,N,r)$.
  We have $\supp \varphi_{z,\lambda} \subset B_{\lambda}( \lambda z )$
  and %for $z \not= 0$
  \begin{align*}
    ||D^k \varphi_{z,\lambda}||_\infty
    &\lesssim C_\varphi \lambda^{-d-k} \frac{1}{1 + \lambda^{-N} |\lambda z - x_0|^{N}} \\
    &= C_\varphi \lambda^{-d-k} \frac{1}{1 + |z - \lambda^{-1} x_0|^{N}}.
    %\lesssim \lambda^{-1-k+N} \frac{1}{1 + |\lambda z - x_0|^{N}} \\
    %\lesssim \lambda^{-1-k} \frac{1}{\lambda^{-1}+|z - \lambda^{-1}x_0|^N}.
  \end{align*}
  %For $z=0$ it just scales like $\lambda$.
  Then
  \begin{align*}
    |\langle T, \varphi \rangle|
    &\le
    %\langle T, \varphi_0 \rangle
    %+
    \sum_{z \in \Z^d} |\langle T, \varphi_{z,\lambda} \rangle| \\
    &\le
    %\lambda^\gamma
    %+
    C_\varphi
    ||T||_{C^\gamma(\R^d)}
    \lambda^\gamma
    \sum_{z \in \Z^d}
    \frac{1}{1 + |z - \lambda^{-1} x_0|^{N}} \\
    %\frac{1}{\lambda^{-1}+|z - \lambda^{-1}x_0|^N} \\
    &\lesssim_N
    C_\varphi
    ||T||_{C^\gamma(\R^d)}
    \lambda^\gamma,
  \end{align*}
  as desired.
  We used the fact that 
  $\sum_{z \in \Z^d}
  \frac{1}{1 + |z - \lambda^{-1} x_0|^{N}}$
  is upper bounded by $\int_{\R^d} \frac{1}{1 + |z - \lambda^{-1} x_0|^{N}} dz = \int_{\R^d} \frac{1}{1 + |z|^{N}} dz$,
  which is finite, since $N>d,$ and independent of $\lambda$.
\end{proof}

\begin{definition}
  \label{def:cgamma}
  Let $M$ be a closed Riemannian manifold.
  Let a finite partition of unity $(\phi_i)_{i\in I}$ be given on $M$, subordinate to a finite atlas $(\Psi_i, U_i)_{i\in I}$.
  For $\gamma \in \R,$ define
  \begin{align*}
    C^\gamma(M) := C^\gamma(M;(\Psi_i,U_i),\phi_i) := \{  f: M \to \R: (\Psi_i)_* (\phi_i f) \in C^\gamma( \R^d ), \quad \forall i\in I \},
  \end{align*}
  and 
  \begin{align*}
    ||f||_\gamma := \sup_{i\in I} ||(\Psi_i)_* (\phi_i f)||_{C^\gamma( \R^d )}.
  \end{align*}
\end{definition}

For $\gamma > 0,$ an equivalent characterization of $C^\gamma(M)$ will be shown
in Theorem \ref{thm:cgammaDgamma}. We now give one in the case $\gamma \le 0$.
\begin{lemma}
  \label{lem:cgammaViaExp}
  For $\gamma \le 0$, $M$ a closed Riemannian manifold, an equivalent norm on $C^\gamma(M)$ is given by
  \begin{align*}
    \sup_{p \in M, \lambda \in (0,1], \varphi \in \balloftestfunctions^{r,\delta}} \frac{|\langle T, \varphi^\lambda_p \rangle|}{\lambda^\gamma},
  \end{align*}
  where we recall that $\varphi^\lambda_p$ is defined in \eqref{eq:varphiplambda}.
\end{lemma}
\begin{proof}
  Fix an atlas $(\Psi_i, U_i)$ with subordinate partition of unity $\phi_i$.
  Denote
  \begin{align*}
    C_1 &:= ||T||_{C^\gamma(M;(\Psi_i,U_i),\phi_i} \\ 
    C_2 &:= \sup_{p \in M, \lambda \in (0,1], \varphi \in \balloftestfunctions^r} \frac{|\langle T, \varphi^\lambda_p \rangle|}{\lambda^\gamma}.
  \end{align*}

  ($C_1 \ge C_2$):
  Let $\varphi \in \balloftestfunctions^r$, $p \in M$.
  Then
  \begin{align*}
    \langle T, \varphi^\lambda_p \rangle
    &=
    \sum_i \langle T \phi_i, \varphi^\lambda_p \rangle \\
    &=
    \sum_i \langle (\Psi_i)_*( T \phi_i ), \varphi^\lambda_p \circ \Psi^{-1}_i \rangle \\
    &=
    \sum_i \langle (\Psi_i)_*( T \phi_i ), \tilde{\phi}_i \circ \Psi^{-1}_i \varphi^\lambda_p \circ \Psi^{-1}_i \rangle \\
    &=:
    \sum_i \langle (\Psi_i)_*( T \phi_i ), \eta_i \rangle.
  \end{align*}
  Here $\tilde{\phi}_i$ is such that $\supp \tilde \phi_i \subset U_i$ and $\phi_i \tilde{\phi_i} = \phi_i$.
  Now
  \begin{align*}
    |\supp \eta_i| \le |\Psi_i\left( B_{c \lambda}( p ) \cap \tilde{\phi}_i \right)| \lesssim \lambda.
  \end{align*}
  Indeed, this follows from
  \begin{align*}
    \frac{1}{c} I \le D\left[ \exp^{-1}_p \circ \Psi^{-1}_i \right](x),
  \end{align*}
  for some constant $c > 0$, for all $i$, and $x \in \supp \tilde{\phi}_i \circ \Psi^{-1}_i$.
  Moreover
  \begin{align*}
    |D^k\left[ \exp^{-1}_p \Psi^{-1}_i \right](x)| \lesssim_k 1,
  \end{align*}
  for all $i$.
  Hence
  \begin{align*}
    |D^k \eta_i(x)| \lesssim C_1 \lambda^{-d - k}.
  \end{align*}
  The result then follows from Remark \ref{rem:221}.

  ($C_1 \le C_2$):
  %So
  %\begin{align*}
  %  C_1 := \sup_{p \in M \lambda \in (0,1], \varphi \in \balloftestfunctions} \frac{|\langle T, \varphi^\lambda_p \rangle|}{\lambda^\gamma} < \infty.
  %\end{align*}
  %Let $(\Psi_i, U_i)$ the atlas with subordinate partition of unity $\phi_i$ be given.
  We have to show
  \begin{align*}
    |\langle (\Psi_i)_*\left( T \phi_i \right), \varphi^\lambda_x \rangle| \lesssim C_2 \lambda^\gamma,
  \end{align*}
  for all $\supp \varphi \subset B_1(0)$, $||\varphi||_{C^r} \le 1$.
  Now for $x \in \R^d$
  \begin{align*}
    \langle (\Psi_i)_*\left( T \phi_i \right), \varphi^\lambda_x \rangle 
    &=
    \langle T, \phi_i \varphi^\lambda_x \circ \Psi_i \rangle \\
    &=
    \langle T, \left( \phi_i \varphi^\lambda_x \circ \Psi_i \right) \circ \exp_{\Psi^{-1}_i(x)} \circ \exp^{-1}_{\Psi^{-1}_i(x)} \rangle,
  \end{align*}
  where the last equality holds if $\supp h_{i;\lambda;x} \subset B_{\delta/2}( \Psi^{-1}_i(x) )$,
  with $h_{i;\lambda;x} := \left( \phi_i \varphi^\lambda_x \circ \Psi_i \right)$.

  Claim: there exists $\lambda_0 > 0$ such that for all $\lambda \le \lambda_0$, $i \in I$, $x \in \R^d$
  either $h_{i;\lambda;x} = 0$
  or $\supp h_{i;\lambda;x} \subset B_{\delta/2}( \Psi^{-1}_i(x) )$.
  Indeed, since $I$ is finite and for all $i\in I$, $\phi_i$ is compactly supported in $U_i$
  there exists $\vareps_0, \vareps_1 > 0$ such that
  if $\lambda < \vareps_0$
  and if $h_{i;\lambda;x} \not\equiv 0$ then $d(x, \partial \Psi_i(U_i) > \vareps_1$.
  Away from the boundary, the differential of $\Psi_i$ is bounded,
  and then for $z \in \supp h_i$ one has $d(z,\Psi_i^{-1}(x)) = O(\lambda)$.
  This proves the claim.
  
  Now one checks that $h_{i;\lambda} \circ \exp_{\Psi^{-1}_i(x)}$ falls under Remark \ref{rem:221},
  and applies Remark \ref{rem:independentOfLambda0}.
\end{proof}

As immediate consequence we get the following statement.
\begin{corollary}
  \label{lem:cgammaInEveryAtlas}
  Let $(\bar \Psi_j, \bar U_j)_{j\in J}$ be another finite atlas with subordinate partition of unity $(\bar \phi_j)_{j\in J}$.
  Then for $\gamma \le 0$
  \begin{align*}
    C^\gamma(M;(\Psi_i,U_i),\phi_i)
    =
    C^\gamma(M;(\bar \Psi_j,\bar U_j),\bar \phi_j)
  \end{align*}
  with equivalent norms.
\end{corollary}

%%%% \begin{lemma}
%%%%   Another equivalent norm is \dots wavelets on coordinate patches \dots
%%%% \end{lemma}
%%%% 
%%%% {\color{red} 
%%%%   We also need a lemma showing the reverse direction for models;
%%%%   i.e. that one take the norm using a partition of unity.
%%%% 
%%%%   What exactly do we need?
%%%% }

\section{Regularity structures on manifolds}
\label{sec:rs}

Let $M$ be a $d$-dimensional Riemannian manifold without boundary.
The two cases we are most interested in are
\begin{itemize}
  \item $M$ is compact without boundary (i.e. closed)
  \item $M$ is an open bounded subset of $\R^d$ with induced Euclidean metric
\end{itemize}

We now give our definition of a regularity structure and a model on a manifold $M$.
For concrete incarnations of these abstract definitions we refer
the reader to Section \ref{sec:poly} for the implementation of a first order ``polynomial'' structure;
to Section \ref{sec:higherOrderPolynomials} for a structure implementing ``polynomials'' of any order
and right before Lemma \ref{lem:theseAreInFactModels} for the structure used for the parabolic Anderson model.

\begin{definition}[Regularity structure]
  \label{def:regularityStructure}
  A \textbf{regularity structure}
  is a graded vector bundle $\mcG$ on $M$,
  with a finite grading $A = A(\mcG) \subset \R$.
  %such that for every $\gamma \in \R$, $\mcG_{< \gamma}$ is finite dimensional
  %and for some $\alpha \in \R$, $\mcG_{<\alpha} = \emptyset$.
  %Such that for every $\alpha \in A$, $T_\alpha$ is a finite
  For $\alpha \in A$, $\mcG_\alpha$ denotes the vector bundle of homogeneity $\alpha$.
  It is assumed to be finite dimensional. We denote the fiber at $p \in M$
  by $\mcG|_p$ and the fiber of homogeneity $\alpha$ at $p$ by $\mcG_\alpha|_p$.
  For $p \in M, \tau \in \mcG|_p$, $\alpha \in A$ we write $\proj_{\mcG_\alpha|_p}$
  for the projection of $\tau$ onto $\mcG_\alpha|_p$.
  %Here we denote for $\gamma \in \R$ by $\mcG_{< \gamma}$
  %the restriction of $\mcG$ to the homogeneities $\ell \in A$ with $\ell < \gamma$.
  %For $\ell \in A$, $\mcG_\ell$ denotes the restriction to homogeneity $\ell$.
\end{definition}

%There are a few differences to the classical definition of a regularity structure in \cite{bib:hairer}.
%Firstly we consider $\mcG$ to be a vector bundle instead of one fixed vector space.
%This stems from the necessity to store derivatives of smooth enough functions
%in the structure, which on a manifold is only possible with a fibered structrue
%(see Section \ref{sec:poly} for details).
%
%\dots..

\begin{definition}[Model]
  \label{def:modelGeneral}
  Let a collection of open sets $\U_q \subset M$, $q \in M$, with $q \in \U_q$,
  and maps
  \begin{align*}
    \Pi_q: \mcG|_q &\to D'( \U_q ) \\
    \Gamma_{p \leftarrow q}: \mcG|_q &\to \mcG|_p,
  \end{align*}
  be given.
  We assume there is for every compactum $\K \subset M$
  a constant
  $\delta_\K = \delta_\K(\Pi,\Gamma,\{\U_q\}_q) > 0,$ such that $\Gamma_{p \leftarrow q}$ is defined for $p,q \in \K, d(p,q) < \delta_\K$
  and for $q \in \K$,
  $\exp_q|_{B_{\delta_\K}(0)}$ is a diffeomorphism and
  $\exp_q(B_{\delta_\K}(0)) \subset \U_q$.

  %$\supp \varphi \subset B_\delta(0) \subset \R^d$
  %and $||\varphi||_{C^r} \le 1$.

  Given $\beta \in \R$,
  we say that $(\Pi, \Gamma)$ is a \textbf{model with transport precision $\beta$} if
  the following entity is finite
  for every compactum $\K \subset M$
  \begin{align*}
    ||\Pi,\Gamma||_{\beta;\K}
    &:=
    \sup_{p\in \K, \ell \in A(\mcG), \tau \in \mcG_\ell|_p, \lambda \in (0,1], \varphi \in \balloftestfunctions^{r,\delta_\K}}
    \frac{ |\langle \Pi_p \tau, \varphi_p^\lambda \rangle| }{ \lambda^\ell ||\tau|| } \\
    &\qquad
    +
    \sup_{p,q \in \K: d(p,q) < \delta_\K / 2, \tau \in \mcG|_q, \lambda \in (0,1], \varphi \in \balloftestfunctions^{r,\delta_\K/2}}
    \frac{ |\langle \Pi_q \tau - \Pi_p \Gamma_{p \leftarrow q} \tau, \varphi_p^\lambda \rangle| }{ \lambda^\beta ||\tau|| } \\
    &\qquad
    +
    \sup_{ \ell \in A(\mcG); m ; p, q \in \K, d(p,q) < \delta_\K; \tau \in \mcG_\ell|_q}
    \frac
    { |\Gamma_{p \leftarrow q} \tau|_m }
    {d(p,q)^{(\ell - m) \vee 0} |\tau| },
  \end{align*}
  where we recall that the set of test functions $\balloftestfunctions^{r,\delta}$ was defined in \eqref{eq:balloftestfunctions}.

  %Define also
  %\begin{align*}
  %  &||(\Pi,\Gamma);(\bar \Pi,\bar \Gamma)||_{\beta;\gamma;\K} \\
  %  &:=
  %  \sup_{p\in \K, \ell \in A(\mcG), \tau \in \mcG_\ell|_p, \lambda \in (0,1], \varphi \in \balloftestfunctions^{r, \delta_\K \wedge \bar \delta_\K}}
  %  \frac{ |\langle \Pi_p \tau - \bar \Pi_p \tau, \varphi_p^\lambda \rangle| }{ \lambda^\ell ||\tau|| } \\
  %  &\qquad
  %  +
  %  \sup_{p,q\in \K, d(p,q) < \delta_\K \wedge \bar \delta_\K / 2, \ell \in A(\mcG), \tau \in \mcG_\ell|_p, \lambda \in (0,1], \varphi \in \balloftestfunctions^{r, \delta_\K \wedge \bar \delta_\K / 2}}
  %  \frac{ |\langle \Pi_q \tau - \Pi_p \Gamma_{p \leftarrow q} \tau - (\bar \Pi_q \tau - \bar \Pi_p \bar \Gamma_{p \leftarrow q} \tau), \varphi^\lambda_p \rangle| }
  %       { \lambda^{\beta} ||\tau|| } \\
  %  &\qquad
  %  +
  %  \sup_{ \ell < \gamma; m < \ell; p, q \in \K, d(p,q) < \delta_K \wedge \delta_\K; \tau \in \mcG_q|_\ell}
  %  \frac
  %  { |\Gamma_{p \leftarrow q} \tau - \bar \Gamma_{p \leftarrow q} \tau|_m }
  %  {d(p,q)^{\ell - m} |\tau| }.
  %\end{align*}
\end{definition}

\begin{remark}
  Note that the conditions on a model do not pin down the global regularity
  of $\Pi_q \tau$.
  Without loss of generality we will assume that $\Pi_q \tau \in C^\alpha(U_q)$
  for all $q \in M, \tau \in \mcG|_q$
  and $\alpha := \min A(\mcG)$.
\end{remark}

%Let us make some remarks about the differences with the classical definition of a regularity structure and of a model in \cite{bib:hairer}.
%Instead of 
Our definition of a regularity structure and a corresponding model
are slightly more general than the original formulation by Hairer \cite{bib:hairer}.
This extension is necessary to accomodate the ``polynomial regularity structure'',
which will be constructed up to first order in Section \ref{sec:poly} and
up to any order in Section \ref{sec:higherOrderPolynomials}.
Let us point out the key differences.
\begin{itemize}
  \item %Unless the manifold is parallelizable,
    Derivatives of functions on a general manifold $M$
    can only be stored in a fibered space.
    Hence the regularity structure has to be a vector bundle and not a fixed vector space.

  \item 
    %There is no fixed structure group $G$.
    For this reason there cannot be a fixed structure group $G$ in which the transport maps $\Gamma_{p \leftarrow q}$ take value.

  \item 
    The transport maps $\Gamma_{p \leftarrow q}$ can also act ``upwards'', see Remark \ref{rem:upwards}.

  \item
    The distributions $\Pi_p \tau$ as well as the transports $\Gamma_{p \leftarrow q}$ only make sense locally.
\end{itemize}

It turns out that the theory can handle these slight extensions. In particular
the reconstruction theorem still holds, Theorem
\ref{thm:reconstructionOnTheManifold-new}.
Finally, we remark that our regularity structure does not include time
and that the parabolic Anderson model will be treated by considering functions in time, valued
in modelled distributions (Definition \ref{def:modelledDistributions}) on a manifold.

As in Lemma \ref{lem:rem221ForS} we know how $\Pi_p \tau$ acts on a more general class of functions:
\begin{lemma}
  \label{lem:rem221ForTheModel}
  For a regularity structure $\mcG$ let be
  given a model $(\Pi,\Gamma)$ of transport precision $\beta$ with $\beta \ge \sup_{a \in A(\mcG)} |a|$.
  Let $p \in \K$, a compactum in $M$.
  Let $\varphi$ satisfy the assumptions of Lemma \ref{lem:rem221ForS}
  with the additional condition
  $\supp \varphi \subset B_{\delta_\K / 4}(0) \subset \R^d$.
  Assume moreover that $B_{\delta_\K / 2}(p) \subset \K$ (which can always be achieved by making $\delta_\K$ smaller.)
  Then for $\tau \in \mcG_\ell|_p$
  \begin{align*}
    |\langle \Pi_p \tau, \varphi \circ \exp^{-1}_p \rangle|
    \lesssim
    C_\varphi
    ||\Pi,\Gamma||_{\beta,\K} \lambda^\ell,
  \end{align*}
  where $C_\varphi := C(\varphi,\lambda,0,N,r)$ is defined in Lemma \ref{lem:rem221ForS}.
\end{lemma}
\begin{proof}
  Let $\phi_z$, $z \in \Z^d$, be a partition of unity of $\R^d$ such that
  $\supp \phi_z \subset B_1(z)$ and $\sup_{z \in \Z^d} ||\phi_z||_{C^r} < \infty$.
  Let $\lambda_\K := \lambda \delta_\K / 4$.
  Define
  \begin{align*}
    \varphi_{z,\lambda_\K} := \phi_z( \lambda_\K^{-1} \cdot ) \varphi \qquad z \in \Z^d,
  \end{align*}
  so that
  \begin{align*}
    \sum_{z \in \Z^d} \varphi_{z,\lambda_\K} = \varphi.
  \end{align*}
  Then $\supp \varphi_{z,\lambda_\K} \subset B_{\lambda_\K}(\lambda_\K z) \cap B_{\delta_\K / 4}(0)$.
  Hence $\varphi_{z,\lambda_\K} \equiv 0$ for $|\lambda_\K z| \ge \delta_\K / 2$.
  Moreover
  \begin{align*}
    ||D^k \varphi_{z,\lambda_\K}||_\infty
    &\lesssim 
    C_\varphi \lambda^{-d-k} \frac{1}{1 + |z|^{N}}.
  \end{align*}

  Then
  \begin{align*}
    \langle \Pi_p \tau, \varphi \circ \exp^{-1}_p \rangle
    &=
    \sum_{z \in \Z^d} \langle \Pi_p \tau, \varphi_{z,\lambda_\K} \circ \exp^{-1}_p \rangle \\
    &=
    \sum_{z \in \Z^d, |\lambda_\K z| < \delta_\K / 2} \langle \Pi_p \tau, \varphi_{z,\lambda_\K} \circ \exp^{-1}_p \rangle \\
    &=
    \sum_{z \in \Z^d, |\lambda_\K z| < \delta_\K / 2}
      \langle \Pi_{\exp_p(\lambda_\K z)} \Gamma_{\exp_p(\lambda_\K z) \leftarrow p} \tau, \varphi_{z,\lambda_\K} \circ \exp^{-1}_p \rangle \\
      &\qquad
      +
      \langle \Pi_p \tau - \Pi_{\exp_p(\lambda_\K z)} \Gamma_{\exp_p(\lambda_\K z) \leftarrow x} \tau, \varphi_{z,\lambda_\K} \circ \exp^{-1}_p \rangle \\
  \end{align*}
  Note that in the sum $|\lambda_\K z| < \delta_\K / 2$. Hence, by assumption $q := \exp_p(\lambda_\K z ) \in \K$.
  Hence by definition of a model, $|\Gamma_{p \leftarrow q} \tau|_m \le ||\Pi,\Gamma||_{\beta,\K} d(p,q)^{(\ell-m)\vee 0}$ for $\tau \in \mcG_\ell|_q$.
  Then for those $z$
  \begin{align*}
    |\langle \Pi_{\exp_p(\lambda_\K z)} \Gamma_{\exp_p(\lambda_\K z) \leftarrow x} \tau, \varphi_{z,\lambda_\K} \circ \exp^{-1}_p \rangle|
    &\le
    \sum_{n \le \ell}
    |\langle \Pi_{\exp_p(\lambda_\K z)} \operatorname{proj}_{\mcT_n} \Gamma_{\exp_p(\lambda_\K z) \leftarrow x} \tau, \varphi_{z,\lambda_\K} \circ \exp^{-1}_p \rangle| \\
    &\qquad
    +
    \sum_{n > \ell}
    |\langle \Pi_{\exp_p(\lambda_\K z)} \operatorname{proj}_{\mcT_n} \Gamma_{\exp_p(\lambda_\K z) \leftarrow x} \tau, \varphi_{z,\lambda_\K} \circ \exp^{-1}_p \rangle| \\
    &\lesssim
    C_\varphi
    ||\Pi,\Gamma||_{\beta,\K}
    \left(
      \sum_{n \le \ell} \lambda^n \frac{1}{1 + |z|^N} |\lambda z|^{\ell-n}
      +
      \sum_{n > \ell} \lambda^n \frac{1}{1+ |z|^N}
    \right).
  \end{align*}
  Moreover
  % TRANSPORT-PRECISION-USED
  \begin{align*}
    |\langle \Pi_p \tau - \Pi_{\exp_p(\lambda_\K z)} \Gamma_{\exp_p(\lambda_\K z) \leftarrow x} \tau, \varphi_{z,\lambda_\K} \circ \exp^{-1}_p \rangle|
    &\lesssim
    C_\varphi
    ||\Pi,\Gamma||_{\beta,\K}
    \lambda^\beta \frac{1}{1+|z|^N}
  \end{align*}
  Combining,
  \begin{align*}
    |\langle \Pi_p \tau, \varphi \circ \exp^{-1}_p \rangle|
    &\lesssim
    C_\varphi
    ||\Pi,\Gamma||_{\beta,\K}
    \sum_{z \in \Z^d}
    \left(
      \lambda^\ell \sum_{n \le \ell} \frac{1}{1 + |z|^N} |z|^{\ell-n}
      +
      \lambda^\ell \frac{1}{1+ |z|^N}.
      +
      \lambda^\beta \frac{1}{1+|z|^N}
    \right) \\
    &\lesssim
    C_\varphi
    ||\Pi,\Gamma||_{\beta,\K}
    \lambda^\ell.
  \end{align*}
\end{proof}

\begin{definition}
  \label{def:modelledDistributions}
  Let $\mcG$ be a regularity structure and $(\Pi,\Gamma)$ a model of precision $\beta \in \R$.
  Define for $\gamma > \sup_{\alpha \in A(\mcG)} |\alpha|$ the space of \textbf{modelled distributions}
  \begin{align*}
    \msD^\gamma(M,\mcG) := \{ f: M \to \mcG : f \text{ is a section of }\mcG, ||f||_{\msD^\gamma(\K,\mcG)} < \infty \text{ for all compacta } \K \subset M \}.
  \end{align*}
  with
  \begin{align*}
    ||f||_{\msD^\gamma(\K,\mcG)}
    := 
    \sum_{\ell < \gamma} \sup_{p \in \K} |f(p)|_\ell + \sup_{\ell < \gamma} \sup_{p,q \in \K, d(p,q) < \delta_\K } \frac{|f(p) - \Gamma_{p \leftarrow q} f(q)|_\ell}{d(p,q)^{\gamma - \ell}}.
  \end{align*}
  Here $\delta_\K$ is the distance of points in $\K$ for which $\Gamma$ makes sense, see Definition \ref{def:modelGeneral}.
  Note that the precision of transport $\beta$ plays no role here.
\end{definition}

\begin{remark}
  \label{rem:mcDCutoff}
  As usual for H\"older norms,
  for every compactum $\K$
  an equivalent norm is obtained by replacing in the supremum, for any $\delta' \in (0,\delta_\K]$, the condition $d(p,q) < \delta_\K$
  with the condition $d(p,q) < \delta'$.
\end{remark}

%%\begin{lemma}[Restriction]
%%  Let $\mcG$ be a regularity structure
%%  on $M$ and $(\Pi,\Gamma)$ a model with transport precision $\beta$.
%%
%%  Let $N \subset M$ be an open subset.
%%  Define ..
%%
%%  Then .. is a regularity structure and .. is a model with transport
%%  precision $\beta$.
%%\end{lemma}
%%\begin{proof}
%%  This is immediate from definition
%%  via contained compacta.
%%
%%  Note though that
%%  for $\K \subset N \subset M$ compact,
%%  for example the norm
%%  $||\Pi||_{\beta;\K}$ considered as subset of $N$
%%  is in general smaller than
%%  $||\Pi||_{\beta;\K}$ considered as subset of $M$.
%%  In the latter a larger set of testfunctions is allowed.
%%\end{proof}

\begin{lemma}[Push-forward]
  \label{lem:pushForward-new}
  Let $M, N$ be Riemannian manifolds.
  Let $\Psi: M \to N$ a diffeomorphism.

  Let $\mcG$ be a regularity structure
  on $M$ with model $(\Pi,\Gamma)$ with transport precision $\beta \in \R$.
  Define
  \begin{align*}
    \bar \U_q &:= \Psi( \U_{\Psi^{-1}(q)} ), \\
    \bar \mcG|_q &:= \mcG_{\Psi^{-1}(q)}, \qquad q \in N, \\
    \bar \Gamma_{p \leftarrow q} &:= \Gamma_{\Psi^{-1}(p) \leftarrow \Psi^{-1}(q)}, \qquad p,q \in N,  \\
    \bar \Pi_q \tau &:= \Psi_* \Pi_{\Psi^{-1}(q)} \tau, \qquad q \in N, \tau \in \bar \mcG|_q.
  \end{align*}%
  %\footnote{Unwrapping the definition of $\bar \Pi_q \tau$
  %\begin{align*}
  %  \langle \bar \Pi_q \tau, \varphi \rangle := \langle \Pi_{\Psi^{-1}(q)}, \varphi \circ \Psi \rangle,
  %\end{align*}
  %which is well-defined for $\supp \varphi \subset \bar \U_q$.}

  Then, $\bar \mcG$ is a regularity structure on $N$ with grading $\bar A = A$
  and $(\bar \Pi, \bar \Gamma)$ is a model with transport precision $\beta$.
  Moreover
  \begin{enumerate}
    \item 
      \begin{align*}
        ||\bar \Pi, \bar \Gamma||_{\beta,\K} &\lesssim ||\Pi,\Gamma||_{\beta,\Psi^{-1}(\K)} \\
  %\textcolor{red}{ ||(\bar \Pi, \bar \Gamma); (\bar \Pi', \bar \Gamma')||_{\beta,\K} }&\lesssim ||\Pi,\Gamma;\Pi',\Gamma'||_{\beta,\Psi^{-1}(\K)}.
      \end{align*}
    \item Let $f, f' \in \msD^\gamma( M, \mcG )$ and define
      $\tilde f(x) := f( \Psi^{-1}(x) ), \tilde f'(x) := f'(\Psi^{-1}(x))$. Then $\tilde f, \tilde f' \in \msD^\gamma(\Psi(\U), \mcG)$
      and
      \begin{align*}
        ||\bar f||_{\msD^\gamma(\K, \mcG)} &\lesssim ||f||_{\msD^\gamma( \Psi^{-1}(\K), \mcG )} \\
       ||\tilde f - \tilde f'||_{\msD^\gamma(\K, \mcG)} &\lesssim ||f - f'||_{\msD^\gamma( \Psi^{-1}(\K), \mcG )}.
      \end{align*}
  \end{enumerate}
\end{lemma}
\begin{proof}
  Since $\Psi$ has derivatives bounded below and above for every compactum,
  one can choose for every compactum $\K$ a constant $\bar \delta_\K$ as in the definition of a model, such that $\bar \Gamma_{p \leftarrow q}$ is well-defined
  for $p,q \in \K$ and $d(p,q) < \delta_\K$
  as well as $\exp^N_q( B_{\delta_\K}(0) ) \subset \bar \U_q$.
  Here $\exp^N$ denotes the exponential map on $N$. 

  1. 
  %\textbf{distribution}\\
  Let $q \in \K \subset N$ and $\tau \in \tilde \mcG_a|_q$
  and $\varphi \in \balloftestfunctions^{r, \delta_\K}$
  \begin{align*}
    |\langle \tilde \Pi_q \tau, \varphi^\lambda_q \rangle|
    &=
    |\langle \Psi_* \left( \Pi_{\Psi^{-1}(q)} \tau \right), \varphi^\lambda_q \rangle| \\
    &=
    |\langle \Pi_{\Psi^{-1}(q)} \tau, \varphi^\lambda_q \circ \Psi \rangle| \\
    &=
    |\langle \Pi_{\Psi^{-1}(q)} \tau, \varphi^\lambda_q \circ \Psi \circ \exp_q \circ \exp^{-1}_q \rangle| \\
    &\lesssim ||\Pi,\Gamma||_{\beta,\Psi^{-1}(\K)} \lambda^a,
  \end{align*}
  since $\varphi^\lambda_q \circ \Psi \circ \exp_q$ falls under Remark \ref{rem:221}.
  For $p,q \in \K \subset N$ with $d(p,q) < \delta_\K$
  and $\tau \in \tilde \mcG|_q$, we have
  % TRANSPORT-PRECISION-USED
  \begin{align*}
      |\langle \tilde \Pi_q \tau - \tilde \Pi_p \tilde \Gamma_{x \leftarrow y} \tau, \varphi_p^\lambda \rangle|
      &=
      |\langle \Psi_*\left( \Pi_{\Psi^{-1}(q)} \tau - \Pi_{\Psi^{-1}(p)} \Gamma_{\Psi^{-1}(p) \leftarrow \Psi^{-1}(q)} \tau \right), \varphi_p^\lambda \rangle| \\
      &=
      |\langle \Pi_{\Psi^{-1}(q)} \tau - \Pi_{\Psi^{-1}(p)} \Gamma_{\Psi^{-1}(p) \leftarrow \Psi^{-1}(q)} \tau, \varphi_p^\lambda \circ \Psi \rangle| \\
      &=
      |\langle \Pi_{\Psi^{-1}(q)} \tau - \Pi_{\Psi^{-1}(p)} \Gamma_{\Psi^{-1}(p) \leftarrow \Psi^{-1}(q)} \tau, \varphi_p^\lambda \circ \Psi \circ \exp_p \circ \exp^{-1}_p \rangle| \\
      &\lesssim
      ||\Pi,\Gamma||_{\beta,\Psi^{-1}(\K)}
      \lambda^\beta,
  \end{align*}
  again by Remark \ref{rem:221}.
  %
  %The bounds on the norms are also immediate.
  Finally
  for $p,q \in \K \subset N$ with $d(p,q) < \delta_\K$
  and $\tau \in \tilde \mcG_a|_q$, we have
  \begin{align*}
    |\bar \Gamma_{p \leftarrow q} \tau|
    &=
    |\Gamma_{\Psi^{-1}(p) \leftarrow \Psi^{-1}(q)} \tau|_m
    \lesssim
    ||\Pi,\Gamma||_{\beta,\Psi^{-1}(\K)}
    d(p,q)^{ (\ell-m) \vee 0}.
  \end{align*}

  2.  Let $p,q \in \K \subset N$ then
  \begin{align*}
    ||\bar f(q) - \bar \Gamma_{p\leftarrow q} \bar f(q)||_m
    &=
    ||f( \Psi^{-1}(q) ) - \Gamma_{\Psi^{-1}(p) \leftarrow \Psi^{-1}(q)} f( \Psi^{-1}(q) )||_m \\
    &\lesssim
    ||f||_{\mcD^{\gamma}(\Psi^{-1}(\K),\mcG)}
    d( \Psi^{-1}(p), \Psi^{-1}(q) )^{\gamma - m} \\
    &\lesssim
    ||f||_{\mcD^{\gamma}(\Psi^{-1}(\K),\mcG)}
    d( p, q )^{\gamma - m},
  \end{align*}
  and similarily for the distance of two modelled distributions.
\end{proof}

\begin{lemma}[{Reconstruction for $M \subset \R^d$}]
  \label{thm:reconstructionOnOpenSet-new}
  Let $\mcG$ be a regularity struture on $M$, an open connected subset of $\R^d$.
  Let $(\Pi, \Gamma)$ be a model with precision $\beta \in \R$.
  Let $\gamma > 0$ and assume $\beta \ge \gamma$.
  Denote $\alpha := \inf A$.
  Assume either that $\alpha < 0,$ or 
  that $\alpha = 0$
  and that the lowest homogeneity in $\mcG$ is given by the constant distribution
  (of the polynomial regularity structure of Section \ref{sec:poly}).
  
  For every $f \in \msD^\gamma(M, \mcG)$ there exists a unique $\mcR f \in C^\alpha(M)$ such that
  for every compactum $\K \subset M$
  \begin{align}
    \label{eq:reconstructionCloseness-new}
    |\langle \mcR f - \Pi_x f(x), \varphi^\lambda_x \rangle|
    \lesssim
    \lambda^\gamma
    ||\Pi,\Gamma||_{\beta;\overline{\mathcal{K}}}
    ||f||_{\gamma;\overline{\mathcal{K}}}
  \end{align}
  Here $\varphi \in \balloftestfunctions^{r,\delta_\K}$, $r > |\alpha|$,
  (so that the action of $\Pi_x f(x)$ is well-defined) and
  $\overline \K := \overline{ B_{\delta_\K}( \K ) }$.
\end{lemma}
\begin{remark}
  \label{rem:uniqueness}
  Uniqueness actually holds in the class of operators $\mathcal{R}$ that satisfy
  \eqref{eq:reconstructionCloseness-new} with $\gamma$ replaced by any $\theta > 0$.
\end{remark}
\begin{proof}
  \textbf{Existence}\\
  We will apply \cite[Proposition 3.25]{bib:hairer}.%
  \footnote{Compare also \cite[Theorem 2.10]{bib:hairerIntro} for a concise presentation of
  the (wavelet) techniques involved in its proof.}
  This Proposition is formulated for $\R^d$,
  but the statement is local and also holds for $M \subset \R^d$.
  So we have to verify for $\zeta_x := \Pi_x f(x)$
  \begin{align}
    \label{eq:330_1-new}
    |\langle \varphi^n_x, \zeta_x - \zeta_y \rangle|
    &\le
    C_1 |x-y|^{\gamma-\alpha} 2^{-n d/2 - \alpha n} \\
    \label{eq:330_2-new}
    |\langle \varphi^n_x, \zeta_x \rangle|
    &\le
    C_2
    2^{-\alpha n - n d / 2},
  \end{align}
  uniformly over $x,y \in \overline \K$,
  $n \ge n_0$, $n_0 = \log_2( \delta_{\overline \K} ) \vee 0$
  and $2^{-n} \le |x-y| \le \delta_{\overline \K}$.
  In \cite[Proposition 3.25]{bib:hairer} the upper bound $1$
  is chosen on $|x-y|$, but any upper bound works,
  so we chose $\delta_{\overline \K}$, since we need $\Gamma_{x \leftarrow y}$ to be well-defined.
  
  Here
  \begin{align*}
    \varphi^n_x := 2^{nd/2} \varphi( 2^n ( \cdot - x )  ),
  \end{align*}
  and $\varphi$ is a scaling function for a wavelet basis of regularity $r > |\alpha|$.
  We have chosen $n_0$ also such that for $n \ge n_0$ and $x \in \overline \K$, $\tau \in \mcG|_x$
  the expression $\langle \Pi_x \tau, \varphi^n_x \rangle$ is well-defined.
  First, \eqref{eq:330_2-new} follows from the fact that $\alpha$ is the lowest homogeneity in $A(\mcG)$
  (note that $\varphi^n_x$ is scaled to preserve the $L^2$-norm, whereas the scaling in the definition of a model
  preserves the $L^1$-norm).

  Now
  \begin{align*}
    \langle \varphi^n_x, \zeta_x - \zeta_y \rangle
    &=
    \langle \varphi^n_x, \Pi_x f(x) - \Pi_y f(y) \rangle \\
    &=
    \langle \varphi^n_x, \Pi_x \left[ f(x) - \Gamma_{x \leftarrow y} f(y) \right] \rangle
    +
    \langle \varphi^n_x, \Pi_x \Gamma_{x \leftarrow y} f(y) - \Pi_y f(y) \rangle.
  \end{align*}
  We bound the first term as %using \eqref{eq:reconstructionConditionII} and the definition of $\msD^\gamma$ as
  \begin{align*}
    |\langle \varphi^n_x, \Pi_x \left[ f(x) - \Gamma_{x \leftarrow y} f(y) \right] \rangle|
    &\le
    \sum_a
    |\langle \varphi^n_x, \Pi_x \operatorname{proj}_{\mcG_a} \left[ f(x) - \Gamma_{x \leftarrow y} f(y) \right] \rangle| \\
    &\lesssim
    ||\Pi,\Gamma||_{\beta,\overline \K}
    \sum_a
    2^{-n a - nd/2} |x-y|^{\gamma - a} \\
    &\lesssim
    ||\Pi,\Gamma||_{\beta,\overline \K}
    2^{-n \alpha - nd/2} |x-y|^{\gamma - \alpha},
  \end{align*}
  since $2^{-n} \le |x-y|$.
  The second term is bounded as
  % TRANSPORT-PRECISION-USED
  \begin{align*}
    |\langle \varphi^n_x, \Pi_x \Gamma_{x \leftarrow y} f(y) - \Pi_y f(y) \rangle|
    &\lesssim
    ||\Pi,\Gamma||_{\beta,\overline \K}
    2^{-n\beta - nd/2} \\
    &=
    ||\Pi,\Gamma||_{\beta,\overline \K}
    2^{-n\alpha - nd/2 - n (\beta - \alpha)} \\
    &\lesssim
    ||\Pi,\Gamma||_{\beta,\overline \K}
    2^{-n\alpha - nd/2} |x-y|^{\beta - \alpha} \\
    &\lesssim
    ||\Pi,\Gamma||_{\beta,\overline \K}
    2^{-n\alpha - nd/2} |x-y|^{\gamma - \alpha}.
  \end{align*}
  This proves \eqref{eq:330_1-new}
  and an application of \cite[Proposition 3.25]{bib:hairer}
  gives the existence of $\mcR f$ satisfying the bound \eqref{eq:reconstructionCloseness-new}.

  The preceding argument is valid for $\alpha < 0$.
  For $\alpha = 0$, one can run the argument for some $\alpha' < 0$ and
  get unique existence of $\mcR f \in C^{\alpha'}$ with the claimed properties.
  In Corollary \ref{cor:reconstructionPositiveHomogeneity} below it is shown
  that actually $\mcR f \in C^0$.

  %{\color{red}(If $\alpha = 0$ we have to invoke \cite[Prop 3.28]{bib:hairer}.)}

  \textbf{Uniqueness}\\
  Uniqueness follows exactly as in \cite[Section 3]{bib:hairer}.
  %%%%% \textbf{Continuity}\\
  %%%%% We use \cite[Proposition 3.25]{bib:hairer} on
  %%%%% $\zeta_x := \Pi_x f(x) - \bar \Pi_x \bar f(x)$.
  %%%%% %
  %%%%% We have to verify
  %%%%% \eqref{eq:330_1-new} and \eqref{eq:330_2-new} with $C_1,C_2$ of order
  %%%%% \begin{align*}
  %%%%%   D := \Bigl(
  %%%%%   || \Pi, \Gamma ||_{\beta,\gamma,\overline B_{\delta \K}(\K)}\ ||f; \bar f||_{\gamma,\overline B_{\delta \K}(\K)} + ||(\Pi,\Gamma);(\bar \Pi,\bar \Gamma)||_{\beta,\gamma,\overline B_{\delta \K}(\K)}\ ||f||_{\gamma,\overline B_{\delta \K}(\K)}
  %%%%%        \Bigr).
  %%%%% \end{align*}
  %%%%% Inequality \eqref{eq:330_2-new} is immediate. Regarding \eqref{eq:330_1-new} we write
  %%%%% \begin{align*}
  %%%%%   \zeta_x - \zeta_y
  %%%%%   &=
  %%%%%   \Pi_x f(x) - \bar \Pi_x \bar f(x)
  %%%%%   -
  %%%%%   \left( \Pi_y f(y) - \bar \Pi_y \bar f(y) \right) \\
  %%%%%   &=
  %%%%%   \Pi_x \left[ f(x) - \Gamma_{x \leftarrow y} f(y) - \left( \bar f(x) - \bar \Gamma_{x \leftarrow y} \bar f(y) \right) \right]
  %%%%%   +
  %%%%%   (\Pi_x - \bar \Pi_x)(\bar f(x) - \bar \Gamma_{x \leftarrow y} \bar f(y)) \\
  %%%%%   &\quad
  %%%%%   +
  %%%%%   \Pi_x \Gamma_{x\leftarrow y} f(y)
  %%%%%   -
  %%%%%   \bar \Pi_x \bar \Gamma_{x \leftarrow y} \bar f(y)
  %%%%%   -
  %%%%%   \left(
  %%%%%   \Pi_y f(y)
  %%%%%   -
  %%%%%   \bar \Pi_y \bar f(y)
  %%%%%   \right)
  %%%%% \end{align*}
  %%%%% Now only the last line needs consideration (the rest is straightforward),
  %%%%% but this is included in the definition of $||\Pi;\bar \Pi||_{\gamma,\beta}$, so we are done.

\end{proof}

\begin{lemma}[{Reconstruction for $M$ a closed Riemannian manifold}]
  \label{thm:reconstructionOnTheManifold-new}
  Let $M$ be a closed Riemannian manifold with
  regularity structure $\mcG$ and 
  $(\Pi, \Gamma)$ a model with transport precision $\beta \in \R$.
  Let $\gamma > 0$,
  and $f \in \msD^\gamma(M,\mcG)$
  and assume $\beta \ge \gamma$.

  Denote $\alpha := \inf A$.
  Assume either that $\alpha < 0$ or 
  that $\alpha = 0$
  and that the lowest homogeneity in $\mcG$ is given by the constant distribution
  (of the polynomial regularity structure).

  Then, there exists a unique distribution
  $\mcR f \in C^{\alpha}(M)$ such that
  \begin{align}
    \label{eq:thmReconstructionCharacterizingInequality-new}
    |\langle \mcR f - \Pi_x f(x), \varphi^\lambda_x \rangle| \lesssim \lambda^\gamma ||\Pi,\Gamma||_{\beta;M}\ ||f||_{\mcD^\gamma(M,\mcG)},
  \end{align}
  for $\varphi \in \balloftestfunctions^{r,\delta_M}$, $r > |\alpha|$.

\end{lemma}
\begin{proof}
  By a cutting up procedure, it is enough to show
  \eqref{eq:thmReconstructionCharacterizingInequality-new} for $\varphi \in
  \balloftestfunctions^{r,\delta'}$, with $\delta' \in (0,\delta_M]$ to be
  chosen.

  Let $(\Psi_i, \U_i)_{i\in I}$ an atlas with subordinate partition of unity $(\phi_i)_{i\in I}$, with $I$ finite.
  On each chart, we push-forward the regularity structure, model and $f$ to $\Psi_i(\U_i),$
  with corresponding reconstruction operation $\tilde\mcR_i$, model $\tilde\Pi_i$ and modelled distribution $\tilde f_i$.
  %Denote $\K_i := \supp \phi_i$.
  For each $i\in I,$ fix a compactum $\K_i \subset \U_i$ such that $\supp \phi_i$ is strictly contained in $\K_i$.
  By Lemma \ref{lem:pushForward-new},
  \begin{align*}
    || \tilde f_i ||_{\msD^\gamma(\Psi_i(\K_i), \mcG)} &\lesssim ||f||_{\msD^\gamma( M, \mcG )} \\
    || \tilde f_i - f'_i||_{\msD^\gamma(\Psi_i(\K_i), \mcG)} &\lesssim ||f - f'||_{\msD^\gamma( M, \mcG )} \\
    ||\tilde \Pi_i, \Gamma \Pi_i||_{\beta,\Psi_i(\K_i)} &\lesssim ||\Pi,\Gamma||_{\beta,M}.
    %\\
    %||(\tilde \Pi_i, \tilde \Gamma_ i); (\tilde \Pi'_i, \tilde \Gamma'_i)||_{\beta,\gamma,\Psi_i(\K_i)} &\lesssim ||(\Pi,\Gamma);(\Pi',\Gamma')||_{\beta,\gamma,M}.
  \end{align*}
  Now reconstruct in each coordinate chart as $\tilde T_i := \tilde \mcR_i \tilde f_i$ using Theorem \ref{thm:reconstructionOnOpenSet-new}.
  Define $\mcR f := \sum_{i\in I} \phi_i (\Psi^{-1}_i)_* \tilde T_i$. Then
  \begin{align*}
    \langle \mcR f - \Pi_x f(x), \varphi^\lambda_x \rangle
    &=
    \sum_i
    \langle \phi_i \left( (\Psi^{-1}_i)_* T_i - \Pi_x f(x) \right), \varphi^\lambda_x \rangle \\
    &=
    \sum_i
    \langle (\Psi^{-1}_i)_* T_i - \Pi_x f(x), \phi_i \varphi^\lambda_x \rangle \\
    &=
    \sum_i
    \langle T_i - (\Psi_i)_* \left( \Pi_x f(x) \right), \left( \phi_i \varphi^\lambda_x \right) \circ \Psi_i^{-1} \rangle.
    %&\lesssim
    %\sum_i \lambda^\gamma \\
    %&\lesssim
    %\lambda^\gamma.
  \end{align*}
  If $x \not \in \U_i,$ we want the summand to vanish.
  So let $\delta' := \min_i d(\supp \phi_i, \partial \U_i)$.
  Then for $\varphi \in \balloftestfunctions^{r, \delta'}$, $\lambda \in (0,1],$
  we have $\phi_i \varphi^\lambda_x \not= 0$ implies $x \in \U_i$.
  Hence, if $x \not\in \U_i,$ we have $\phi_i \varphi^\lambda_x = 0$, so the summand vanishes.
  
  Otherwise, with $z := \Psi_i(x)$
  \begin{align*}
    |\langle T_i - (\Psi_i)_* \left( \Pi_x f(x) \right), \left( \phi_i \varphi^\lambda_x \right) \circ \Psi_i^{-1} \rangle|
    &=
    |\langle T_i - (\Psi_i)_* \left( \Pi_x f(x) \right), \left( \phi_i \lambda^{-d} \varphi( \lambda^{-1} \exp_x^{-1}( \cdot ) ) \right) \circ \Psi_i^{-1} \rangle| \\
    &=
    |\langle T_i - \bar \Pi_z f(z), \left( \phi_i \lambda^{-d} \varphi( \lambda^{-1} \exp_x^{-1}( \cdot ) ) \right) \circ \Psi_i^{-1} \rangle| \\
    &\lesssim
    ||\tilde \Pi_i, \tilde \Gamma_i||_{\beta,\Psi_i(\U_i)}
    ||\tilde f||_{\mcD^{\gamma}( \Psi_i(\U_i), \tilde \mcG_i )}
    \lambda^\gamma \\
    &\lesssim
    ||\Pi_i, \Gamma_i||_{\beta,M}
    ||f||_{\mcD^{\gamma}( M, \mcG )}
    \lambda^\gamma,
  \end{align*}
  since $\left( \phi_i \lambda^{-d} \varphi( \lambda^{-1} \exp_x^{-1}( \cdot ) ) \right) \circ \Psi_i^{-1}$ falls under Remark \ref{rem:221} around $z$.
  Summing over $i$ gives \eqref{eq:thmReconstructionCharacterizingInequality-new}.

  %%%%% \textbf{Continuity}
  %%%%% \begin{align*}
  %%%%%   \langle \mathcal{R} f  - \bar{\mathcal{R}} \bar f - \Pi_x f(x) + \bar \Pi_x \bar f(x), \varphi_x^\lambda \rangle
  %%%%%   &=
  %%%%%   \sum_i \langle \phi_i \left( (\Psi_i^{-1})_\ast T_i - \Psi_i^{-1})_\ast \bar T_i - (\Pi_x f(x) + \bar \Pi_x \bar f(x)) \right), \varphi_x^\lambda \rangle \\
  %%%%%   &=
  %%%%%   \sum_i \langle T_i - \bar T_i - ( (\Psi_i)_\ast \Pi_x f(x) + (\Psi_i)_\ast\bar \Pi_x \bar f(x)), (\phi_i \varphi_x^\lambda) \circ \Psi^{-1}_i \rangle.
  %%%%% \end{align*}

  %%%%% With $z := \Psi_i(x)$ we get
  %%%%% \begin{align*}
  %%%%%   &\langle T_i - \bar T_i - ( (\Psi_i)_\ast \Pi_x f(x) + (\Psi_i)_\ast\bar \Pi_x \bar f(x)), (\phi_i \varphi_x^\lambda) \circ \Psi^{-1}_i \rangle \\
  %%%%%   &=
  %%%%%   \langle T_i - \bar T_i - ( \tilde \Pi_x \tilde f(z) + \tilde{\bar \Pi}_x \tilde{\bar f}(z)), (\phi_i \varphi_x^\lambda) \circ \Psi^{-1}_i \rangle,
  %%%%% \end{align*}
  %%%%% and the estimate then follows as in the existence step.
\end{proof}

\begin{corollary}
  \label{cor:reconstructionPositiveHomogeneity}
  In setting of the previous theorem,
  assume that the lowest homogeneity in $\mcG$ is $0$
  and that it is given by the constant (as in the polynomial regularity structure of Section \ref{sec:poly}).
  Then $\mcR f$ is given by projection onto that homogeneity, i.e.
  \begin{align*}
    (\mcR f)(p) = f_0(p).
  \end{align*}
\end{corollary}
\begin{proof}
  Define $\tilde \mcR f(p) := f_0(p)$, then
  \begin{align*}
    |\langle \tilde \mcR f(\cdot) - \Pi_p f(p)(\cdot), \varphi^\lambda_p \rangle|
    &=
    |\langle \tilde f_0(\cdot) - f_0(p), \varphi^\lambda_p \rangle|
    +
    |\langle \sum_{\ell > 0} \Pi_p \proj_{\mcG_\ell|_p} f(p), \varphi^\lambda_p \rangle|.
  \end{align*}
  Recall that the projection $\proj$ is defined in Definition \ref{def:regularityStructure}.
  %where $\eta$ is the smallest homogeneity strictly larger than $0$.
  %
  The last term is of bounded by a constant times $\lambda^\eta$, where $\eta$ is 
  the smallest homogeneity strictly larger than $0$.

  For the second to last term we first write
  \begin{align*}
    f_0(\cdot) - f_0(p)
    &=
    \left(
      f(\cdot)
      - \Gamma_{\cdot \leftarrow p} f(p)
      + \Gamma_{\cdot \leftarrow p} \proj_{\mcG_{>0|_p}} f(p)
    \right)_0.
  \end{align*}

  Now, since $f \in \mcD^\gamma$,
  \begin{align*}
    \left(
      f(\cdot)
      - \Gamma_{\cdot \leftarrow p} f(p) \right)_0
      \lesssim
      d(\cdot,p)^\gamma.
  \end{align*}
  
  By the properties of a model
  \begin{align*}
    | \Gamma_{\cdot \leftarrow p} \proj_{\mcG_{>0|_p}} f(p) |_0
    \lesssim d(\cdot, p)^{\eta}.
  \end{align*}

  Hence $|f_0(\cdot) - f_0(p)| \lesssim d(\cdot,p)^{\eta\wedge\gamma}$ and then
  \begin{align*}
    |\langle \tilde f_0(\cdot) - f_0(p), \varphi^\lambda_p \rangle| \lesssim \lambda^{\eta \wedge \gamma}
  \end{align*}

  Hence, by Remark \ref{rem:uniqueness}, $\tilde \mcR = \mcR$.
\end{proof}

We want to apply the Lemma \ref{thm:reconstructionOnTheManifold-new} to the terms in the heat kernel asymptotics (Theorem \ref{thm:heatKernelEstimates}).
The problem is that their support will be of order $1$ (and not of order $\lambda$ as for $\varphi^\lambda_x$).
Hence we need the following refinement which is similar to Lemma \ref{lem:rem221ForS}.
\begin{lemma}
  \label{lem:rem221ForSForReconstruction}
  In the setting of Lemma \ref{thm:reconstructionOnTheManifold-new},
  let $\varphi$ satisfy the assumptions of Lemma \ref{lem:rem221ForS}
  with the additional condition
  $\supp \varphi \subset B_{\delta_\K / 4}(0) \subset \R^d$.
  %Assume moreover that $B_{\delta_\K / 2}(p) \subset \K$ (which can always be achieved by making $\delta_\K$ smaller.)
  %
  Then %for $\tau \in \mcG_\ell|_p$
  \begin{align*}
    %|\langle \Pi_p \tau, \varphi \circ \exp^{-1}_p \rangle|
    %\lesssim
    %C_\varphi
    %||\Pi,\Gamma||_{\beta,\gamma,\K} \lambda^\ell,
    |\langle \mcR f - \Pi_p f(p), \varphi \circ \exp^{-1}_p ) \rangle| \lesssim C_\varphi \lambda^\gamma ||\Pi||_{\beta;M}\ ||f||_\gamma.
  \end{align*}
  where $C_\varphi := C(\varphi,\lambda,0,N,r)$ is defined in Lemma \ref{lem:rem221ForS}.
\end{lemma}
\begin{proof}
  Let 
  $\varphi_{z,\tilde \lambda}$
  be given as in the proof of Lemma \ref{lem:rem221ForTheModel} with $\K := M$.
  Recall $\lambda_M = \lambda \delta_M / 4$ that $\supp \varphi_{z,\lambda_M} \subset B_{\lambda_M}(\lambda_M z) \cap B_{\delta_M / 4}(0)$.
  Hence $\varphi_{z,\lambda_M} \equiv 0$ for $|\lambda_M z| \ge \delta_M / 2$.
  Then with $\zeta_r := \mcR f - \Pi_r f(r)$
  \begin{align*}
    \langle \zeta_p, \varphi \circ \exp^{-1}_p \rangle
    &=
    \sum_{z \in \Z^d}
    \langle \zeta_p, \varphi_{z,\tilde \lambda} \circ \exp^{-1}_p \rangle \\
    &=
    \sum_{z \in \Z^d, |\lambda_M z| < \delta_M / 2}
    \langle \zeta_p, \varphi_{z,\tilde \lambda} \circ \exp^{-1}_p \rangle \\
    &=
    \sum_{z \in \Z^d, |\lambda_M z| < \delta_M / 2}
    \left[
      \langle \zeta_{\exp_p(\lambda_M z)}, \varphi_{z,\tilde \lambda} \circ \exp^{-1}_p \rangle
    +
    \langle \zeta_p - \zeta_{\exp_p(\lambda_M z)}, \varphi_{z,\tilde \lambda} \circ \exp^{-1}_p \rangle
   \right].
  \end{align*}
  Note that in the sum $|\lambda_M z| < \delta_M / 2$.
  Hence $\exp_p(\lambda_M z ) \in M$ is well-defined.
  %
  %Hence by definition of a model, $|\Gamma_{p \leftarrow q} \tau|_m \le ||\Pi,\Gamma||_{\beta,\gamma,M} d(p,q)^{(\ell-m)\vee 0}$ for $\tau \in \mcG_\ell|_q$.
  %
  %Then for those $z$.
  Now the first summand can be written as
  \begin{align*}
    \langle \zeta_{\exp_p(\lambda_M z)}, \varphi_{z,\lambda_M} \circ \exp^{-1}_p \rangle
    =
    \langle \zeta_{\exp_p(\lambda_M z)}, \varphi_{z,\lambda_M} \circ \exp^{-1}_p \circ \exp_{\exp_p(\lambda_M z)} \circ \exp^{-1}_{\exp_p(\lambda_M z)}\rangle.
  \end{align*}
  Applying Remark \ref{rem:221}
  to $\varphi_{z,\lambda_M} \circ \exp^{-1}_p \circ \exp_{\exp_p(\lambda_M z)}$
  and \eqref{eq:thmReconstructionCharacterizingInequality-new},
  this
  is bounded by a constant times $C_\varphi ||\Pi,\Gamma||_{\beta,M}\ ||f||_{\mcD^\gamma(M,\mcG)} \lambda^\gamma \frac{1}{1 + |z|^N}$.

  The second summand is bounded as
  \begin{align*}
    \Big|\Big\langle \zeta_p - \zeta_{\exp_p(z)}, \varphi_{z,\lambda_M} \circ \exp_p \Big\rangle\Big|
    &=
    \Big| \Big\langle \Pi_{\exp_p(z)} f(\exp_p(z)) - \Pi_p f(x), \varphi_{z,\lambda_M} \circ \exp_p \Big\rangle\Big|  \\
    &\le
    \Big| \Big\langle \Pi_{\exp_p(z)} \Bigl( f(\exp_p(z)) - \Gamma_{\exp_p(z) \leftarrow x} f(x) \Bigr), \varphi_{z,\lambda_M} \circ \exp_p \Big\rangle\Big|  \\
    &\qquad
    +
    \Big| \Big\langle \Pi_{\exp_p(z)} \Gamma_{\exp_p(z) \leftarrow x} f(x) - \Pi_p f(x), \varphi_{z,\lambda_M} \circ \exp_p \Big\rangle\Big|  \\
    &\lesssim
    C_\varphi
    ||\Pi,\Gamma||_{\beta,M}\ ||f||_{\mcD^\gamma(M,\mcG)} 
    \left(
    \sum_\ell
      \lambda^\ell |\lambda_M z|^{\gamma-\ell}
      +
      \lambda^\beta
    \right)
    \frac{1}{1+|z|^N}.
  \end{align*}
  Hence
  \begin{align*}
    |\langle \zeta_p, \varphi \circ \exp^{-1}_p \rangle|
    &\lesssim
    C_\varphi
    ||\Pi,\Gamma||_{\beta,M}\ ||f||_{\mcD^\gamma(M,\mcG)} 
    \Bigl(
    \sum_{z \in \Z^d} \lambda^\gamma \frac{1}{1 + |z|^N}
    +
    \sum_{z \in \Z^d}
    \left( \sum_a \lambda^a |\lambda_M z|^{\gamma-a} + \lambda^\beta \right) \frac{1}{1 + |z|^N} \Bigr) \\
    &\lesssim
    C_\varphi
    ||\Pi,\Gamma||_{\beta,M}\ ||f||_{\mcD^\gamma(M,\mcG)} 
    \Bigl(
    \lambda^\gamma
    +
    \lambda^\beta
    +
    \lambda^\gamma \sum_{z\in \Z^d} \frac{1}{1 + |z|^N} |z|^{\gamma+|\inf_{a \in A} a|} \Bigr) \\
    &\lesssim
    C_\varphi
    ||\Pi,\Gamma||_{\beta,M}\ ||f||_{\mcD^\gamma(M,\mcG)} 
    \lambda^\gamma,
  \end{align*}
  for $N > d + \gamma + |\inf_{a \in A} a|$.
\end{proof}

%%%%% \begin{corollary}
%%%%%   If the model is smooth that is if $\Pi_x \tau \in C^\infty(M)$ for all $x \in M, \tau \in \mcG|_x$, then
%%%%%   \begin{align*}
%%%%%     \mcR{f}(x) = (\Pi_x f(x)(x). 
%%%%%   \end{align*}
%%%%% \end{corollary}
%%%%% \begin{proof}
%%%%%   If the model is smooth,
%%%%%   we have $(\Pi_x \tau)(x) = 0$ for all $|\tau| > 0$.
%%%%%   Moreover, since $\beta \ge \gamma > 0$
%%%%%   \begin{align*}
%%%%%     (\Pi_x \tau)(x) - (\Pi_y \Gamma_{y \leftarrow x} \tau)(x) = 0.
%%%%%   \end{align*}
%%%%%   Then
%%%%%   \begin{align*}
%%%%%     (\Pi_z f(z))(z) - 
%%%%%     (\Pi_x f(x))(x)
%%%%%     &=
%%%%%     \sum_{a \in A} \left( \Pi_z \operatorname{proj}_{T_a} \left( f(z) - \Gamma_{z \leftarrow x} f(x) \right) \right)(z) \\
%%%%%     &=
%%%%%     \sum_{a \in A, a \le 0} \left( \Pi_z \operatorname{proj}_{T_a} \left( f(z) - \Gamma_{z \leftarrow x} f(x) \right) \right)(z) \\
%%%%%     &\lesssim
%%%%%     \sum_{a \in A, a \le 0} d(z,x)^{\gamma-a} \\
%%%%%     &\lesssim
%%%%%     d(z,x)^\gamma.
%%%%%   \end{align*}
%%%%%   By uniqueness of the reconstruction operator, the result follows.
%%%%% \end{proof}

\section{Linear ``polynomials'' on a Riemannian manifold}
\label{sec:poly}

The regularity structure for linear ``polynomials'' on the Riemannian manifold $M$
will be built on the vector bundle $(M \times \R) \oplus T^* M$.
For readability
introduce the symbol $\symbol{\1}$ and decree that it forms a basis for $\R$.
Define the graded vector bundle
\begin{align*}
  \mcT
  &:=
  (M \times \R \symbol{1})
  \oplus
  T^* M,
\end{align*}
with grading $A(\mcT) = \{0, 1\}$.
For $q \in M$ let
  $\mcT|_q = \operatorname{span}\{ \symbol{1} \}
  \oplus
  T^*_q M$ be the fiber at $q$.
A generic element of $\mcT_q$ will be written as
\begin{align*}
  \symbol{1} a + \cotang{\omega},
\end{align*}
with $a \in \R, \cotang{\omega} \in T^*_q M$.
Let $\U_q := B_\delta(q)$, where $\delta$ is the radius of injectivity of $M$.
Define the linear map $\Pi_q: \mcT_q \to D'(\U_q)$ as
\begin{align*}
  (\Pi_q \symbol{1} )(z) &= 1 \\
  (\Pi_q \cotang{\omega} )(\cdot) &= \cotang{\omega} \exp^{-1}_q( \cdot ), \qquad \cotang{\omega} \in T_q^* M.
\end{align*}
%the latter is well-defined for $d(q,\cdot) < \delta$; $\delta$ the radius of injectivity of $M$.

Note that since $\R \symbol{1}$ is a trivial fiber bundle, it is enough to specify it on
the basis element $\symbol{1}$. This is not possible on $T^* M$.
Note also that $\Pi_q \cotang{\omega}$ is chosen to have value $0$ and differential $\cotang{\omega}$ at $q$.

Finally define the re-expansion maps $\Gamma_{p \leftarrow q}: \mcT_q \to \mcT_p$ as
\begin{align*}
  \Gamma_{p\leftarrow q} \symbol{1} &= \symbol{1} \\
  \Gamma_{p\leftarrow q} \cotang{\omega}
  &=
  \cotang{\omega} \exp^{-1}_q( p ) \symbol{1}
  +
  d[ \cotang{\omega} \exp^{-1}_q ](p),
\end{align*}
which is well-defined for $d(p,q) < \delta$;  $\delta$ the radius of injectivity of $M$.
$\Pi$ and $\Gamma$ together form the polynomial model, where we take $\delta_M = \delta$ in Definition \ref{def:modelGeneral}.

The transport of $\cotang{\omega} \in T_y^*M$ is chosen
such that $\Pi_q \cotang{\omega}$
and $\Pi_p \Gamma_{p\leftarrow q} \cotang{\omega}$
have, at $p$, the same value and the same first derivative.
Our re-expansion is not exact, i.e. we \underline{do not} have
$\Pi_q \tau = \Pi_p \Gamma_{p \leftarrow q} \tau$, but we have the following.
\begin{lemma}
  \label{lem:polynomialError2}
  For $\cotang{\omega} \in T^*_y M$,
  uniformly for $d(p,y)$ bounded,
  \begin{align*}
    |(\Pi_q \cotang{\omega})(z) - (\Pi_p \Gamma_{p\leftarrow q} \cotang{\omega})(z)|
      &\lesssim
      d(z,p)^2 \\
      |D(\Pi_q \cotang{\omega})(z) - D(\Pi_p \Gamma_{p\leftarrow q} \cotang{\omega})(z)|
      &\lesssim
      d(z,p) \\
      |D^2(\Pi_q \cotang{\omega})(z) - D^2(\Pi_p \Gamma_{p\leftarrow q} \cotang{\omega})(z)|
      &\lesssim
      1.
  \end{align*}
\end{lemma}
\begin{proof}
  Let
  \begin{align*}
    f( z ) &:= (\Pi_q \cotang{\omega})(z) \\
    g( z ) &:= (\Pi_x \Gamma_{p \leftarrow q} \cotang{\omega})(z).
  \end{align*}
  
  By construction $f(p) = g(p), df(p) = dg(p)$ and hence the statement follows from Taylor's theorem.
\end{proof}

%\textcolor{red}{\begin{remark}
% If $(\psi,U)$ is a gauge on the manifold $M,$ as defined in \cite[Section 2.3]{bib:driverSemko},
%  one could consider the transport given for any $x,y\in M$ by $\Gamma_{x\leftarrow y} \1= \1$ and for $\omega_y\in T_y^*M,$ 
%$$\Gamma_{x\leftarrow y} (\omega_y)=  \omega_y\circ \psi(y,x)+ \omega_y\circ U_{y,x}$$
%  and the model $\Pi_x \1 =1,$  $$\Pi_x (\omega_x)= \omega_x\circ \psi (x,\cdot),\quad \omega_x\in T_x^*M.$$
%  A transport map is dual to a gauge on $M$.
%\end{remark}
%} 
\begin{remark} 
  In the setting of the previous Lemma, not only $f(p) = g(p)$ but also $f(q) = g(q)$.
  Indeed,
  for two points $p,q\in M,$ at distance smaller than the cut locus and $\omega_q\in T^*_yM$,
  \begin{align*}
    \Gamma_{p\leftarrow q}(\omega_q)= \omega_q(\exp^{-1}_q(p))\1+d|_p\left[\omega_q(\exp^{-1}_q(\cdot))\right]
    =\omega_q(\exp^{-1}_q(p))\1+\omega_q\circ d|_p(\exp^{-1}_q),
  \end{align*}
  where the tangent map  satisfies indeed $d|_p (\exp_q^{-1}):T_pM\to T_q M.$  By definition,   
  \begin{align*}
    \Pi_p(\Gamma_{p\leftarrow q} (\omega_q))=  \omega_q(\exp_q^{-1}(p))+ \omega_q\circ d|_p(\exp_q^{-1})\circ \exp^{-1}_p
  \end{align*}
  does  a priori disagree with $\Pi_q(\omega_q)=\omega_q\circ\exp_q^{-1}, $ but at $p$.  Let us set $v_q= \exp_q^{-1}(p)$ and $v_p=\exp_p^{-1}(q).$   The path $\gamma=(\exp_q((1-t) v_q))_{0\le t\le 1}$   is  the unique path from $p$ to $q,$ with length and speed $d(p,q),$ staying within the cut-locus from $y$, that is $(\exp_p( t  v_p  ))_{0\le t\le 1}:$ in other words, for any $0\le t\le 1,$ 
  \begin{align*}
    \exp_p (t \exp^{-1}_p( q)   )=   \exp_q((1-t) \exp^{-1}_q (p) ).
  \end{align*}
  Hence,
  \begin{align*}
  d|_p(\exp_q^{-1})(v_p)&= \left.\frac{d}{dt}\right|_{t=0} \exp_q^{-1} ( \exp_p(t v_p ))= \left.\frac{d}{dt}\right|_{t=0} \exp_q^{-1} ( \exp_q((1-t) v_q ))\\
  &=-v_q
  \end{align*}
  and 
  \begin{align*}
    \Pi_p(\Gamma_{p\leftarrow q}(\omega_q))(q)=\omega_q(v_q)+ \omega_q\circ d|_p(\exp_q^{-1})(v_p)=0=\Pi_q(\omega_q)(q).
  \end{align*}
\end{remark}

The next lemma follows from Lemma \ref{lem:polynomialError2}
and is shown in more generality in Theorem \ref{thm:polynomialRS}.
\begin{lemma}
  \label{lem:linearPolynomialsAreAModel}
  The above is a model of transport precision $\beta=2$.
\end{lemma}

As a sanity check for our construction, we mention the following lemma, which is almost immediate in the flat case (see \cite[Lemma 2.12]{bib:hairer}).
We will prove it in Section \ref{sec:higherOrderPolynomials} in a more general setting.
\begin{lemma}
  \label{lem:cgammaVcDgamma}
  For $\gamma \in (1,2)$,
  a function $f: M \to \R$ is in $C^\gamma(M)$
  if and only if there exists a function
  $\hat f(x) = f_0(x) \symbol{1} + \cotang{f_1(x)} \in \msD^\gamma(M,\mcT)$
  with $f_0(x) = f(x)$ and $f_1(x) \in T^* M$.

  In that case: $f_1(x) = df(x)$.
\end{lemma}

\section{The regularity structure for PAM on a manifold}
\label{sec:pam}

In the next four sections $M$ is a $2$-dimensional closed manifold.

The regularity structure for PAM will be built on two copies of the vector
bundle, $\left(  M\times\mathbb{R}^{2}\right)  \oplus T^{\ast}M.$ We denote
these two copies  by $\mcV$ and $\mcW$. In order to distinguish the different elements of these
bundles we introduce the symbols $\left\{  \symbol{\1},\symbol{\Xi}
,\symbol{\I[\Xi]},\symbol{\I[\Xi]\Xi}\right\}  $ and decree that they form a basis for
$\mathbb{R}^{4}.$ We then write%
\begin{align*}
  \mcW &  =\left(  M\times\left[  \mathbb{R} \symbol{\1}\oplus\mathbb{R}%
\symbol{\I[\Xi]}\right]  \right)  \oplus T^{\ast}M\text{ and}\\
\mcV &  =\left(  M\times\left[  \mathbb{R}\symbol{\Xi}\oplus\mathbb{R}%
\symbol{\I[\Xi]\Xi}\right]  \right)  \oplus(\symbol{\Xi}T^{\ast}M),
\end{align*}
where $\symbol{\Xi}T^{\ast}M$ is simply another copy of $T^{\ast}M.$ Formally
we have, $\mcV=\mcW\symbol{\Xi}$. As usual we will let $\mcT|_{p},\mcV|_{p},$
and $\mcW|_{p}$ denote the fibers of these bundles over $p\in M.$

The vector bundles $\mcV$ and $\mcW$ are graded, with
gradings
\begin{align*}
  A(\mcV) &  :=\{\alpha,2\alpha+2,\alpha+1\}\\
  A(\mcW) &  :=\{0,1,\alpha+2\},
\end{align*}
for some $\alpha\in(-3/2,-1)$ corresponding to the regularity of the driving
white noise $\xi$.

For $\beta\in A(\mcV)$ (or $\beta\in A(\mcW))$ 
recall (Definition \ref{def:regularityStructure}) that 
$\proj_{\beta}:\mcV\rightarrow\mcV$ ($\proj_{\beta}:\mcW\rightarrow\mcW)$ is
the projection taking an element to its $\beta$ -- component. To be concrete,
generic elements $\tau\in\mcV|_{p}, \tau' \in \mcW|_p$ are of the form
\begin{align*}
  \tau&=\symbol{\Xi}a+\symbol{\I[\Xi]\Xi}b+\symbol{\Xi}\cotang{c} \\
  \tau'&=\symbol{1} d+\symbol{\I[\Xi]}e+\cotang{f},
\end{align*}
with $a,b,d,e\in\R,\cotang{c},\cotang{f}\in T_{p}^{\ast}M$. And then for
example
\begin{align*}
  \proj_{\alpha}\tau &  =\symbol{\Xi}a\in\mcV_{\alpha}|_{p}\\
  \proj_{\alpha+2}\tau' &  =\symbol{\I[\Xi]}e\in\mcW_{\alpha+2}|_{p}.
\end{align*}
All the graded fibers have a canonical norm, where on the cotangent space we use
the norm induced by the Riemannian metric. For $\beta\in A$, $\tau\in
\mcV|_{p}$ (or $\tau \in \mcW|_p$) we write, in a slight abuse of notation, $|\tau|_{\beta
}:=|\proj_{\beta}\tau|$.

%%{\color{red}
%%Recall that for a function $\varphi: \R^d \to \R$, with $\supp \varphi \subset B_{\deltaf}(0)$ we define for $\lambda \in (0,1]$,
%%$p \in M$ the function
%%\begin{align*}
%%  \varphi^\lambda_p( z ) := \lambda^{-d} \varphi( \lambda^{-1} \exp^{-1}_p( z ) ),
%%\end{align*}
%%which, is of course, extended to $0$ outside the domain of definition.
%%
%%Denote with $\balloftestfunctions$ the set of functions $\varphi \in C^r(\R^d)$
%%with support in $B_{\deltaf}(0)$ and $||\varphi||_{C^r(\R^d)} \le 1$.
%%Here $r$ will be depend on the situation, and will always be large enough so that the distributions
%%under consideration can act on $C^r$.
%%
%%As we have seen for ``polynomials'' in Section \ref{sec:poly},
%%one should think of a model as taking an abstract Taylor expansion at a point $p$ (that is an element $\tau \in \mcG|_p$)
%%and realize it as a distribution on $M$ that ``has this Taylor expansion at $p$''.
%%%
%%The maps $\Gamma$ are necessary for comparing abtract Taylor expansions at different points.}

The model we shall use for the parabolic Anderson model will be time
dependent, so we need slight extensions of our definitions.

\begin{definition}
  For $\mcG = \mcV, \mcW$,
  assume we are given a family of models $(\Pi^t, \Gamma^t)$ on $M$ parametrized by $t \in [0,T]$.
  Define
  \begin{align*}
    ||\Pi, \Gamma||_{\beta,M,T}
    :=
    \sup_{t \le T}
      ||\Pi^t, \Gamma^t||_{\beta,M},
  \end{align*}
  where $||\Pi^t,\Gamma^t||_{\beta,M}$ is defined in Definition \ref{def:modelGeneral}.
  Note that for fixed $t$, the model comes with a reconstruction operator (Theorem \ref{thm:reconstructionOnTheManifold-new}),
  which we shall denote $\mathcal{R}_t$.
\end{definition}

\newcommand{\normscaling}{\mathfrak{N}} % the factor in the modified norm
\begin{definition}
  [Time-dependent modelled distributions]
  \label{def:CD}
  For $\mcG = \mcV, \mcW$,
  given a family of models $(\Pi^t, \Gamma^t)$ parametrized by $t \in [0,T]$,
  denote by $\mathcal{D}^{t,\gamma}(\mcG) = \mathcal{D}^{t,\gamma}(M,\mcG)$ the corresponding
  spaces of modelled distributions.
  That is, as defined in Definition \ref{def:modelledDistributions},
  \begin{align*}
    ||g||_{\mcD^{t,\gamma}(M,\mcG)}
    :=
    \sup_{p \in M} \sup_{\ell < \gamma} ||g(p)||_\ell
    +
    \sup_{p,q \in M} \sup_{\ell < \gamma} 
    \frac{ ||g(q) - \Gamma^t_{p \leftarrow q} g(q)||_\ell }{d(p,q)^{\gamma-\ell}} < \infty,
  \end{align*}
  %%%%% Define also
  %%%%% \begin{align*}
  %%%%%   ||g; \bar g||_{\mcD^{t,\gamma}(\mcG)}
  %%%%%   &:=
  %%%%%   \sup_x \sup_{\ell < \gamma} ||g(x) - \bar g(x)|_|\ell
  %%%%%   +
  %%%%%   %\sup_{d(x,y) < \deltaf}
  %%%%%   \sup_{x,y}
  %%%%%   \sup_{\ell < \gamma}
  %%%%%   \frac{||g(x) - \bar g(x) - \Gamma^t_{x \leftarrow y} g(y) + \bar \Gamma^t_{x \leftarrow y} \bar g(y)||_\ell}{d(x,y)^{\gamma - \ell}}.
  %%%%% \end{align*}

  %We shall need the modified norm, 
  For $\normscaling > 0$, define the modified norm
  \begin{align*}
    ||g||_{\mcD^{t,\gamma,\normscaling}(M,\mcG)}
    :=
    \sup_{p \in M} \sup_{\ell < \gamma} ||g(p)||_\ell
    +
    \sup_{p,q \in M} \sup_{\ell < \gamma, \ell \not= \mu} 
    \frac{||g(q) - \Gamma^t_{p \leftarrow q} g(q)||_\ell}{d(p,q)^{\gamma-\ell}}
    +
    \normscaling
    \sup_{p,q \in M} \frac{||g(q) - \Gamma^t_{p \leftarrow q} g(q)||_\mu}{d(p,q)^{\gamma-\mu}}.
  \end{align*}
  Here $\mu = \alpha,$ if $\mcG = \mcV$
  and $\mu = 0,$ if $\mcG = \mcW$.
  
  %%%%% And finally
  %%%%% \begin{align*}
  %%%%%   ||g;\bar g||_{\mcD^{t,\gamma,\normscaling}(\mcG)}
  %%%%%   &:=
  %%%%%   \sup_x \sup_{\ell < \gamma} ||g(x) - \bar g(x)||_\ell
  %%%%%   +
  %%%%%   %\sup_{d(x,y) < \deltaf}
  %%%%%   \sup_{x,y}
  %%%%%   \sup_{\ell < \gamma, \ell \not= \eta}
  %%%%%   \frac{||g(x) - \bar g(x) - \Gamma^t_{x \leftarrow y} g(y) + \bar \Gamma^t_{x \leftarrow y} \bar g(y)||_\ell}{d(x,y)^{\gamma - \ell}} \\
  %%%%%   &\quad
  %%%%%   +
  %%%%%   \normscaling
  %%%%%   \sup_{x,y} \frac{||g(x) - \Gamma^t_{x \leftarrow y} g(y)||_\eta}{d(x,y)^{\gamma-\eta}}.
  %%%%% \end{align*}

  Define $\mathcal{D}^{\gamma,\gamma_0}_T(M,\mcG)$ to be the space of functions $f: [0,T] \to C(M, \mcG)$
  with $f(t) \in \mathcal{D}^{t,\gamma}(M,\mcG)$ and
  \begin{align*}
    ||f||_{\mcD^{\gamma,\gamma_0}_T(M,\mcG)}
    :=
    \sup_{t \le T} ||f(t)||_{\mcD^{t,\gamma}(M,\mcG)}
    + 
    \sup_{p \in M, s,t \le T} \frac{||f(t,p) - f(s,p)||_\upsilon}{|t-s|^{\gamma_0}} < \infty,
  \end{align*}
  where $\upsilon = \alpha,$ if $\mcG = \mcV$
  and $\upsilon = 0,$ if $\mcG = \mcW$.
  %%%%% Define also
  %%%%% \begin{align*}
  %%%%%   ||f; \bar f||_{\mcD^{\gamma,\gamma_0}_T(\mcG)}
  %%%%%   &:=
  %%%%%   \sup_{t\le T} 
  %%%%%   ||f(t); \bar f(t)||_{\mcD^{t,\gamma}(\mcG)}
  %%%%%   +
  %%%%%   \sup_{s,t,x} \frac{||f(t,x) - f(s,x)||_\upsilon}{|t-s|^{\gamma_0}}.
  %%%%% \end{align*}
  %
  For $\normscaling > 0$, define the modified norm
  \begin{align*}
    ||f||_{\mcD^{\gamma,\gamma_0,\normscaling}_T(M,\mcG)}
    &:=
    \sup_{t\le T} ||f(t)||_{\mcD^{t,\gamma,\normscaling}(M,\mcG)}
    + 
    \sup_{p \in M, s,t \le T} \frac{||f(t,p) - f(s,p)||_\upsilon}{|t-s|^{\gamma_0}}
    .
    %%%%% \\
    %%%%% ||f; \bar f||_{\mcD^{\gamma,\normscaling}_T(\mcG)}
    %%%%% &:=
    %%%%% \sup_{t\le T} ||f(t);\bar f(t)||_{\mcD^{t,\gamma,\gamma_0,\normscaling}(\mcG)}
    %%%%% + 
    %%%%% \sup_{s,t,x} \frac{||f(t,x) - f(s,x)||_\upsilon}{|t-s|^{\gamma_0}}.
  \end{align*}
\end{definition}

\begin{remark}
  The modified norms with scaling parameter $\normscaling$ are necessary for the fixpoint argument,
  see Remark \ref{rem:modified}.
  %Another approach would be to apply the fixpoint map twice.

  As usual with H\"older-type spaces on compact domains, these spaces are complete Banach spaces.
\end{remark}

%\begin{remark}
%  Note that there is no blowup allowed in the norms at $t=0$.
%  This means we have to start with quite regular initial condition
%  and that we cannot stitch short-time solutions together.
%\end{remark}

We now build the model for the structures $\mcV, \mcW$.
As input we need realization of $\symbol{\Xi}$ and $\symbol{\I[\Xi] \Xi}$.

\begin{definition}
  \label{def:xiz}
  Assume for $T>0$ we are given $\xi \in C^\alpha(M)$
  and a family of distributions $Z^t_p \in C^\alpha(M)$, $t \in [0,T], p \in M$,
  satisfying
  \begin{align*}
    |\langle Z^t_p, \varphi^\lambda_p \rangle| &\lesssim \lambda^{2\alpha + 2} \\
    Z^t_q(r) &= Z^t_p(r) +\int_0^t \langle \p_{t-s}(p,\cdot) - \p_{t-s}(q,\cdot), \xi \rangle ds\ \xi(r),
  \end{align*}
  where the action of the heat kernel $\p$ on $\xi$ is well-defined by Theorem \ref{thm:schauderForDistribution}.
  Define
  \begin{align*}
    ||\xi,Z||_{\alpha,2\alpha + 2, T}
    :=
    ||\xi||_{C^\alpha(M)}
    +
    \sup_{t \le T, q \in M, \lambda \in (0,1], \varphi \in \balloftestfunctions^{r,\delta}} \frac{|\langle Z^t_q, \varphi^\lambda_q \rangle|}{\lambda^{2\alpha + 2}},
  \end{align*}
  where $r := \lceil |\alpha| \rceil$ and $\delta$ is the radius of injectivity of $M$.
\end{definition}

In our application to white-noise forcing, $\xi$ will be the white noise on $M$
and $Z$ will be constructed via Gaussian renormalization in Section \ref{sec:gaussian}.

%wait ..isnt $Z_0$ enough then?
%\begin{align*}
%  \langle Z_x, \varphi^\lambda_x \rangle
%  :=
%  \langle Z_0 \pm \left[ (K \ast \xi)(x) - (K \ast \xi)(y) \right] \xi, \varphi^\lambda_x \rangle,
%\end{align*}
%and
%\begin{align*}
%  \langle \left[ (K \ast \xi)(x) - (K \ast \xi)(y) \right] \xi, \varphi^\lambda_x \rangle 
%  \lesssim
%  |x-y|
%  \lambda^\alpha,
%\end{align*}
%so apparently not ??

Now define the models for $\mcV$ and $\mcW$ as
\begin{align*}
  (\Pi^{\mcV}_p \symbol{\Xi})(z) &= \xi(z) \\
  (\Pi^{\mcV}_p \symbol{\I[\Xi]\Xi})(z) &= Z^t_p \\
  (\Pi^{\mcV}_p \cotang{\omega} \symbol{\Xi})(z) &= (\Pi^t_p \cotang{\omega})(z)(\Pi^t_p \symbol{\Xi})(z)\\
  (\Pi^{\mcW}_p \symbol{1})(z) &= 1 \\
  (\Pi^{\mcW}_p \symbol{\I[\Xi]})(z) &=
  \int_0^t \Bigl\langle \p_{t-r}(z,\cdot) - \p_{t-r}(p,\cdot), \xi \Bigr\rangle dr \\
  (\Pi^{\mcW}_p \cotang{\omega})(z) &= \omega \exp^{-1}_p( z ),
\end{align*}
with transports
\begin{align*}
  \Gamma^{t,\mcV}_{p\leftarrow q} \symbol{\Xi} &= \symbol{\Xi} \\
  \Gamma^{t,\mcV}_{p\leftarrow q} \symbol{\I[\Xi] \Xi}
  &= \symbol{\I[\Xi] \Xi}
  + 
  \left[
    \int_0^t \Bigl\langle \p_{t-r}(p,\cdot) - \p_{t-r}(q,\cdot), \xi \Bigr\rangle dr 
  \right] \symbol{\Xi} \\
  \Gamma^{t,\mcV}_{p\leftarrow q} \cotang{\omega} \symbol{\Xi}
  &=
  \cotang{\omega} \exp^{-1}_q(p) \symbol{\Xi}
  +
  d_p[ \cotang{\omega} \exp^{-1}_q ] \symbol{\Xi} \\
  \Gamma^{t,\mcW}_{p \leftarrow q} \symbol{1} &= \symbol{1} \\
  \Gamma^{t,\mcW}_{p \leftarrow q} \symbol{\I[\Xi]}
  &=
  \symbol{\I[\Xi]}
  +
  \left[
    \int_0^t \Bigl\langle \p_{t-r}(p,\cdot) - \p_{t-r}(q,\cdot), \xi \Bigr\rangle dr 
  \right] \symbol{1} \\
  \Gamma^{t,\mcW}_{p \leftarrow q} \cotang{\omega} &= \omega \exp^{-1}_q(p) \symbol{1}
  +
  d_p[ \cotang{\omega} \exp^{-1}_q ].
\end{align*}

\begin{lemma}
  \label{lem:theseAreInFactModels}
  These are in fact models with $\delta_M = \delta$ the radius of injectivity of $M$
  and the distances/norms of the model only depend on $\xi, Z$.
  Indeed for $\mcG = \mcV, \mcW$, $\gamma \in \R$
  \begin{align*}
    ||\Pi^{t,\mcG}||_{\beta;\gamma;M} &\lesssim 1 +  ||\xi, Z||_{\alpha,2\alpha +2,T}, % \\
    %||\Pi^{t,\mcG};\bar \Pi^{t,\mcG}||_{\beta;\gamma;M} &\lesssim ||(\xi-\bar \xi, Z - \bar Z)||_{\alpha,2\alpha+2},
  \end{align*}
  with $\beta = 2$ for $\mcG = \mcW$ and $\beta = 2 + \alpha$ for $\mcG = \mcV$.
\end{lemma}
\begin{proof}
  By Lemma \ref{lem:cgammaViaExp}
  \begin{align*}
    |\langle \Pi^t_p \Xi, \varphi^\lambda_p \rangle|
    =
    |\langle \xi, \varphi^\lambda_p \rangle|
    \lesssim
    \lambda^{\alpha}
    ||\xi||_{C^\alpha(M)}.
  \end{align*}
  By definition
  \begin{align*}
    |\langle \Pi^t_p \I[ \Xi ] \Xi, \varphi^\lambda_p \rangle|
    =
    |\langle Z^t_p, \varphi^\lambda_p \rangle|
    \lesssim
    \lambda^{2\alpha + 2}
    ||\xi,Z||_{\alpha,2\alpha + 2,T}.
  \end{align*}

  Moreover 
  \begin{align*}
    |\langle \Pi^t_p \omega_p \Xi, \varphi^\lambda_p \rangle|
    =
    |\langle \xi, \omega_p \exp^{-1}_p(\cdot) \varphi^\lambda_p(\cdot) \rangle|
    \lesssim
    \lambda^{\alpha+1}
    ||\xi||_{C^\alpha(M)},
  \end{align*}
  since
  %\begin{align*}
    $\omega_p \exp^{-1}_p
    \varphi^\lambda_p
    %&=
    %\lambda^{-2}
    %(\omega_p(\circ) \varphi(\lambda^{-1} \cdot) \circ \exp^{-1}_p \\
    %&=
    %\lambda
    %\lambda^{-1}
    %(\omega_p(\lambda^{-2} \circ) \varphi(\lambda^{-1} \cdot) \circ \exp^{-1}_p \\
    =
    \lambda \psi^{\lambda}_p$,
  %\end{align*}
  with $\psi(\cdot) = \omega_p(\cdot) \phi(\cdot)$.

  Regarding transport, both the transport of $\Xi$ and $\I[\Xi] \Xi$ are exact by definition
  and
  \begin{align*}
    |\langle
      \Pi^{t,\mcV}_q \omega_q \Xi
      -
      \Pi^{t,\mcV}_p \Gamma_{p \leftarrow q} \omega_q \Xi, \varphi^\lambda_p
    \rangle|
    &=
    |\langle
      \xi
      \left( 
        \Pi^{t,\mcW}_q \omega_q
        -
        \Pi^{t,\mcW}_p \Gamma_{p \leftarrow q} \omega_q
      \right),
      \varphi^\lambda_p
    \rangle|
    \lesssim
    \lambda^{\alpha + 2}
    ||\xi||_{C^\alpha(M)},
  \end{align*}
  where we used Lemma \ref{lem:polynomialError2} for the last step.

  Finally
  \begin{align*}
    |\Gamma^{t,\mcV}_{p \leftarrow q} \I[ \Xi ] \Xi|_\alpha
    =
    |\int_0^t \langle \p_s(p,\cdot) - \p_s(q,\cdot), \xi \rangle ds|
    \lesssim
    d(p,q)^{\alpha + 2},
  \end{align*}
  by the Schauder estimate Theorem \ref{thm:schauderForDistribution},
  and
  \begin{align*}
    |\Gamma^{t,\mcV}_{p \leftarrow q} \omega_q \Xi|_{2\alpha + 2}
    &= 0 \\
    |\Gamma^{t,\mcV}_{p \leftarrow q} \omega_q \Xi|_{\alpha}
    &=
    |\Gamma^{t,\mcW}_{p \leftarrow q} \omega_q|_0 \\
    &\le
    d(p,q)
  \end{align*}
  by Lemma \ref{lem:linearPolynomialsAreAModel}.
  Hence
  \begin{align*}
    &||\Pi^{t,\mcV};\Gamma^{t,\mcV}||_{\beta,\gamma,M}
    \lesssim 1 + ||\xi,Z||_{\alpha,2\alpha +2,T},
  \end{align*}
  when $\beta := 2 + \alpha$.
  %what do we actually need here?
  %{\color{red}THIS DOES NOT VANISH FOR $\xi$ VANISHING !!??}

  %%%%% Regarding the distance, we get similarily
  %%%%% \begin{align*}
  %%%%%   &||\Pi^{t,\mcV}, \Gamma^{t,\mcV};\bar \Pi^t \bar \Gamma^{t,\mcV}||_{\beta,\gamma,M}
  %%%%%   \lesssim
  %%%%%   ||(\xi-\bar \xi, Z - \bar Z)||_{\alpha,2\alpha+2}.
  %%%%% \end{align*}

  Analogously, one gets the bounds for $\mcW$ with $\beta = 2$.
\end{proof}

\section{Schauder estimates}
\label{sec:schauder}
Let $\p$ be the heat kernel on $M$.
We start with a Schauder estimate for distributions.
Since its proof follows the same idea as the upcoming Schauder estimate
for modelled distributions, we omit the proof of the next theorem.
\begin{theorem}
  \label{thm:schauderForDistribution}
  Let $T > 0$, 
  and $F \in L^\infty([0,T],C^\alpha(M))$, for $\alpha \in (-2,-1)$.
  Then for $t \in [0,T]$
  \begin{align*}
    |\int_0^t \Bigl\langle \p_{t-r}(p,\cdot), F_r \Bigr\rangle dr
    -
    \int_0^t \Bigl\langle \p_{t-r}(q,\cdot), F_r \Bigr\rangle dr|
    \lesssim
    \sup_{t\le T} ||F_t||_{C^\alpha(M)} d(p,q)^{2+\alpha}.
  \end{align*}
\end{theorem}

We now prove an extension of this classical result to the space of modelled distributions.
For
\begin{align*}
  f(t,p) = f_\alpha(t,p) \symbol{\Xi} + f_{2 + 2\alpha}(t,p) \symbol{\I[\Xi]\Xi} + \cotang{f_{1+\alpha}(t,p) \symbol{\Xi}},
\end{align*}
an element of $\mcD^{\gamma,\gamma_0}_T( \mcV )$%
\footnote{Recall from the beginning of this section that $f_\alpha, f_{2+2\alpha}$ are real-valued and $f_{1_\alpha}$ is a section
          of $T^* M$.}, define
\begin{align*}
  (\mathcal{K}_t f)(p) := h := h_0(t,p) \symbol{\1} + h_{2+\alpha}(t,p) \symbol{\I[\Xi]} + \cotang{h_1^i(t,p)},
\end{align*}
with
\begin{align*}
  h_0(t,p) &= \int_0^t \langle \p_{t-s}(p,\cdot), \mathcal{R}_s f(s) \rangle ds \\
  h_{2+\alpha}(t,p) &=  f_\alpha(t,p) \\
  \cotang{h_1(t,p)} &= d|_p \left( z \mapsto \int_0^t \langle \p_{t-s}(z,\cdot), \mathcal{R}_s f(s) - f_\alpha(t,p) \Pi^t_p \symbol{\Xi} \rangle ds \right)
\end{align*}
The well-definedness of these terms is part of the following theorem.
\begin{theorem}[Schauder estimate]
  \label{thm:schauder}
  For $\alpha \in (-4/3,-1),$ with $\gamma \in (0,2\alpha + 8/3)$, set $\vareps := (2\alpha + 8/3 - \gamma)/4$ and $\gamma_0 = \alpha/2 + 1 - \vareps$.
  Let $T > 0$
  and $f \in \mcD^{\gamma,\gamma_0}_T( \mcV )$. Then, for all $t\in [0,T],$ 
  $$\mcR_t \mathcal{K} f = \int_0^t \langle \p_{t-s}, \mcR_s f(s) \rangle ds.$$
  Moreover, $\mathcal{K} f \in \mcD^{\bar \gamma, \bar \gamma_0}(\mcW)$,
  with $\bar \gamma = \gamma + 4/3$, $\bar \gamma_0 = \gamma_0$
  and
  \begin{align*}
    ||\mathcal{K} f||_{\mcD^{\bar \gamma, \bar \gamma_0, \normscaling}_T(\mcW)}
    \lesssim
    ||f||_{\mcD^{\gamma, \gamma_0, \normscaling}_T(\mcV)}
    \left( T^\vareps + T^\vareps \normscaling + \frac{1}{\normscaling} \right).
  \end{align*}
\end{theorem}
\begin{remark}
  \label{rem:modified}
  Here we can see why we introduced the modified norm 
    $||.||_{\mcD^{\bar \gamma, \bar \gamma_0, \normscaling}_T(\mcW)}$.
  Without it, i.e. with $\normscaling \equiv 1$, the factor on the right hand side cannot be made small,
  which is necessary for the fixpoint argument.
\end{remark}
\begin{remark}
  Contrary to classical Schauder estimates, we only get an ``improvement of $4/3$ derivatives''.
  In order to get an ``improvement of $2$ derivatives'' one has to include quadratic polynomials in the regularity structure.
  This is also the reason why we have to choose $\gamma, \gamma_0$ in such a specific way.
  Note that an improvement by $4/3$ will be enough to set up the fix-point argument.

  To be specific, in order to get an ``improvement of $2$ derivatives'' the complete list of symbols necessary is, ordered by homogeneity,
  \begin{align*}
    &\Xi,
    \Xi \I[ \Xi ] ],
    \Xi X_i,
    1,
    \Xi \I[ \Xi \I[ \Xi ] ],
    \Xi \I[ \Xi X_i ],
    \I[ \Xi ],
    \Xi X_i X_j,
    \Xi X_1,
    X_i,
    \Xi \I[ \Xi \I[ \Xi \I[ \Xi ] ] ],
    \Xi \I[ \Xi \I[ \Xi X_i ] ],\\
    &
    \Xi \I[ \Xi X_i X_j ],
    \I[ \Xi \I[ \Xi ] ],
    \Xi X_i X_j X_k,
    \I[ \Xi X_i ],
    X_i X_j, X_1,
  \end{align*}
  where $i,j = 2,3$ stand for the space-directions.\footnote{Assuming that one builds a regularity structure including space \emph{and} time.}
  These symbols would be the building blocks for the regularity structure on flat space.
  On a manifold the polynomials would represent the respective symmetric covariant tensor bundles, as laid out in Section \ref{sec:poly}.
  The Schauder estimate has to be shown on the level of each of theses symbols, and hence a treatment ``by hand'' as we do here would be cumbersome.
\end{remark}
\begin{remark}
  The following proof based on the heat kernel (almost) being a scaled test function
  goes back, in the flat case, to 
  \cite{bib:CM2016}.
  A proof 
  splitting up the heat kernel into a sum of smooth, compactly supported kernels
  (following the strategy of \cite{bib:hairer})
  is also possible, but more cumbersome.
\end{remark}
\begin{proof}[Proof of Theorem \ref{thm:schauder}]
  The first statement follows from the definition of $h_0$
  and the fact that reconstruction of modelled distributions
  taking values only in positive homogeneities is given
  by the projection onto homogeneity $0$,
  see Lemma \ref{cor:reconstructionPositiveHomogeneity}.
  
  Recall that $\delta_M = \delta$, the radius of injectivity.
  %The H\"older type norm for $\mcD^{\gamma,\gamma_0,\normscaling}$ can
  By Remark \ref{rem:mcDCutoff} we can, and will
  only consider $d(p,y) < \deltaf$.
  Introduce the short notation
  \newcommand{\normf}{C_f}
  \newcommand{\normpi}{C_\Pi}
  \begin{align*}
     \normf &:= ||f||_{\mcD^{\gamma,\gamma_0,\normscaling}_T(M,\mcV)} \\
    \normpi &:= \sup_{t \le T} ||\Pi^{t,\mcV},\Gamma^{t,\mcV}||_{\beta,M}.
  \end{align*}
  Note that $||\xi||_{C^\alpha(M)} \le \normpi$.

  We shall need the following facts.
  Since
  \begin{align*}
    |f(t,p) - \Gamma^t_{p \leftarrow q}f(t,q)|_\alpha \lesssim \frac{\normf}{\normscaling} d(p,q)^{\gamma - \alpha},
  \end{align*}
  we have
  \newcommand{\fHolderConstant}{\left( \frac{\normf}{\normscaling} + \normf + \normf \normpi \right)}
  \begin{align}
    &|f_\alpha(t,p) - f_\alpha(t,q)| \notag \\
    &\lesssim
    |f(t,p) - \Gamma^t_{p \leftarrow q}f(t,q)|_\alpha
    +
    |f_{2\alpha + 2}(t,q) \int_0^t \langle \p_{t-s}(p,\cdot) - \p_{t-s}(q,\cdot), \xi \rangle ds|
    +
    |f_{1+\alpha}(t,q) \exp_q^{-1}(p)| \notag \\
    &\lesssim
    \frac{\normf}{\normscaling}
    d(p,q)^{\gamma-\alpha}
    +
    \normf
    ||\xi||_{C^\alpha(M)}
    d(p,q)^{2 + \alpha}
    +
    \normf
    d(p,q) \notag \\
    &\lesssim
    \fHolderConstant 
    d(p,q)^{2+\alpha}, \label{eq:falphaIsHolder}
  \end{align}
  where we used the classical Schauder estimate Theorem \ref{thm:schauderForDistribution}.

  Moreover for a function $\varphi$ satisfying the assumptions of
  Lemma \ref{lem:rem221ForTheModel}
  and
  Lemma \ref{lem:rem221ForSForReconstruction}
  (recall that $\mcR_t$ is the reconstruction operator of Theorem \ref{thm:reconstructionOnTheManifold-new} associated to the model $(\Pi^t,\Gamma^t)$)
  \begin{align}
    &|\langle \mcR_t f(t) - f_\alpha(t,p) \xi, \varphi \circ \exp_p^{-1} \rangle| \notag \\
    &\le
    |\langle \mcR_t f(t) - \Pi^t_p f(t), \varphi \circ \exp_p^{-1} \rangle|
    +
    |\langle \Pi^t_p f(t) - f_\alpha(t,p) \xi, \varphi \circ \exp_p^{-1} \rangle| \notag \\
    &=
    |\langle \mcR_t f(t) - \Pi^t_p f(t), \varphi \circ \exp_p^{-1} \rangle|
    +
    |\langle
    f_{2\alpha+2}(t,p) \Pi^t_p\left( \symbol{\I[\Xi]\Xi} \right)
    +
    \Pi^t_p\left( \cotang{f_{\alpha+1}(t,p)} \symbol{\Xi} \right), \varphi \circ \exp_p^{-1} \rangle| \notag \\
    &\lesssim
    \normf \normpi \lambda^\gamma
    +
    \normf \normpi \lambda^{2\alpha + 2}
    +
    \normf \normpi \lambda^{\alpha + 1} \notag \\
    &\lesssim \normf \normpi \lambda^{2\alpha + 2}, \label{eq:almostReconstructionBound}
  \end{align}
  and similarily
  \begin{align}
    |\langle \mcR_t f(t), \varphi \circ \exp_p^{-1} \rangle|
    &\le
    |\langle \mcR_t f(t) - \Pi^t_p f(t), \varphi \circ \exp_p^{-1} \rangle|
    +
    |\langle \Pi^t_p f(t), \varphi \circ \exp_p^{-1} \rangle| \notag \\
    &\lesssim
    \normf \normpi
    \lambda^\gamma
    +
    \normf \normpi \left( \lambda^{\alpha} + \lambda^{2\alpha + 2} + \lambda^{\alpha + 1} \right) \notag \\
    &\lesssim
    \normf \normpi
    \lambda^\alpha. \label{eq:almostReconstructionBoundII}
  \end{align}

  We now estimate each term in the definition of the norm $||\mathcal{K} f||_{\mcD^{\bar \gamma, \bar \gamma_0, \normscaling}_T(\mcW)}$.

  \textbf{Space regularity}\\
  \textbf{Homogeneity $0$}
  \begin{align*}
    &(h(t,p) - \Gamma^t_{p \leftarrow q} h(t,q))_0 \\
    &=
    h_0(t,p) - h_0(t,q)
    - h_{\alpha+2}(t,q) (\Gamma^t_{p\leftarrow q} \symbol{ \I[\Xi] })_0
    - (\Gamma^t_{p\leftarrow q} \cotang{h_1(t,q)} ))_0 \\
    &=
    \int_0^t \Bigl\langle \p_{t-s}(p,\cdot), \mcR_s f(s) \Bigr\rangle ds
    -
    \int_0^t \Bigl\langle \p_{t-s}(q,\cdot), \mcR_s f(s) \Bigr\rangle ds
    - f_\alpha(t,q) \int_0^t \langle \p_{t-r}(p,\cdot) - \p_{t-r}(q,\cdot), \xi \rangle dr \\
    &\qquad
    - d|_q \left( z \mapsto \int_0^t \Bigl\langle \p_{t-s}(z,\cdot), \mathcal{R}_s f(s) - f_\alpha(t,q) \Pi^t_q \symbol{\Xi} \Bigr\rangle ds \right) \exp^{-1}_q(p) \\
    &=
    \int_0^t
      \Bigl\langle \p_{t-s}(p,\cdot) - \p_{t-s}(q,\cdot) - d|_q \p_{t-s}(q,\cdot) (\exp^{-1}_q(p)),
      \mathcal{R}_s f(s) - f_\alpha(t,q) \xi \Bigr\rangle ds \\
    &=
    \int_0^t
      \Bigl\langle \p^N_{t-s}(p,\cdot) - \p^N_{t-s}(q,\cdot) - d|_q \p^N_{t-s}(q,\cdot)( \exp^{-1}_q(p)),
      \mathcal{R}_s f(s) - f_\alpha(t,q) \xi \Bigr\rangle ds \\
    &\quad
    +
    \int_0^t
      \Bigl\langle R^N_{t-s}(p,\cdot) - R^N_{t-s}(q,\cdot) - d|_q R^N_{t-s}(q,\cdot) (\exp^{-1}_q(p)), \mathcal{R}_s f(s) - f_\alpha(t,q) \xi \Bigr\rangle ds,
  \end{align*}
  where $\p = \p^N + R^N$ using heat asymptotics, Theorem \ref{thm:heatKernelEstimates}.

  Regarding the easier term involving $R^N$ we write
  \begin{align*}
    &\int_0^t
    \Bigl\langle R^N_{t-s}(p,\cdot) - R^N_{t-s}(q,\cdot) - d|_q R^N_{t-s}(q,\cdot) \exp^{-1}_q(p), \mathcal{R}_s f(s) - f_\alpha(t,q) \xi \Bigr\rangle ds \\
    &\quad=
    \frac{1}{2}
    \int_0^t
      \Bigl\langle
      \int_0^1
      \nabla^2|_{\gamma(\theta)}
        R^N( \circ, \cdot )\left( \dot \gamma(\theta)^{\otimes 2} \right)
        (1-\theta)^2 d\theta,
      \mathcal{R}_s f(s) - f_\alpha(t,q) \xi \Bigr\rangle ds,
  \end{align*}
  where $\nabla$ acts on the dummy variable $\circ$
  and convolution acts on $\cdot$ and $\gamma$ is the geodesic connection $q$ to $p$.
  Since 
  \begin{align*}
    ||\mcR_s f(s) - f_\alpha(t,q) \xi||_\alpha
    \lesssim
    \normf \normpi,
  \end{align*}
  this expression is well-defined for $N$ large enough
  and of order
  \begin{align*}
    \normf \normpi \sup_{\theta \le 1} |\dot \gamma(\theta)|^2
    = 
    \normf \normpi d(p,q)^2.
  \end{align*}

  We now treat the term involving $\p^N$. Denoting by $g(t,s)$ the integrand of the above integral,  for $s \in [t-d(p,q)^2,t]$,
  \begin{align*}
    |g(t,s)|
    &\le
    |\Bigl\langle \p^N_{t-s}(p,\cdot), \mathcal{R}_s f(s) - f_\alpha(t,q) \xi \Bigr\rangle|
    +
    |\Bigl\langle \p^N_{t-s}(q,\cdot), \mathcal{R}_s f(s) - f_\alpha(t,q) \xi \Bigr\rangle|\\
    &\quad
    +
    |\Bigl\langle d|_q \p^N_{t-s}(q,\cdot)( \exp^{-1}_q(p)), \mathcal{R}_s f(s) - f_\alpha(t,q) \xi \Bigr\rangle|.
  \end{align*}

  The first term we bound as
  \begin{align*}
    \Big| \Bigl\langle \p^N_{t-s}(p,\cdot), \mathcal{R}_s f(s) - f_\alpha(t,q) \xi \Bigr\rangle \Big|
    &\le
    \Big| \Bigl\langle \p^N_{t-s}(p,\cdot), \mathcal{R}_s f(s) - f_\alpha(s,p) \xi \Bigr\rangle \Big|
    +
    \Big| \Bigl\langle \p^N_{t-s}(p,\cdot), \left( f_\alpha(s,p) - f_\alpha(t,q) \right) \xi \Bigr\rangle \Big| \\
    &\lesssim
    \normf \normpi 
    |t-s|^{(2\alpha + 2)/2} \\
    &\qquad
    +
    \normpi
    |t-s|^{\alpha/2}
    \left(
      \fHolderConstant
      d(p,q)^{2+\alpha}
      +
      \normf
      |t-s|^{\gamma_0}
    \right),
  \end{align*}
  where we used \eqref{eq:almostReconstructionBound} together with Lemma \ref{lem:kernelIsScaledFunctionI} \ref{item:lemKernelIsScaledFunctioni},
  as well as to H\"older continuity of $f_\alpha$ in space \eqref{eq:falphaIsHolder} and in time.

  The second we bound as
  \begin{align*}
    &\Big| \Bigl\langle \p^N_{t-s}(q,\cdot),\mathcal{R}_s f(s) - f_\alpha(t,q) \xi \Bigr\rangle\Big| \\
    &\le
    \Big| \Bigl\langle \p^N_{t-s}(q,\cdot), \mathcal{R}_s f(s) - f_\alpha(s,q) \xi \Bigr\rangle\Big|
    +
    \Big| \Bigl\langle \p^N_{t-s}(q,\cdot), \left( f_\alpha(s,q) -  f_\alpha(t,q) \right) \xi \Bigr\rangle\Big| \\
    &\lesssim
    \normf \normpi
    |t-s|^{(2\alpha + 2)/2}
    +
    \normpi
    \normf
    |t-s|^{\alpha/2} |t-s|^{\gamma_0},
  \end{align*}
  where we used \eqref{eq:almostReconstructionBound} together with Lemma \ref{lem:kernelIsScaledFunctionI} \ref{item:lemKernelIsScaledFunctioni}
  as well as the H\"older continuity of $f_\alpha$ in time.

  The last one we bound as
  \begin{align*}
    &\Big| \Bigl\langle d|_q \p^N_{t-s}(q,\cdot) \exp^{-1}_q(p),
                 \mathcal{R}_s f(s) - f_\alpha(t,q) \xi \Bigr\rangle\Big|  \\
    &\le
    \Big| \Bigl\langle d|_q \p^N_{t-s}(q,\cdot) \exp^{-1}_q(p),
                 \mathcal{R}_s f(s) - f_\alpha(s,q) \xi \Bigr\rangle\Big|
                 \\
                 &\quad
    +
    \Big| \Bigl\langle d|_q \p^N_{t-s}(q,\cdot) \exp^{-1}_q(p),
                 \left( f_\alpha(s,q) - f_\alpha(t,q) \right) \xi \Bigr\rangle\Big|  \\
    &\lesssim
    \normf \normpi
    d(p,q) |t-s|^{(2\alpha +2)/2 - 1/2}
    +
    \normpi
    d(p,q) |t-s|^{\alpha/2 - 1/2}
      \normf
      |t-s|^{\gamma_0},
  \end{align*}
  where we used \eqref{eq:almostReconstructionBound} together with Lemma \ref{lem:kernelIsScaledFunctionI} \ref{item:lemKernelIsScaledFunctionii}
  as well as to H\"older continuity of $f_\alpha$ in space \eqref{eq:falphaIsHolder} and in time.

  Hence
  \begin{align*}
    |g(t,s)|
    &\lesssim
    \Bigl(
      |t-s|^{1/2}
      +
      d(p,q)
    \Bigr) \\
    &\quad \times
    \Bigl(
      \normf \normpi 
      |t-s|^{(2\alpha + 2)/2 - 1/2}
      +
      \normpi
      \fHolderConstant
       d(p,q)^{2+\alpha}
       |t-s|^{\alpha/2 - 1/2}\\
      &\qquad
      +
      \normpi \normf |t-s|^{\alpha/2 - 1/2} |t-s|^{\gamma_0}
    \Bigr),
  \end{align*}
  and then by Lemma \ref{lem:timeIntegral}
  \begin{align*}
    \int_{t-d(p,q)^2}^t |g(t,s)| ds
    &\lesssim
    T^{\vareps}
    \Bigl(
      \normf \normpi 
      d(p,q)^{2\alpha + 4 - 2\vareps}
      %|t-s|^{(2\alpha + 2)/2 + 1/2}
      +
      \normpi
      \fHolderConstant
        d(p,q)^{2 \alpha + 4 - 2\vareps}
        %|t-s|^{\alpha/2 + 1/2}
        \\
      &\qquad
      +
      \normpi \normf
      d(p,q)^{\alpha + 2 + 2 \gamma_0 - 2\vareps}
    \Bigr),
  \end{align*}
  if
  \begin{empheq}[box=\mybluebox]{align*}
    (2\alpha+2)/2,\
    \alpha/2 + \gamma_0,\
    \alpha/2,\
    (\alpha+2)/2 - 1/2,\
    \alpha/2 - 1/2 + \gamma_0 > -1 + \vareps.
  \end{empheq}
  Then the following are upper bounds to $\bar \gamma$
  \begin{empheq}[box=\mybluebox]{align*}
    2\alpha + 4 - 2\vareps,\ \alpha + 2 \gamma_0 + 2 - 2\vareps.
  \end{empheq}
  Both are satisfied under our assumptions.

  ~\\
  Now consider $s \in [0,t-d(p,q)^2]$.
  By \cite[Theorem 6.1]{bib:driverSemko} we have
  \begin{align*}
    \p^N_{t-s}(p,\cdot) - \p^N_{t-s}(q,\cdot) - d|_q \p^N_{t-s}(q,\cdot) \exp^{-1}_q(p)
    &=
    \int_0^1 
    \nabla^2 \p^N_{t-s}( \gamma(r), \cdot ) \left( \dot \gamma(r) \otimes \dot \gamma(r) \right) (1-r) dr,
  \end{align*}
  where $\gamma(r) := \exp_q( r v ), v := \exp_q^{-1}(p),$ for any $r\in [0,1],$ and $\nabla^2$ is acting on the first variable of $\p^N$.
  Now
  \begin{align*}
    g(t,s) &=\Bigl\langle \p^N_{t-s}(p,\cdot) - \p^N_{t-s}(q,\cdot) - d|_q \p^N_{t-s}(q,\cdot) \exp^{-1}_q(p), \mathcal{R}_s f(s) - f_\alpha(t,q) \xi \Bigr\rangle \\
    &=
    \Bigl\langle \p^N_{t-s}(p,\cdot) - \p^N_{t-s}(q,\cdot) - d|_q \p^N_{t-s}(q,\cdot) \exp^{-1}_q(p), \mathcal{R}_s f(s) - f_\alpha(s,q) \xi \Bigr\rangle \\
    &\quad
    +
    \Bigl\langle \p^N_{t-s}(p,\cdot) - \p^N_{t-s}(q,\cdot) - d|_q \p^N_{t-s}(q,\cdot) \exp^{-1}_q(p), \left( f_\alpha(s,q) - f_\alpha(t,q) \right) \xi \Bigr\rangle.
  \end{align*}
  The first term we bound as
  \begin{align*}
    &\Big| \Bigl\langle \p^N_{t-s}(p,\cdot) - \p^N_{t-s}(q,\cdot) - d|_q \p^N_{t-s}(q,\cdot) \exp^{-1}_q(p), \mathcal{R}_s f(s) - f_\alpha(s,q) \xi \Bigr\rangle \Big| \\
    &=
    \Big| \int_0^1 \Bigl\langle (\nabla^2 \p_{t-s}( \gamma(r), \cdot ) \left( \dot \gamma(r) \otimes \dot \gamma(r) \right),  \mathcal{R}_s f(s) - f_\alpha(s,q) \xi \Bigr\rangle (1-r) dr \Big| \\
    &\lesssim
      \int_0^1 |v|^2 \normf \normpi |t-s|^{(2+2\alpha)/2-1} (1-r) dr \\
    &\lesssim
    |v|^2
    \normf \normpi |t-s|^{(2+2\alpha)/2-1} \\
    &=
    d(p,q)^2
    \normf \normpi |t-s|^{(2+2\alpha)/2-1},
  \end{align*}
  where we used \eqref{eq:almostReconstructionBound} together with Lemma \ref{lem:kernelIsScaledFunctionI}.
  \footnote{
    In coordinates,
    \begin{align*}
      \nabla^2 \p^n_{t-s}(\gamma(r), \cdot)
      =
      \left( \partial_{ij} \p^n_{t-s}(\gamma(r), \cdot) - \sum_k\Gamma_{ij}^k \partial_k \p^n_{t-s}(\gamma(r), \cdot) \right) dx^i \otimes dx^j,
    \end{align*}
    where $\Gamma$ are the Christoffel symbols.
    This gives the quadratic factor in $|\dot \gamma(r)| = d(p,q)$.
    The blowup in $t-s$ follows from an application of
    Lemma \ref{lem:kernelIsScaledFunctionI} \ref{item:lemKernelIsScaledFunctioni}, \ref{item:lemKernelIsScaledFunctionii} to the components here.
  }

  The second term we bound as
  \begin{align*}
    &\Big| \Bigl\langle \p^N_{t-s}(p,\cdot) - \p^N_{t-s}(q,\cdot) - d|_q \p^N_{t-s}(q,\cdot) \exp^{-1}_q(p), \left( f_\alpha(s,q) - f_\alpha(t,q) \right) \xi \Bigr\rangle\Big| \\
    &=
    \Big|\int_0^1 \Bigl\langle (\nabla d \p_{t-s}( \gamma(r), \cdot ) \left( \dot \gamma(r) \otimes \dot \gamma(r) \right),  \left( f_\alpha(s,q) - f_\alpha(t,q) \right) \xi \Bigr\rangle dr \Big| \\
    &\lesssim
    d(p,q)^2
    \normpi
    \normf
    |t-s|^{\alpha/2 - 1}
    |t-s|^{\gamma_0},
  \end{align*}
  where we used Lemma \ref{lem:kernelIsScaledFunctionI}
  and the H\"older continuity of $f_\alpha$ in time.

  Hence by Lemma \ref{lem:timeIntegral}
  \begin{align*}
    \int_0^{t-d(p,q)^2} g(t,s) ds
    &\lesssim
    T^\vareps
    \Bigl(
      \normf \normpi d(p,q)^{4+2\alpha - 2\vareps}
      +
      \normpi
      \normf
      d(p,q)^{\alpha + 2 + 2 \gamma_0 - 2\vareps}
    \Bigr),
    %||f||_{\mcD^\gamma_T(\mcV)}
    %T^\vareps
    %\left(
    %d(p,q)^{2\alpha+4 - 2\vareps}
    %+
    %d(p,q)^{2 +\alpha+2\gamma_0 - 2\vareps}
    %\right),
  \end{align*}
  if
  \begin{empheq}[box=\mybluebox]{align*}
    (2\alpha + 2)/2 - 1, \ \alpha/2 + \gamma_0 &< - 1 + \vareps.
  \end{empheq}

  Then the following are upper bounds to $\bar \gamma$
  \begin{empheq}[box=\mybluebox]{align*}
    2\alpha + 4 - 2\vareps,\ 2 + \alpha + 2\gamma_0 - 2\vareps.
  \end{empheq}
  Both are satisfied under our assumptions.

  Hence
  \begin{align*}
    \normscaling |f(t,p) -  \Gamma^t_{p \leftarrow q} f(t,q)|_\alpha
    \lesssim
    \normf
    \Bigl(
      T^\vareps
      +
      T^\vareps
      \normscaling
    \Bigr)
    d(p,q)^{\bar \gamma}.
  \end{align*}

  \textbf{Homogeneity $\alpha + 2$}
  \begin{align*}
    |h(t,p) - \Gamma^t_{p \leftarrow q} h(t,q)|_{\alpha + 2}
    &=
    |h_{\alpha+2}(t,p) - h_{\alpha+2}(t,q)| \\
    &=
    |f_{\alpha}(t,p) - f_{\alpha}(t,q)| \\
    &\lesssim
    \frac{1}{\normscaling} ||f||_{\mcD^{\gamma,\normscaling}_T(\mcV)}
    d(p,q)^{\gamma - \alpha} \\
    &=
    \frac{1}{\normscaling}
    ||f||_{\mcD^{\gamma,\normscaling}_T(\mcV)} d(p,q)^{(\gamma + 2) - (\alpha+2)},
  \end{align*}
  so we need
  \begin{empheq}[box=\mybluebox]{align*}
    \bar \gamma \le \gamma + 2,
  \end{empheq}
  which is satisfied under our assumptions.

  \textbf{Homogeneity $1$}\\
  As on homogeneity $0,$ we write $\p = \p^N + R^N$. We only treat the term involving $\p^N$.
  \begin{align*}
    (h(t,p) - \Gamma^t_{p \leftarrow q} h(t,q))_1
    &=
    \int_0^t
      \Bigl\langle d \p^N_{t-s}(p,\cdot), \mathcal{R}_s f(s) - f_\alpha(t,p) \Pi^t_p \symbol{\Xi} \Bigr\rangle ds \\
    &\quad
    -
    d|_p \left[
      z \mapsto \int_0^t \Bigl\langle d \p^N_{t-s}(q,\cdot), \mathcal{R}_s f(s) - f_\alpha(t,q) \Pi^t_q \symbol{\Xi} \Bigr\rangle ds
      \exp^{-1}_q( z ) \right] \\
    &=: \int_0^t g(t,s) ds.
  \end{align*}
  It is enough to bound this expression acting on $X \in T_p M$.
  Write
  \begin{align*}
    \zeta^s_p := \mathcal{R}_s f(s) - f_\alpha(s,p) \Pi^s_p \symbol{\Xi}.
  \end{align*}

  For $s \in [t-d(p,q)^2, t]$ we bound
  ($\bullet$ denotes the dummy variable on which $X$ is acting, $\cdot$ denotes the dummy variable in the distribution-pairing)
  \begin{align*}
    &\Big| \Big\langle d|_p \p^N_{t-s}(\bullet,\cdot)\Bigl(X\Bigr), \mathcal{R}_s f(s) - f_\alpha(t,p) \Pi^s_p \symbol{\Xi} \Big\rangle \Big| \\
    &\le
    \Big| \Big\langle d|_p \p^N_{t-s}(\bullet,\cdot)\Bigl(X\Bigr), \zeta^s_q \Big\rangle \Big|
    +
    \Big| \Big\langle d|_p \p^N_{t-s}(\bullet,\cdot)\Bigl(X\Bigr), \left( f_\alpha(s,q) - f_\alpha(t,q) \right) \xi \Big\rangle \Big| \\
    &=
    \Big| \Big\langle X\Bigl(\p^N_{t-s}(\bullet,\cdot)\Bigr), \zeta^s_p \Big\rangle \Big|
    +
    \Big| \Big\langle X\Bigl(\p^N_{t-s}(\bullet,\cdot)\Bigr), \left( f_\alpha(s,q) - f_\alpha(t,q) \right) \xi \Big\rangle \Big| \\
    &\lesssim
    \normf \normpi |t-s|^{(2+2\alpha)/2-1/2}
    +
    \normf \normpi |t-s|^{\alpha/2 - 1/2} |t-s|^{\gamma_0},
  \end{align*}
  where we used \eqref{eq:almostReconstructionBound} together with Lemma \ref{lem:kernelIsScaledFunctionI} \ref{item:lemKernelIsScaledFunctionii},
  as well as the H\"older continuity of $f_\alpha$ in time.

  Now
  \begin{align*}
    &\Big| \Big\langle d|_p\left[ z \mapsto d \p^N_{t-s}(q,\cdot) \exp^{-1}_q( z ) \right]\Bigl(X\Bigr), \mathcal{R}_s f(s) - f_\alpha(t,p) \Pi^s_p \symbol{\Xi} \Big\rangle \Big| \\
    &\le
    \Big| \Big\langle d|_p\left[ z \mapsto d \p^N_{t-s}(q,\cdot) \exp^{-1}_q( z ) \right]\Bigl(X\Bigr), \zeta^s_q \Big\rangle \Big|
    +
    \Big| \Big\langle d|_p\left[ z \mapsto d \p^N_{t-s}(q,\cdot) \exp^{-1}_q( z ) \right]\Bigl(X\Bigr), \left( f_\alpha(s,q) - f_\alpha(t,q) \right) \xi \Big\rangle \Big| \\
    &=  
    \Big| \Big\langle d|_q \p^N_{t-s}(q,\cdot) d|_p \exp^{-1}_q( z ) \Bigl(X\Bigr), \zeta^s_q \Big\rangle \Big|
    +
    \Big| \Big\langle d|_q \p^N_{t-s}(q,\cdot) d|_p \exp^{-1}_q( z ) \Bigl(X\Bigr), \left( f_\alpha(s,q) - f_\alpha(t,q) \right) \xi \Big\rangle \Big| \\
    &\lesssim
    \normf \normpi |t-s|^{(2+2\alpha)/2-1/2}
    +
    \normf \normpi |t-s|^{\alpha/2 - 1/2} |t-s|^{\gamma_0},
  \end{align*}
  where we used \eqref{eq:almostReconstructionBound} together with Lemma \ref{lem:kernelIsScaledFunctionI}
  \ref{item:lemKernelIsScaledFunctionii} with $Y_p := d|_p \exp^{-1}_q( z ) \Bigl(X\Bigr)$,
  as well as the H\"older continuity of $f_\alpha$ in time.

  Hence by Lemma \ref{lem:timeIntegral}
  \begin{align*}
    \int_{t-d(p,q)^2}^t |g(t,s)| ds
    \lesssim
    T^\vareps \Bigl(
        \normf \normpi d(p,q)^{3+2\alpha-2\vareps}
        +
        \normf \normpi d(p,q)^{\alpha + 1 + 2 \gamma_0 - 2\vareps},
        \Bigr)
  \end{align*}
  if
  \begin{empheq}[box=\mybluebox]{align*}
    (1 + 2\alpha)/2,\ \alpha/2 - 1/2 + \gamma_0 > - 1 + \vareps.
  \end{empheq}
  Then the following are upper bounds to $\bar \gamma - 1$
  \begin{empheq}[box=\mybluebox]{align*}
    3 + 2\alpha - 2\vareps, \alpha + 1 + 2 \gamma_0 - 2\vareps.
  \end{empheq}
  Both are satisfied under our assumptions.

  Consider now $s \in [0,t-d(p,q)^2]$.
  Again it is enough to bound the term acting on some $X \in T_p M$.
  For notational simplicity %let $\p(z) := \p^{N}_{t-s}(z,\cdot)$
  \newcommand{\tmp}{v}
  let $\tmp(z) := d|_z \p^N_{t-s}(z,\cdot) \ d|_p \exp^{-1}_z \langle X \rangle$
  and $\zeta^s_p = \mcR_s f(s) - f_\alpha(s,p) \xi$.
  We then write the term to bound as
  \newcommand{\px}
  {\tmp(p)}
  %{d \p_{t-s}(p,\cdot)}
  \newcommand{\py}
  {\tmp(q)}
  %{d|_p \left[ z \mapsto d \p_{t-s}(q,\cdot) \exp^{-1}_q( z ) \right]}
   \begin{align*}
    &\Bigl\langle d \p^N_{t-s}(p,\cdot), \mathcal{R}_s f(s) - f_\alpha(t,p) \Pi^t_p \symbol{\Xi} \Bigr\rangle \langle X \rangle
    -
    \Bigl\langle d|_p \left[ z \mapsto d \p^N_{t-s}(q,\cdot) \exp^{-1}_q( z ) \right], \mathcal{R}_s f(s) - f_\alpha(t,q) \Pi^t_q \symbol{\Xi} \Bigr\rangle \langle X \rangle \\
    &=\Big\langle \px, \mathcal{R}_s f(s) - f_\alpha(t,p) \Pi^t_p \symbol{\Xi} \Big\rangle
    -
    \Big\langle \py, \mathcal{R}_s f(s) - f_\alpha(t,q) \Pi^t_q \symbol{\Xi} \Big\rangle \\
    &=
    \Big\langle \px - \py, \zeta^s_p \Big\rangle
    +
    \Big\langle \px, \left( f_\alpha(s,p) - f_\alpha(t,p) \right) \xi \Big\rangle
    -
    \Big\langle \py, \left( f_\alpha(s,p) - f_\alpha(t,q) \right) \xi \Big\rangle \\
    &=
    \Big\langle \px - \py, \zeta^s_p \Big\rangle
    +
    \Big\langle \px - \py, \left( f_\alpha(s,p) - f_\alpha(t,p) \right) \xi \Big\rangle
    +
    \Big\langle \py, \left( f_\alpha(t,q) - f_\alpha(t,p) \right) \xi \Big\rangle.
  \end{align*}
  Now with $\gamma(t) := \exp_q( t v ), v := \exp^{-1}_q(p)$,
  \begin{align*}
    \Big\langle \tmp(p) - \tmp(q), \zeta^s_p \Big\rangle
    &=
    \int_0^1 \Big\langle d|_{\gamma(r)} \tmp \langle \dot \gamma(r) \rangle, \zeta^s_p \Big\rangle dr \\
    &\lesssim
    d(p,q) \normf \normpi |t-s|^{(2+2\alpha)/2-1}.
  \end{align*}
  where we used \eqref{eq:almostReconstructionBound} together with Lemma \ref{lem:kernelIsScaledFunctionI} \ref{item:lemKernelIsScaledFunctioniii}.

  Similarily
  \begin{align*}
    | \Big\langle \px - \py, \left( f_\alpha(s,p) - f_\alpha(t,p) \right) \xi \Big\rangle |
    &=
    \Big| \int_0^1 \Big\langle d|_{\gamma(r)} \tmp \langle \dot \gamma(r) \rangle, \left( f_\alpha(s,p) - f_\alpha(t,p) \right) \xi \Big\rangle dr \Big| \\
    &\lesssim
    d(p,q) \normf \normpi |t-s|^{\gamma_0} |t-s|^{\alpha/2-1},
  \end{align*}
  where we used Lemma \ref{lem:kernelIsScaledFunctionI} \ref{item:lemKernelIsScaledFunctioniii} and the H\"older continuity of $f_\alpha$ in time.

  Finally
  \begin{align*}
    \Big| \Big\langle \py, \left( f_\alpha(t,q) - f_\alpha(t,p) \right) \xi \Big\rangle \Big|
    &\lesssim
    \fHolderConstant \normpi d(p,q)^{2+\alpha} |t-s|^{\alpha/2 - 1/2},
  \end{align*}
  where we used Lemma \ref{lem:kernelIsScaledFunctionI} \ref{item:lemKernelIsScaledFunctionii} and the H\"older continuity of $f_\alpha$ in space \eqref{eq:falphaIsHolder}.

  Hence by Lemma \ref{lem:timeIntegral}
  \begin{align*}
    %\int_{t-d(p,q)^2}^t g(t,s) ds
    \int_0^{t-d(p,q)^2} |g(t,s)| ds
    &\lesssim
    T^\vareps
    \Bigl(
      \normf \normpi d(p,q)^{3+2\alpha-2\vareps} \\
      &\qquad
      +
      \normf \normpi d(p,q)^{\gamma_0 + \alpha+1-2\vareps}
      +
      \fHolderConstant \normpi d(p,q)^{3+2\alpha-2\vareps}
    \Bigr)
    %\lesssim
    %T^\vareps
    %\Bigl(
    %  \normpi \left( 1 + \normpi \right) d(p,q)^{2+\alpha} |t-s|^{\alpha/2-1/2} 
    %  +
    %  \normf \normpi d(p,q) |t-s|^{2+2\alpha-2}.
    %\Bigr)
  \end{align*}
  if
  \begin{empheq}[box=\mybluebox]{align*}
    (2+2\alpha-2)/2,\
    \alpha/2 - 1/2,\
    \alpha/2 - 1 + \gamma_0 < - 1 + \vareps.
  \end{empheq}

  Then the following are upper bounds for $\bar \gamma - 1$
  \begin{empheq}[box=\mybluebox]{align*}
    2\alpha + 3 - 2\vareps,\ 1 + \alpha + \gamma_0 - 2\vareps.
  \end{empheq}
  Both are satisfied under our assumptions.

  Then
  \begin{align*}
    |f(t,p) - \Gamma^t_{p \leftarrow q} f(t,q)|_1
    \lesssim
    \normf
    \Bigl(
      T^\vareps
      +
      T^\vareps
      \normscaling
    \Bigr)
    d(p,q)^{\bar \gamma - 1}.
  \end{align*}

  ~\\

  %{\color{orange}
  \textbf{Time regularity}\\
  As on homogeneity $0$ we write $\p = \p^N + R^N$. We only treat the term involving $\p^N$.
  \begin{align*}
    h_0(t,p) - h_0(s,p)
    &=
    \int_0^t \Big\langle \p^N_{t-r}(p,\cdot), \mcR_r f(r) \Big\rangle dr
    -
    \int_0^s \Big\langle \p^N_{s-r}(p,\cdot), \mcR_r f(r) \Big\rangle dr \\
    &=
    \int_s^t \Big\langle \p^N_{t-r}(p,\cdot), \mcR_r f(r) \Big\rangle dr
    +
    \int_0^s \Big\langle \p^N_{t-r}(p,\cdot) - \p^N_{s-r}(p,\cdot), \mcR_r f(r) \Big\rangle dr.
  \end{align*}

  Now using \eqref{eq:almostReconstructionBoundII} and Lemma \ref{lem:kernelIsScaledFunctionI} \ref{item:lemKernelIsScaledFunctioni}
  \begin{align*}
    \int_s^t \Big| \Big\langle \p^N_{t-r}(p,\cdot), \mcR_r f(r) \Big\rangle \Big| dr
    &\lesssim
    \normf \normpi \int_s^t (t-r)^{\alpha/2} dr \\
    &=
    \normf \normpi \int_s^t (t-r)^\vareps (t-r)^{\alpha/2-\vareps} dr \\
    &\lesssim
    T^\vareps \normf \normpi |t-s|^{(\alpha+2)/2 - \vareps}.
  \end{align*}

  Further, again using \eqref{eq:almostReconstructionBoundII} and Lemma \ref{lem:kernelIsScaledFunctionI} \ref{item:lemKernelIsScaledFunctioni}
  \begin{align*}
    \int_0^s \Big| \Big\langle \p^N_{t-r}(p,\cdot) - \p^N_{s-r}(p,\cdot), \mcR_r f(r) \Big\rangle\Big| dr
    &=
    \normf \normpi \int_0^s \Big| \int_{s-r}^{t-r} \Big\langle \partial_t \p^N_\theta( p, \cdot ), \mcR_r f(r) \Big\rangle d\theta\Big| dr \\
    &\lesssim
    \normf \normpi \int_0^s \int_{s-r}^{t-r} \theta^{\alpha/2 - 1} d\theta dr \\
    &=
    \normf \normpi \frac{1}{(\alpha/2)} \int_0^s \left[  (t-r)^{\alpha/2} - (s-r)^{\alpha/2} \right] dr \\
    &\le
    T^\vareps \normf \normpi \left| \frac{1}{(\alpha/2)} \int_0^s \left[  (t-r)^{\alpha/2-\vareps} - (s-r)^{\alpha/2-\vareps} \right] dr \right| \\
    &=
    T^\vareps
    \normf \normpi
    \Big| - \frac{1}{(\alpha/2)} \frac{1}{(\alpha/2) + 1 - \vareps} \\
    &\qquad \times
    \left[ |t-s|^{\alpha/2 + 1 - \vareps} - t^{\alpha/2 + 1 - \vareps} + s^{\alpha/2 + 1 - \vareps} \right] \Big| \\
    &\lesssim
    T^\vareps
    \normf \normpi
    |t-s|^{\alpha/2 + 1 - \vareps},
  \end{align*}
  if
  \begin{empheq}[box=\mybluebox]{align*}
    \alpha/2 - \vareps &> -1.
  \end{empheq}
  We then need
  \begin{empheq}[box=\mybluebox]{align*}
    \bar \gamma_0 - \vareps \le \alpha/2 + 1 - \vareps.
  \end{empheq}
  Both are satisfied under our assumptions.

  Then
  \begin{align*}
    |h_0(t,p) - h_0(s,p)|
    \lesssim
    T^\vareps \normf |t-s|^{\bar \gamma_0}.
  \end{align*}
\end{proof}

We used the following lemmas.
\begin{lemma}   \label{lem:timeIntegral}
  Let $\rho_1, \rho_2 \in \R$, $g: \R^2 \to \R$
  and assume
  %If for for $\rho_1 > -1 $, $\rho_2 < -1$
  \begin{align*}
    g(t,s) &\le C_1 |t-s|^{\rho_1} \qquad, s \in [t-A, t] \\
    g(t,s) &\le C_2 |t-s|^{\rho_2} \qquad, s \in [0, t-A].
  \end{align*}
  Let $\vareps \ge 0$ such that $\rho_1 - \vareps > -1$, $\rho_2 - \vareps < -1$.
  Then
  \begin{align*}
    \int_{t-A}^t g(t,s)ds \lesssim C_1 T^\vareps A^{\rho_1 + 1 - \vareps} \\
    \int_0^{t-A} g(t,s)ds \lesssim C_2 T^\vareps A^{\rho_2 + 1 - \vareps}.
  \end{align*}
\end{lemma}
\begin{proof}
  Indeed
  \begin{align*}
    \int_{t-A}^t g(t,s) ds
    &\le
    \int_{t-A}^t |t-s|^{\rho_1} ds \\
    &\le
    T^\vareps  \int_{t-A}^t |t-s|^{\rho_1-\vareps} ds \\
    &=
    T^\vareps ( -\frac{1}{\rho_1+1} |t-s|^{\rho_1+1-\vareps} |^{t}_{t-A} ) \\
    &=
    T^\vareps (-\frac{1}{\rho_1+1} \left[ 0 - A^{\rho_1+1-\vareps} \right]) \\
    &\lesssim
    T^\vareps A^{\rho_1+1-\vareps},
  \end{align*}
  and
  \begin{align*}
    \int_0^{t-A} g(t,s)ds
    &\le
    \int_0^{t-A} |t-s|^{\rho_2} ds \\
    &=
    T^\vareps \int_0^{t-A} |t-s|^{\rho_2-\vareps} ds \\
    &=
    T^\vareps (-\frac{1}{\rho_2+1} |t-s|^{\rho_2+1-\vareps} |^{t-A}_{0}) \\
    &=
    T^\vareps (-\frac{1}{\rho_2+1} \left[ A^{\rho_2+1-\vareps} - t^{\rho_2+1-\vareps} \right]) \\
    &\lesssim
    T^\vareps A^{\rho_2 + 1-\vareps}.
  \end{align*}
\end{proof}

The following result on heat kernel asymptotics is classical
and its proof can be found for example  in
\cite[Theorem 3.10]{bib:driverHeatKernels}
and
\cite[Theorem 2.30]{bib:berlineGetzlerVergne}
See also \cite[Section 3.2]{bib:rosenberg}.
In these references the norm $||\cdot||_{C^\ell(M \times M)}$ is defined via a partition of unity
as in Definition \ref{def:cgamma}.
There is a slight difference to our notation. In the cited references,
$C^1$ for example means ``continuously differentiable'',
while in our notation it only means ``Lipschitz continuous''.
But it is enough to know that our norm is dominated by the norm in the references.
\begin{theorem}
  \label{thm:heatKernelEstimates}
  Let $M$ be a $d$-dimensional, closed Riemannian manifold
  and $\p$ be the heat kernel on $M$.
  Then there exist smooth functions $(\Phi_i(p,q))_{i\ge 0}$ such that if we define
  for $N \ge 1$
  \begin{align*}
    \p^N(t,p,q)
    :=t^{-d/2}    \exp( - d(p,q)^2 / 4t )
    %\psi( d(p,q)^2 )
    \sum_{i=0}^N t^i \Phi_i(p,q),
  \end{align*}
  we have
  \begin{align*}
    ||\partial_t^k (\p_t - \p^N_t)||_{C^\ell(M \times M)} \lesssim t^{N - d/2 - \ell/2 - k}.
  \end{align*}
  Here $\Phi_i(p,q) = 0,$ for $d(p,q) \ge \deltaf$.
  %Here $\psi = 1$ on $B_{\delta/8}(0)$ and $\psi = 0$ outside of $B_{\deltaf}(0)$.
\end{theorem}

\begin{lemma}
  \label{lem:kernelIsScaledFunctionI}
  Let
  \begin{align*}
    \p^N_t(p,q) &=
    t^{-d/2}
    \exp( - d(p,q)^2 / 4t )
    %\psi( d(p,q)^2 )
    \sum_{i=0}^N t^i \Phi_i(p,q),
  \end{align*}
  $\psi, \Phi_i$ smooth and with %$\psi = 1$ on $B_{\delta/8}(0)$ and $\psi = 0$ outside of $B_{\deltaf}(0)$.
  $\Phi_i(p,q) = 0,$ for $d(p,q) \ge \deltaf$.

  Let $p \in M$
  and define for $z$ in the range of $\exp_p^{-1}$,
  $Y_p \in T_p M$ a tangent vector and $Z \in \Gamma(TM)$ a vector field
  \begin{align*}
      \varphi_t(z) &:= \p^N_t(p,\exp_p( z ) ) \\
      \varphi^Y_t(z) &:= X_p \p^N_t(\bullet, \exp_p( z )) \\
      \varphi^{Y,Z}_t(z) &:= Y_p\left[ \ast \mapsto Z_\ast \p^N_t(\bullet, \cdot) \right] |_{\cdot = \exp_p( z )}.
  \end{align*}
  (Note that because of the small support of $\p^N$, these are globally well-defined smooth functions by continuation with zero
  outside of the range of $\exp_p^{-1}$.)
  
  Then for any multiindex $k$, any $n \ge 0$ and $\ell = 0,1$.
  \begin{enumerate}[label=(\roman*)]
    \item 
      \label{item:lemKernelIsScaledFunctioni}
      $|\partial_t^\ell D^k \varphi_t(z)| \lesssim_{\ell,n,k} \left(\sqrt{t}\right)^{-d - k - 2 \ell} \frac{1}{1 + (|z|/\sqrt{t})^n },$

    \item
      \label{item:lemKernelIsScaledFunctionii}
      $\ \ |D^k \varphi^Y_t(z)| \lesssim_{n,k} |Y| \left(\sqrt{t}\right)^{-d - 1 - k} \frac{1}{1 + (|z|/\sqrt{t})^n },$

    \item
      \label{item:lemKernelIsScaledFunctioniii}
      $|D^k \varphi^{Y,Z}_t(z)| \lesssim_{Z,n,k} |Y| \left(\sqrt{t}\right)^{-d - 2 - k} \frac{1}{1 + (|z|/\sqrt{t})^n }.$
  \end{enumerate}
\end{lemma}
\begin{proof}
  The summands of $\p^N$ are of the same form, apart from the factors $t^i$, $i=0,\dots,N$.
  Since for $i \ge 1$ they improve the singularity at $t=0$, it is enough to treat $N=0$.

  Then
  \begin{align*}
    \varphi_t(z) = t^{-d/2} \exp( - |z|^2 / 4t ) \Phi(p,\exp_p(z)).
  \end{align*}
  Since $z \mapsto \Phi(p,\exp_p(z))$ is smooth, uniformly in $p$, with support in $B_1(\delta/8)$,
  and the factor $1/4$ in the exponential is irrelevant, we consider
  \begin{align*}
    \varphi_t(z) = t^{-d/2} \exp( - |z|^2 / t ),
  \end{align*}
  where we abuse notation and keep the same name.
  Now this is the Schwartz function $z \mapsto \exp(-z^2)$ scaled by a factor of $\sqrt{t}$,
  and so part \ref{item:lemKernelIsScaledFunctioni}
  with $\ell=0$ follows from by Remark \ref{rem:rem221ForS}.

  Now
  \begin{align*}
    \partial_t \left[ t^{-d/2} \exp( - |z|^2 / t ) \right]
    &=
    (-d/2) t^{-d/2-1} \exp( - |z|^2 / t )
    +
    t^{-d/2} \exp(-|z|^2/t) |z|^2 t^{-2}.
  \end{align*}
  The first term is treated as above, now having the additional prefactor $t^{-1} = \left( \sqrt{t} \right)^{-2}$.

  We write the second term as
  \begin{align*}
    t^{-d/2 - 1} \exp(-|z|^2/t) \left( \frac{|z|}{\sqrt{t}} \right)^2
    =
    t^{-d/2  - 1} \phi_0^{\sqrt{t}}(z),
  \end{align*}
  where $\phi(s) := s^2 \exp(-s^2)$ is Schwartz.
  By Remark \ref{rem:rem221ForS} part \ref{item:lemKernelIsScaledFunctioni} with $\ell=1$ is proven.

  For the second statement % we also take coordinates in the first argument $x$
  \begin{align*}
    Y_p \p_t(p,q)
    &=
    Y_p\left[ t^{-d/2} \exp( - d(p,q)^2 / 4t ) \Phi(p,q) \right] \\
    &=
    -
    \frac{1}{2} t^{-d/2 - 1} \exp( - d(p,q)^2 / 4t ) Y_p\left[ d^2(p,q) \right] \Phi(p,q)
    +
    \frac{1}{2} t^{-1} \exp( - d(p,q)^2 / 4t ) Y_p\left[ \Phi(p,q) \right].
  \end{align*}
  The first term has worse blowup in $t$ and the factor $1/4$ in the exponential is irrelevant,
  so it is enough to consider
  %Again it is enough to consider $f(q) g(q)$ where
  $f(z) g(z)$ where
  \begin{align*}
    f(z) &:= t^{-d/2 - 1} \exp( - |z|^2 / t ) \\
    g(z) &:= Y_p\left[ d^2(p, \cdot) \right]|_{\cdot=\exp_p(z)}.
  \end{align*}

  Now for a multiindex $k$
  \begin{align*}
    D^k \left[ f(z) g(z) \right]
    =
    \sum_{\beta \le k} c_{\beta,k} D^{k-\beta} f(z) D^\beta g(z).
  \end{align*}

  By Lemma \ref{lem:derivativeOfD2}
  \begin{alignat*}{2}
    |D^\beta g(z)| &\lesssim |z| \quad&  &\text{ if } |\beta| = 0 \\
    |D^\beta g(z)| &\lesssim 1 \quad&  &\text{ else}.
  \end{alignat*}
  and by Lemma \ref{lem:derivativeOfExp}
  \begin{align*}
    |D^{k-\beta} f(z)|
    \lesssim
    t^{-d/2 - 1 -|k-\beta|/2}
    \left(\frac{|z|}{t^{1/2}}\right)^{|k-\beta|} \exp(-|z|^2/t).
    %|f^{k-i}(z)| \lesssim t^{-2 -(k-i)/2} \left(\frac{|z|}{t^{1/2}}\right)^{k-i} \exp(-|z|^2/t).
  \end{align*}

  %Hence for $i=1,\dots,k$
  Hence for $|k-\beta| \le |k| - 1$
  \begin{align*}
    |f^{k-i}(z) g^{(i)}(z)|
    \lesssim t^{-d/2 - 1 -|k-\beta|/2} \left(\frac{|z|}{t^{1/2}}\right)^{k-i} \exp(-|z|^2/t)
    \lesssim t^{-d/2 - 1/2 - |k|/2} \left(\frac{|z|}{t^{1/2}}\right)^{k-i} \exp(-|z|^2/t).
  \end{align*}
  For $|k-\beta| = |k|$ we have $|\beta| = 0$ and then
  \begin{align*}
    |D^{k-\beta} f(z) D^\beta g(z)|
    \lesssim t^{-d/2 - 1 -k/2} \left(\frac{d}{t^{1/2}}\right)^{k} \exp(-|z|^2/t) |z| \\
    \lesssim t^{-d/2 - 1/2 -k/2} \left(\frac{d}{t^{1/2}}\right)^{k+1} \exp(-|z|^2/t).
  \end{align*}
  The second statement then follows, since $s \mapsto s^j \exp(-s^2)$ is a Schwartz function, for any $j \ge 0$.

  %{\color{red}
  %\begin{align*}
  %  \partial^k \left[ f(z) g(z) \right]
  %  =
  %  \sum_{i=0}^n f^{(k-i)}(z) g^{(i)}(z).
  %\end{align*}
  %By Lemma \ref{lem:derivativeOfD2}
  %\begin{align*}
  %  |g^{(0)}(z)| &\lesssim |z| \\
  %  |g^{(i)}(z)| &\lesssim 1, \qquad i\ge 1.
  %\end{align*}
  %and by Lemma \ref{lem:derivativeOfExp}
  %\begin{align*}
  %  |f^{k-i}(z)| \lesssim t^{-2 -(k-i)/2} \left(\frac{|z|}{t^{1/2}}\right)^{k-i} \exp(-|z|^2/t).
  %\end{align*}}
  
  The third statement follows in a similar fashion from
  Lemma \ref{lem:derivativeOfD2}
  and
  Lemma \ref{lem:derivativeOfD2II}.
\end{proof}

\begin{lemma}
  \label{lem:derivativeOfD2}
  Let $Y_p \in T_p M$ acting on the first component of $d^2$ as follows
  \begin{align*}
    g(z) := Y_p\left[ d(p,\cdot)^2 \right]|_{\cdot = \exp_p(z)}.
  \end{align*}

  Then
  \begin{align*}
            |g(z)| &\lesssim |z| |Y| \\
    |D^\beta g(z)| &\lesssim |Y|, \qquad \text{for any multiindex } \beta.
  \end{align*}
\end{lemma}
\begin{proof}
  Since $(p,q) \mapsto d^2(p,q)$ is smooth, we only need to show $g(z) \lesssim |z| |Y|$.

  Let $h(q) = Y_p\left[ d^2(p,q) \right]$.
  Fix $q$ and take coordinates $\exp_q^{-1}$.
  Then
  \begin{align*}
    |h(q)|
    &= |Y^i \partial_{r^i}| d^2(\exp_q(r), q)| \\
    &= |Y^i \partial_{r^i} |r|^2| \\
    &= |Y^i 2 r_i| \\
    &\lesssim |Y| d(p,q).
  \end{align*}
  Then $|g(z)| = |h(\exp_p(z))| \lesssim |Y| |z|$ as desired.
\end{proof}

\begin{lemma}
  \label{lem:derivativeOfD2II}
  Let
  \begin{align*}
    g(z) := Z\left[ Z\left[ d(\bullet,\cdot)^2 \right] \right]|_{\cdot = \exp_p(z)}.
  \end{align*}

  Then, for any multi-index $\beta,$
  \begin{align*}
    |D^\beta g(z)| &\lesssim |Y|, \qquad i \ge 0.
  \end{align*}
\end{lemma}
\begin{proof}
  This follows from the fact that $(p,q) \mapsto d^2(p,q)$ is smooth.
\end{proof}

\begin{lemma}
  \label{lem:derivativeOfExp}
  For any multiindex $k$
  \begin{align*}
    |D^k_z \exp(-|z|^2/t)|
    \lesssim_k
    t^{-|k|/2} \left(\frac{|z|}{t^{1/2}}\right)^{|k|} \exp(-|z|^2/t).
  \end{align*}
\end{lemma}
\begin{proof}
  This can be verified using the Faa di Bruno formula.
\end{proof}

%%  Since $\exp_q( \exp^{-1}_q(p) - \exp^{-1}_q(p) ) = q$,
%%  the expression is (globally) well-defined by continuation with $0$ outside the range of $\exp^{-1}_q$, since $\p^N(q, \exp_q(z)) = 0$ for $|z| > \delta/4$.

%%  Since $|z| > \frac{3}{4} \delta$,
%%  \begin{align*}
%%    d(p,exp_p(z))
%%    \ge |z| - d(p,x)
%%    \ge |z| - \delta/4
%%    \ge \delta/2,
%%  \end{align*}
%%  the expression is (globally) well-defined by continuation with $0$ outside the range of $\exp^{-1}_p$.

\section{Fixpoint argument}
\label{sec:fixpoint}

The following lemma follows from a direct application of the definition of modelled distribution.
\begin{lemma}
  \label{lem:mult}
  Define ``multiplication by $\symbol{\Xi}$'' as the vector bundle morphism
  $m^{\symbol{\Xi}}: \mcW \to \mcV$
  satisfying
  \begin{align*}
    m^{\symbol{\Xi}}( \symbol{1} ) &:= \symbol{\Xi} \\   
    m^{\symbol{\Xi}}( \symbol{\I[\Xi]} ) &:= \symbol{\I[\Xi] \Xi} \\
    m^{\symbol{\Xi}}( \cotang{\omega} ) &:= \cotang{\omega} \symbol{ \Xi } \qquad p \in M, \cotang{\omega} \in T^*_p M.
  \end{align*}
  %Define $\mathcal{D}^{\gamma,\gamma_0}_T(M,\mcG)$ to be the space of functions $f: [0,T] \to C(M, \mcG)$

  If $f \in \mcD^{\gamma,\gamma_0}_T(M,\mcW)$ then $m( f ) \in \mcD^{\gamma,\gamma_0}(M,\mcV)$ and
  for $\normscaling > 0$
  \begin{align*}
    ||m^{\symbol{\Xi}}(f)||_{\mcD^{\gamma,\gamma_0,\normscaling}_T(M,\mcV)}
    =
    ||f||_{\mcD^{\gamma,\gamma_0,\normscaling}_T(M,\mcW)}.
  \end{align*}
\end{lemma}

\begin{theorem}
  Let $u_0 \in C^\infty(\R^2)$.
  Define $v_t := P_t u_0$ and lift it to the regularity structure as
  \begin{align*}
    V_t(p) := \symbol{1} v_t(p) + \symbol{\I[\Xi]} 0 + \cotang{d|_p v_t}.
  \end{align*}
  Let $(\xi,Z)$ be given as in Definition \ref{def:xiz}
  and
  let
  $\Pi^{t,\mcG}_p, \Gamma^{t,\mcG}_{p\leftarrow q}$
  be the corresponding models
  given by Lemma \ref{lem:theseAreInFactModels}, $\mcG = \mcV, \mcW$.
  Let $\alpha \in (-4/3,-1)$,
  $\gamma_0 := \alpha/2 + 1$
  and $\gamma \in (4/3, 2\alpha + 4)$.
  Then there exists $T > 0$ and a unique $u \in \mathcal{D}_T^{\gamma,\gamma_0}(M, \mcW)$ such that on $[0,T]$
  \begin{align*}
    u_t = \mathcal{K}_t \left[ m^{\symbol{\Xi}}( u ) \right] + V_t.
  \end{align*}
\end{theorem}
\begin{proof}
  %This follows immediately from Theorem \ref{thm:schauder}.
  We follow a standard fixpoint argument.
  Denote
  \begin{align*}
    B(R,\normscaling)
    :=
    \{ f \in \mcD_T^{\gamma,\gamma_0}(M, \mcW)
       :
       ||f - V||_{\mcD_T^{\gamma,\gamma_0,\normscaling}(M, \mcW)} \le R
    \}.
  \end{align*}
  Denote for $f \in B(R,\normscaling)$
  \begin{align*}
    \Phi(f) := \mathcal{K}_t \left[ m^{\symbol{\Xi}}( f ) \right] + V_t.
  \end{align*}
  Claim: for any $\normscaling > 0,$ there is $R > 0$ such that $\Phi( B(R,\normscaling) ) \subset B(R,\normscaling)$.

  Indeed, by Theorem \ref{thm:schauder}
  and Lemma \ref{lem:mult}, for a constant $c > 0$ possibly changing from line to line,
  \begin{align*}
    ||\Phi(f) - V||_{\mcD_T^{\gamma,\gamma_0,\normscaling}(M, \mcW)}
    &=
    ||\mathcal{K}_t \left[ m^{\symbol{\Xi}}( f ) \right]||_{\mcD_T^{\gamma,\gamma_0,\normscaling}(M, \mcW)} \\
    &\le
    c
    ||m^{\symbol{\Xi}}(f)||_{\mcD^{\gamma - 4/3, \gamma_0, \normscaling}_T(M, \mcV)}
    \left( T^\vareps + T^\vareps \normscaling + \frac{1}{\normscaling} \right) \\
    &=
    c
    ||f||_{\mcD^{\gamma - 4/3 - \alpha, \gamma_0, \normscaling}_T(M, \mcV)}
    \left( T^\vareps + T^\vareps \normscaling + \frac{1}{\normscaling} \right) \\
    &\le
    c
    ||f||_{\mcD^{\gamma, \gamma_0, \normscaling}_T(M, \mcV)}
    \left( T^\vareps + T^\vareps \normscaling + \frac{1}{\normscaling} \right),
  \end{align*}
  since $\alpha > - 4/3$.
  Hence for $T$ small enough and $\normscaling$ large enough, $\Phi( B(R,\normscaling) ) \subset B(R,\normscaling)$,
  for any $R > 0$.

Let us  show that $\Phi$ is a contraction on $B(R,\normscaling)$: for any $f,f'\in B(R,\normscaling),$
  \begin{align*}
    ||\Phi(f) - \Phi(f')||_{\mcD_T^{\gamma,\gamma_0,\normscaling}(M, \mcW)}
    &=
    ||\mathcal{K}_t \left[ m^{\symbol{\Xi}}( f - f' ) \right]||_{\mcD_T^{\gamma,\gamma_0,\normscaling}(M, \mcW)} \\
    &\le
    c
    ||m^{\symbol{\Xi}}(f - f')||_{\mcD^{\gamma - 4/3, \gamma_0, \normscaling}_T(M, \mcV)}
    \left( T^\vareps + T^\vareps \normscaling + \frac{1}{\normscaling} \right) \\
    &=
    c
    ||f - f'||_{\mcD^{\gamma - 4/3 - \alpha, \gamma_0, \normscaling}_T(M, \mcV)}
    \left( T^\vareps + T^\vareps \normscaling + \frac{1}{\normscaling} \right) \\
    &\le
    c
    ||f - f'||_{\mcD^{\gamma, \gamma_0, \normscaling}_T(M, \mcV)}
    \left( T^\vareps + T^\vareps \normscaling + \frac{1}{\normscaling} \right).
  \end{align*}
  Hence for $T$ small enough and $\normscaling$ large enough, $\Phi$ is a contraction on $B(R,\normscaling)$
  for any $R > 0$.

  We therefore get unique existence of a solution for small $T > 0$.
\end{proof}

\section{Appendix - The Gaussian model}
\label{sec:gaussian}

%{\color{red}
%Goal: find $\xi$, $Z^t_p$, $p \in M$ such that
%\renewcommand{\theenumi}{\Alph{enumi}}
%\begin{enumerate}%[label=(\roman*)]
%  \item $\xi \in C^\alpha(M)$
%  \item \label{item:B} $\langle Z^t_p, \varphi^\lambda_p \rangle \lesssim \lambda^{2\alpha + 2}$
%  \item \label{item:C} $Z^t_q(\cdot) = Z^t_p(\cdot) + \int_0^t \left[ \p_{t-r}(p,\cdot) - \p_{t-r}(q,\cdot) \right] \ast \xi dr \xi(\cdot)$,
%\end{enumerate}
%with some uniformity in $t \le T$, $\xi = $ white noise and $Z^t_p "=" \int_0^t \xi P_s \xi ds$.
%}

%We want to endow the above  space $\mathcal{T}$ with a transport map $\Gamma$ and a model $\Pi$ at precision $\beta,$  with  $\beta\in (1,2),$  such that for any $x\in M,t\in \R_+,$ $\Pi_x^t(\Xi)=\xi $, where $\xi$ is a white noise on $M.$
\newcommand{\vol}{\operatorname{vol}}
Let $\xi$ be white noise on $M$.
We recall that $\xi$ is a Gaussian process associated to the Hilbert space $L^2(M,  \vol_M ),$ on a probability space $(\Omega,\mathcal{B},\mathbb{P}).$ 
\begin{lemma} There exists a realization of $\xi$  such that almost surely for any $\alpha<-1,$  $\xi\in \mathcal{C}^{\alpha}$. \label{reg WN}
\end{lemma} 
\begin{proof}
  For any coordinate chart $\psi$ defined on an open subset $\U\subset M,$ and $\rho$ a positive function with support in $\U,$  $\xi_\U=\rho\circ\psi^{-1} \psi_*\xi$   is a Gaussian process   associated to the Hilbert space $L^2(\R^2, \rho^2\circ\psi^{-1} \det(g\circ\psi^{-1} )   ).$
  Note that $\xi_\U$ has the same law as
  $\eta \nu$, with $\eta := \rho \circ \psi^{-1} \sqrt{ g \circ \psi^{-1} }$
  and $\nu$ a white-noise on $\R^d$.
  According to
  \cite[Lemma 10.2]{bib:hairer}
  $\nu$
  has a version which is almost surely in $C^\alpha(\R^d)$ and hence $\xi_\U \in C^\alpha(\R^d)$.

  Let now $(\rho_i)_{1\le i\le n}$ be a partition a unity subordinated to an atlas $(\U_i,\psi_i)_{1\le i\le n}.$
  Then, there is a realization of $(\xi_{\U_i})_{1\le i\le n}$ such that almost surely for all $\alpha<-1, i\in \{1,\ldots, n\},$  $\xi_{\U_i}\in \mathcal{C}^\alpha(\R^2).$  Then,  $\sum_{i=1}^n \psi_i^*\xi_{\U_i}$ is a realization of $\xi$ belonging almost surely to $\mathcal{C}^\alpha(M)$.
\end{proof}
Thanks to this realization, we can already define the transport map used in the following Lemma (point \ref{item:II}).

\begin{lemma}
  \label{Lemma gaussian models}
  Let
  $\xi$ be the white noise on $M$
  and
  $Z^t_p$, $p \in M, t \in [0,T]$ be a collection of random distributions on $M$
  such that
  for some $\alpha \in (-4/3,1)$, some $\kappa > 0$,
  \begin{enumerate}[label=(\roman*)]
    %\item
    %  \label{item:I}
    %  $\xi \in C^\alpha(M)$
    \item
      \label{item:II}
      $Z^t_q(\cdot) = Z^t_p(\cdot) + \int_0^t \Big\langle \p_{t-r}(p,\cdot) - \p_{t-r}(q,\cdot), \xi \Big\rangle dr \xi(\cdot)$,
    \item
    \label{item:III}
      \begin{equation}
      \sup_{p\in M, 0\le t,s\le T,  \lambda>0,\varphi \in \mathcal{B}^{\lceil |\alpha|\rceil,\delta}}
        \lambda^{-2(2\alpha+2)-\kappa}\mathbb{E}[ \langle{Z^0_p,\varphi^\lambda_p}\rangle^2+|t-s|^{-\kappa}\langle{Z^t_p-Z^s_p,\varphi^\lambda_p}\rangle^2]
        < \infty,
    \end{equation}
  \item for any $\varphi \in C^\infty(M),t \in [0,T], p \in M$
    \label{item:IV}
    $\langle{Z^t_p,\varphi}\rangle$ is in the second Wiener chaos.
  \end{enumerate}

  Then, there is a version of $Z$ and a constant $h>0$ such that a.s.
  \begin{align}
    \label{eq:Z}
    \sup_{p\in M, 0\le t,s\le T,  \lambda>0,\varphi \in \mathcal{B}^{\lceil |\alpha|\rceil,\delta}} \lambda^{-(2\alpha+2) } \left(|\langle{Z^0_p, \varphi_p^\lambda}\rangle|+ |t-s|^{-h} |\langle{Z^t_p-Z^s_p, \varphi_p^\lambda}\rangle|\right)<\infty.
  \end{align}
%  \begin{itemize}
%    \item 
%      $\langle Z^t_p, \varphi^\lambda_p \rangle \lesssim \lambda^{2\alpha + 2}$.
%      \item for any $p\in M,\varphi\in C^\infty(M),$ $t\in[0,T]\mapsto\langle{Z^t_p},\varphi\rangle$ is continuous.
%  \end{itemize}
\end{lemma}
\begin{proof}
  For $t>s\ge 0,$ define for a chart $(\Psi,\U)$
  \begin{align*}
    \bar Z^{s,t}_x &:= \Psi_* (Z^t_{\Psi^{-1}(x)}-Z^s_{\Psi^{-1}(x)}), \qquad x \in \Psi(\U) \\
    \bar \xi &:= \Psi_* \xi.
  \end{align*}

  Note that $\bar Z^{s,t}_x, x \in \Psi(\U)$ and $\bar \xi$ are elements of $D'( \Psi(\U) )$.
  %Let $G^{s,t}(p \leftarrow q) := \int_s^t \left[ \p_{t-r}(p,\cdot) - \p_{t-r}(q,\cdot) \right]dr$.
  Then
  \begin{align}
    \Big\langle \bar Z^{s,t}_y, \varphi \Big\rangle
    &= \Big\langle Z^{s,t}_{\Psi^{-1}(y)}, \varphi \circ \Psi \Big\rangle \notag \\
    &= \Big\langle Z^{s,t}_{\Psi^{-1}(x)}
    +
    \langle \int_s^t \left[ \p_{t-r}(\Psi^{-1}(x),\cdot) - \p_{t-r}(\Psi^{-1}(y),\cdot) \right]dr, \xi \rangle \xi, \varphi \circ \Psi \Big\rangle \notag \\
    &= \Big\langle \bar Z^{s,t}_{x}
    +
    \langle \int_s^t \left[ \p_{t-r}(\Psi^{-1}(x),\cdot) - \p_{t-r}(\Psi^{-1}(y),\cdot) \right]dr, \xi \rangle \bar \xi, \varphi \Big\rangle \notag \\
    &= \Big\langle \bar Z^{s,t}_{x} + S^{s,t}( x \leftarrow y ) \bar \xi, \varphi \Big\rangle, \label{eq:reexpression}
  \end{align}
  where we denote $S^{s,t}( x \leftarrow y ) := \langle \int_s^t \left[ \p_{t-r}(\Psi^{-1}(x),\cdot) - \p_{t-r}(\Psi^{-1}(y),\cdot) \right]dr, \xi \rangle$.

  Define the regularity structure and model (in the stronger sense of \cite{bib:hairer})
  \begin{align*}
    \mcT &:= \spann \{ \symbol{\Xi} \} \oplus \spann\{ \symbol{\I[\Xi] \Xi} \} \oplus \spann\{ \symbol{\1} \} \\
    \Pi^{s,t}_x \symbol{\Xi} &:= \bar \xi \\
    \Pi^{s,t}_x \symbol{\I[\Xi]\Xi]} &:= \bar Z^{s,t}_x \\
    \Pi^{s,t}_x \symbol{\1} &:= 1 \\
    \Gamma^{s,t}_{x \leftarrow y} \symbol{\Xi} &:= \symbol{\Xi} \\
    \Gamma^{s,t}_{x \leftarrow y} \symbol{\I[\Xi] \Xi} &:= \symbol{\I[\Xi] \Xi} + S^{s,t}(x\leftarrow y) \symbol{\Xi} \\
    \Gamma^{s,t}_{x \leftarrow y} \symbol{\1} &:= \symbol{\1},
  \end{align*}
  and the sector (in the sense of \cite[Definition 2.5]{bib:hairer})
  \begin{align*}
    V := \spann \{ \symbol{\Xi} \} \oplus \spann\{ \symbol{\I[\Xi] \Xi} \}.
  \end{align*}
  One can then
  apply \cite[Proposition 3.32]{bib:hairer} to get 
  for every compacta $K_- \subset\subset K \subset \Psi(U)$,
  and\footnote{in the notation of  \cite[Proposition 3.32]{bib:hairer}, $\varphi^n_z$ stands for $2^{-nd/2}\varphi_z^{2^{-n}}$} $\varphi \in \balloftestfunctions^{r,1}$, $r := \lceil |\alpha| \rceil$,
  with $\supp \varphi^\lambda_x \subset K_-,$
  \begin{align}
    |\langle \bar Z^{s,t}_x, \varphi^\lambda_x \rangle|
    &\lesssim
    \lambda^{2\alpha + 2} ||\Pi^{s,t}, \Gamma^{s,t}||_{V, \K} \notag \\
    &\lesssim
    \lambda^{2\alpha + 2}
    \left(
      1
      +
      ||\Gamma^{s,t}||_{V, \K}
    \right) \notag \\
    &\quad
    \times
    \sup_{ a = \alpha, 2\alpha + 2}
    \sup_{ \tau \in V_a }
    \sup_{n \ge 0}
    \sup_{z \in \text{n-dyadics} \cap \K}
      \lambda^{(2\alpha + 2)n }
      \frac{\langle \Pi^{s,t}_z \tau, \varphi^n_z \rangle}{|\tau|} \notag \\
    &\lesssim
    \lambda^{2\alpha + 2}
    \left(
      1
      +
      \sup_{x,y \in \R^d}
       |x-y|^{-(\alpha + 2)}
       S^{s,t}( x \leftarrow y )
    \right) \notag \\
    &\quad
    \times
    \sup_{n \ge 0}
    \sup_{z \in \text{n-dyadics} \cap \K}
    \left(
      2^{(2\alpha + 2)n }
      \langle \bar Z^{s,t}_z, \varphi^n_z \rangle
      +
     2^{\alpha n }
      \langle \bar Z^{s,t}_z, \varphi^n_z \rangle
      \langle \bar \xi, \varphi^n_z \rangle
    \right).
    \label{eq:above}
  \end{align}

  Then, for $q\in \N^*$ large enough, using \ref{item:IV}, equivalence of moments and then \ref{item:II}
\begin{align}
 &\mathbb{E} [\sup_{n \ge 0}
    \sup_{|z|< \delta, z \in \text{n-dyadics}}
    \left(
     2^{(2\alpha + 2)nq }
      |\langle \bar Z^{s,t}_z, \varphi^{2^{-n}}_z \rangle|^q
      +
      2^{\alpha n q}
     | \langle \bar Z^{s,t}_z, \varphi^{2^{-n}}_z \rangle|^q
     | \langle \bar \xi, \varphi^{2^{-n}}_z \rangle|^q
    \right)] \notag \\
    &\lesssim  \sum_{n\ge0}  2^{2 n}  \left(
     2^{(2\alpha + 2)nq }
        \sup_{|z|<2\delta} \mathbb{E}[|\langle \bar Z^{s,t}_z, \varphi^{2^{-n}}_z \rangle|^2]^{\frac{q}{2}}
      +
      2^{\alpha n q}
    \sup_{|z|<2\delta} \mathbb{E}[| \langle \bar Z^{s,t}_z, \varphi^{2^{-n}}_z \rangle|^2
     | \langle \bar \xi, \varphi^{2^{-n}}_z \rangle|^2]^{\frac{1}{2q}}
    \right)] \notag \\
    &\lesssim |t-s|^{\frac{q\kappa}{2}}.
    \label{eq:expectation}
\end{align}
%for some constant $K>0.$  

  %The result of this is, that $\bar Z^t_x, x \in \Psi(\U)$ has the right homogeneity (on $\Psi(\U)$).
  %For test functions of radius of order $\le \delta/2$ we hence know that $Z_p$, $p \in \tilde \U$
  %has the right homogeneity
  %{\color{red}(that is why we chose $d(\tilde \U, \partial U) \ge \delta/2$)}.

  %Now do the above with a finite atlas $(\Psi_i, U_i)$ with the following properties
  %\begin{itemize}
  %  \item there are $\tilde \U_i \subset \U_i$ with $d(\tilde \U_i, \partial \U_i) \ge \delta/2$
  %  \item $\bigcup_i \tilde \U_i = M$
  %\end{itemize}
  Let now $(\Psi_i, \U_i)$ be a finite atlas
  with subordinate partition of unity $\phi_i$.
  Then for $s,t \in [0,T]$, $p \in M$, $\varphi \in \balloftestfunctions^{r,\delta}$
  \begin{align*}
    \langle Z^{s,t}_p, \varphi^\lambda_p \rangle
    &=
    \sum_{i : \phi_i \varphi^\lambda_p \not= 0}
      \langle Z^{s,t}_p, \phi_i \varphi^\lambda_p \rangle,
  \end{align*}
  Now for $\lambda$ small enough,
  $\phi_i \varphi_p^\lambda \not\equiv 0$ implies that $\supp \varphi^\lambda_p \subset \U_i$
  and in particular $p \in \U_i$.
  Hence
  \begin{align*}
    \langle Z^{s,t}_p, \varphi^\lambda_p \rangle
    &=
    \sum_{i : \phi_i \varphi^\lambda_p \not= 0}
      \langle \bar Z^{s,t;i}_{\Psi^i(p)}, \left( \phi_i \varphi^\lambda_p \right) \circ \Psi_i^{-1} \rangle,
  \end{align*}
  where $\bar Z^{s,t;i}_x := (\Psi_i)_* (Z^t_{\Psi_i^{-1}(x)}-Z^s_{\Psi_i^{-1}(x)})$.
  We can apply Remark \ref{rem:221} to $\left( \phi_i \varphi^\lambda_p \right) \circ \Psi_i^{-1}$
  and can estimate, using \eqref{eq:above},
  \begin{align*}
    \langle Z^{s,t}_p, \varphi^\lambda_p \rangle
    &\lesssim
    \lambda^{2\alpha + 2}
    \left(
      1
      +
      \sup_{x,y \in \R^d}
       |x-y|^{-(\alpha + 2)}
       S^{s,t}( x \leftarrow y )
    \right) \notag \\
    &\quad
    \times
    \sum_i 
    \sup_{n \ge 0}
    \sup_{z \in \text{n-dyadics} \cap \K_i}
    \left(
      2^{(2\alpha + 2)n }
      \langle \bar Z^{s,t;i}_z, \varphi^{2^{-n}}_z \rangle
      +
     2^{\alpha n }
      \langle \bar Z^{s,t;i}_z, \varphi^{2^{-n}}_z \rangle
      \langle \bar \xi^i, \varphi^n_z \rangle
    \right),
  \end{align*}
  here for every $i$, $\K_i$ is some compactum satisfying $\Psi_i(\supp \phi_i) \subset\subset \K_i \subset \Psi_i(\U_i)$.
  Then, by \eqref{eq:expectation},
  \begin{align*}
    \E[ \sup_{p\in M, \varphi \in \balloftestfunctions^{r,\delta}} |\langle Z^{s,t}_p, \varphi^\lambda_p \rangle|^q ]
    \lesssim
    |t-s|^{q\kappa/2}.
  \end{align*}
  
  Let us formulate a setting where we can apply Kolmogorov's continuity theorem in time.
  Endow the linear space of maps $Y: M\to \mathcal{D}'(M)$, such that for any $p\in M,$  $\supp(Y_p)\subset B(p,\delta)$, 
  with the norm
  %Endow the vector space of linear maps $ \mathcal{L}(\mathcal{T}, \mathcal{D}'(M))$ with the norm setting for any $\Pi\in \mathcal{L}(\mathcal{T},\mathcal{D}'(M)),$
  %$$\|\Pi\|= \sup_{p\in M,\tau\in \mathcal{T}_p, \|\tau\|\le 1, \lambda>0, \varphi \in \mathcal{B}^{\lceil |\alpha| \rceil ,\delta}} \lambda^{|\tau|} \langle{\Pi(\tau), \varphi_p^{\lambda}}\rangle.$$
  \begin{align*}
    ||Y|| := \sup_{p\in M, \lambda\in (0,1), \varphi \in \balloftestfunctions^{r,\delta}} \lambda^{-(2\alpha+2)} |\langle Y_p, \varphi^\lambda_p \rangle|,
  \end{align*}
  and consider the Banach space $\chi=\{Y: ||Y||<\infty\}.$
  We apply this to
  \begin{align*}
    Y_p^t := Z_p^t \rho_p.
  \end{align*}
  Here $\rho_p := \rho \circ \exp^{-1}_p$, with $\rho$ smooth, $\supp \rho
  \subset B_\delta(0)$ and $\rho \equiv 1$ on $B_{\delta - \epsilon}(0)$ for
  some $\epsilon > 0$ small enough.
  Then, from the argument before, for any $s,t\ge 0$ and $q$ large enough, we have
  $$\mathbb{E}\|Y^t-Y^s\|^q\lesssim |t-s|^{\frac{\kappa q}{2}}.$$
  The result now follows from the Kolmogorov continuity theorem.
\end{proof}

\newcommand{\q}{\mathsf{q}} % was q before .. but that is a point on M now

A simple way  to define $Z^t_p$ is to consider the Wick product of the random variables involved.   For any $t> 0,$ the heat kernel  and the heat operator are  denoted respectively by $\p_t:M^2\to \R$ and $P_t,$ and we write  for any $p,q\in M,$ $\q_t(p)=\p_t(p,p).$
 According to Lemma  
 \ref{reg WN}  and
 %\cite[Theorem 3.10]{bib:driverHeatKernels}
 Theorem \ref{thm:heatKernelEstimates}, we can  consider $P_t(\xi)$ as a function and the map $t\in \R_{> 0} \mapsto P_t(\xi) \in C^{\infty}(M) $  is continuous.

 We set for any $p\in M, t\in \R_{\ge 0}$ and  any  function $\varphi\in C^{\infty}(M),$

$$Z^t_p := \int_0^t ( \xi \diamond  P_{s}(\xi) - \xi P_s(\xi)(p)) ds,$$
%$$\begin{array}{|c|c|c|c|c|}\hline \tau &\Xi &  \omega   \Xi  & \I(\Xi)   & \Xi \I(\Xi)\\  
%\hline\Pi_x^t(\tau)&\xi &   \omega\circ \exp_x^{-1} \xi    &  \int_0^t ( P_{s}(\xi)  -P_s(\xi)(x))ds   &  \int_0^t ( \xi \diamond  P_{s}(\xi) - \xi P_s(\xi)(x)) ds    \\\hline
%\end{array}$$
where for any $s>0$ and $\varphi\in C^{\infty}(M),$ 
$$\langle{ \xi \diamond  P_{s}(\xi) , \varphi}\rangle =  \langle{\xi,  \varphi P_s (\xi)}\rangle -\mathbb{E}[\langle{\xi,  \varphi P_s (\xi)}\rangle].$$
%and for any $y\in M,$
%$$\Gamma_{y\leftarrow x} (\Xi)=\Xi, \hspace{1 cm}  \Gamma_{y\leftarrow x} (\I(\Xi)\Xi) = \I(\Xi) \Xi+ \int_0^t   (P_s(\xi)(y)- P_s(\xi)(x)) ds  \Xi, $$
%$$\Gamma_{y\leftarrow x} (\omega \Xi)= \omega(\exp_x^{-1}(y)) \Xi + \omega\circ T_y(\exp_x^{-1}) \Xi, \hspace{1 cm}  \Gamma_{y\leftarrow x} (1)=1, $$
%$$  \Gamma_{y\leftarrow x} (\I(\Xi))= \I(\Xi)+ \int_0^t (  P_s(\xi)(y)-P_s(\xi)(x)) ds,$$
%and 
%$$\Gamma_{y\leftarrow x} (\omega)= \omega(\exp_x^{-1}(y))+ \omega\circ T_y(\exp^{-1}_x).$$
Note that for any $s>0,$
$$\xi\diamond P_s\xi = P_s(\xi)\xi - \q_s.$$

For any $t\ge 0,$ let us consider the operator $K_t= \int_0^t P_s ds$ and  for any $p,q\in M$ with $p\not=q,$ set  $\mathsf{k}_t(p,q)=\int_0^t \p_{s}(p,q)ds .$ 
Let us note that the operator
\begin{equation}
K_t^2= \int_{0\le s,s' \le t} P_{s+s'} ds ds'= \int_0^{2t} sP_sds, \label{square HES}
\end{equation}
has a continuous kernel according to Theorem \ref{thm:heatKernelEstimates}, that we shall denote $\mathsf{k}_{2,t}$.
\begin{proposition}
  For any  $t\in \R_{\ge 0},$ almost surely  for any $ p \in M$ and $\varphi\in C^{\infty}(M),$    $\langle Z^t_p,\varphi \rangle$ is well-defined and there exists a 
  modification of the process given by $(\langle Z^t_p,\varphi\rangle)_{p\in M, \varphi \in C^{\infty}(M),   \tau\in \mathcal{G}, t\ge 0 }$ such that almost surely
  \eqref{eq:Z} holds true.
  %$\langle Z^t_p, \varphi^\lambda_p \rangle \lesssim \lambda^{2\alpha + 2}$.
  %\ref{item:B} and \ref{item:C} hold.
  %$(\Pi,\Gamma)$  is a regularity structure
  % with transport at precision $\beta.$
  \footnote{In particular, almost surely,  for all   $\varphi\in C^{\infty}(M) $ and $p\in M,$ $ t\mapsto \langle{\Pi^t_p(\tau),\varphi} \rangle$ is measurable and  bounded.}
  \label{Gaussian Model}
\end{proposition}
\begin{proof}[Proof of  Proposition \ref{Gaussian Model}]
  It is enough to prove the assumption of Lemma \ref{Lemma gaussian models}.  Let us fix $p \in M, 0<r<t.$ 
  Recall that $\delta$ is the radius of injectivity of $M$
  and let $r := 1$.

  Let us first check that for any $\varphi\in C^{\infty}(M),$  $Z^t_x(\varphi),$ is well defined.
  Therefor, let us recall -- see Theorem \ref{thm:heatKernelEstimates} -- that
  \begin{equation}
    \label{log HK}
    L := \sup_{p\in M,r\in (0,t]} r \q_{r}(p) < \infty.
  \end{equation}
  \newcommand{\var}{\operatorname{Var}}
  The  Wick formulas imply for any $s>0,$ 
  \begin{align*}
    \var( \langle{\xi, \varphi P_s\xi }\rangle)
    =
    \int \q_{2s}(z)\varphi (z)^2 dz
    + \int \p_s(z,z')^2 \varphi(z)\varphi(z')dz dz'\le  2\vol(M) L  \|\varphi\|^2_\infty s^{-1}.
  \end{align*}
  It follows that 
  $$ \mathbb{E}[\int_0^t  |\langle{P_s\xi\diamond \xi, \varphi}\rangle |ds]\le  \int_0^t \var( \langle{\xi, \varphi P_s\xi }\rangle)^{1/2}ds <\infty. $$ 
  Besides,   $\mathbb{E} \left(\langle{\xi, \varphi }\rangle^2 P_s\xi(p)^2\right)= q_{2s}(p)\|\varphi\|_2^2+  P_s(\varphi)(p)^2\le s (L+1) \|\varphi\|_\infty^2,$ so that $Z^t_p(\varphi)$ is well defined.     We  shall  
  now prove that   for any $\kappa<0, $
  \begin{equation}
    \label{estimate moments}
    \sup_{t\in [0,T], p\in M, \varphi \in \balloftestfunctions^{r,\delta}}
    \lambda^{-2 \kappa}\mathbb{E}[ \left( Z^t_p(\varphi^\lambda_p) \right))^2]
    < \infty,
  \end{equation}
  which together with Lemma \ref{reg WN} shall  yield the claim. We   fix now
  $\kappa<0.$ Let us first prove that the expectation of the second integrand in
  $\Pi_p^t(\I(\Xi)\Xi)$ is almost surely of homogeneity $\kappa$.
  %so that we can work with a centered random distribution.
  Indeed, according to \cite{bib:driverHeatKernels},  for $T>0$ fixed,  there exists $C_T>0$, such that, for all $0, t < T, p,q\in M,$ with $p\not= q,$
  \begin{align}
    |\mathsf{k}_t(p,q)| &\lesssim  C_T+ \int_0^T e^{-\frac{d(p,q)^2}{4s}}  \frac{ds}{s}= C_T+\int_{\frac{d(p,q)^2}{4T}}^{+\infty}e^{-v}\frac{dv}{v}  \nonumber   \\ 
    &\lesssim |\log ( d(p,q)) |. \end{align}
  Since
  $$ \mathbb{E}[\langle{\xi P_s \xi (p), \varphi}\rangle]= P_s \varphi (p),  $$
  it follows that for any $\kappa<0,$
  $$\sup_{p\in M, \varphi \in \balloftestfunctions^{r,\delta}}  |\mathbb{E}[\langle{\xi K_t \xi (p), \varphi^\lambda_p}\rangle]|  = \sup_{p\in M, \varphi \in \balloftestfunctions^{r,\lambda}}  |K_t(\varphi^\lambda_p)(p)|\le C_T \lambda^{\kappa}.$$
  Setting
  $I_{p,s}= \xi P_s \xi - \xi  P_s \xi (p)$ and $:I_{p,s}:\ = I_{p,s} -\mathbb{E}[I_{p,s}],$%
  \footnote{
    Where me denote the distribution $\langle \E[ I_{p,s} ], \varphi \rangle := \E[ \langle I_{p,s}, \varphi \rangle ]$.
  }
  it remains to estimate 
  \begin{align*}
    \langle : Z^t_p :, \varphi^\lambda_p \rangle
    :=
    \langle Z^t_p, \varphi^\lambda_p \rangle
    -
    \E[ \langle Z^t_p, \varphi^\lambda_p \rangle ]
    =
    \int_0^t \langle :I_{p,s}: , \varphi^\lambda_p \rangle ds
  \end{align*}
  %$$ \langle{:\Pi(\I(\Xi)\Xi)_x:,\varphi}\rangle:=
  %\langle{\Pi(\I(\Xi)\Xi)_p,\varphi}\rangle-\mathbb{E}[\langle{\Pi(\I(\Xi)\Xi)_p,\varphi}\rangle]=
  %\int_0^t \langle{:Z_{p,s}:,\varphi }\rangle ds.$$
  For any $\varphi\in \balloftestfunctions^{r,\delta}, s,s' \ge 0,$
  \begin{align*}
    \mathbb{E}[\langle :I_{p,s}:,\varphi^\lambda_p \rangle \langle
    :I_{p,s'}:,\varphi^\lambda_p \rangle] =  & \int_M  (\p_{s+s'}(q,q)
    +\p_{s+s'}(p,p)-2\p_{s+s'}(q,p)) \varphi^\lambda_p (q)^2dq\\ &+\int_{M^2} (\p_s(q,q') -
    \p_s(p,q'))(\p_{s'}(q,q')-\p_{s'}(p,q)) \varphi^\lambda_p (q)\varphi^\lambda_p (q')dqdq'.
  \end{align*}
  and 
  \begin{align*}
    \mathbb{E}[ \langle{:Z^t_p:,\varphi^\lambda_p }\rangle^2]
    =  & \int_M  (\mathsf{k}_{2,t}(q,q) +\mathsf{k}_{2,t}(p,p)-2\mathsf{k}_{2,t}(q,p)) \varphi^\lambda_p (q)^2dq \\
    &+\int_{M^2} (\mathsf{k}_t(q,q') - \mathsf{k}_t(p,q'))(\mathsf{k}_t(q,q')-\mathsf{k}_t(p,q))\varphi^\lambda_p (q)\varphi^\lambda_p (q')dqdq'\\
    &\le
    2  \int_M  (\mathsf{k}_{2,t}(q,q) +\mathsf{k}_{2,t}(p,p)-2\mathsf{k}_{2,t}(q,p)) \varphi^\lambda_p (q)^2dq,
  \end{align*}
  where the second line follows from the Cauchy-Schwarz inequality.    
  It follows from the bound  \eqref{Taylor Kernel second order} that there exists
  $C>0$, such that for any $\varphi \in \balloftestfunctions^{r,\delta}, t \in [0,T],$
  \begin{align*}
    \mathbb{E}[\langle :\Pi(\Xi \I(\Xi))^t_x:,\varphi^\lambda_p \rangle^2]
    \le  C \int_{M}&d(p,q)^{2-\delta}|\varphi^\lambda_p(q)|^2 dq.
  \end{align*}
  Hence for any $\lambda>0$ and $\varphi \in \balloftestfunctions^{r,\delta}$,
  $$\mathbb{E}[\langle :Z^t_p:,\varphi^\lambda_p \rangle^2] \le C \lambda^{-\delta}.$$
\end{proof}

\begin{lemma}
  For any $\nu>\eta>0,T\ge 0$, there exists $C>0,$ such that for any $q\in M, t\in [0,T],$
  \begin{align}
    \label{Taylor Kernel second order}
    |\mathsf{k}_{2,t}(q,q) + \mathsf{k}_{2,t}(p,p)-2\mathsf{k}_{2,t}(q,p)|\le C  t^\eta d(p,q)^{2-2\nu}.
  \end{align}
\end{lemma}
\begin{proof}  
  On the one hand, according to (\ref{square HES}) and  Theorem \ref{thm:heatKernelEstimates}, the left-hand-side of (\ref{Taylor Kernel second order}) is uniformly bounded by $C_T t,$ for all $t\in [0,T],$ for some $C_T>0.$   On the other hand, the estimate (\ref{Taylor Kernel second order}) would hold true, with $\eta=0,$ if $K_t$ 
  would be replaced by  $C^2$ symmetric function on $M^2$.  Indeed if $K:M^2\to \R$ is a  $C^2$ symmetric function,
  \begin{align}
    |K(q,q) + K(p,p)-2K(q,p)|\le  \int_{0\le r,s\le 1}  \|  \nabla _{2, \dot{\gamma}_s}\nabla_{1,\dot{\gamma}_r}  K (\gamma_s,\gamma_r)\|_\infty  drds ,  \label{Proof Lem Taylor along  geo}
  \end{align} 
  where the index below the connexion symbol indicate the variable on which the latter is acting, and $\gamma$ is a geodesic from $p$ to $q$.
  According to Theorem \ref{thm:heatKernelEstimates}, one can  therefore consider  $K^{2,N}_{t}= \int_0^t s P_s^N ds $  in place of $K^2_t,$ as soon as $N$ is large enough.
  This same Theorem ensures that  there exists a smooth 
  function $\Phi:[0,T]\times M^2\to \R_{\ge 0}$ such that  for all $\tau\in (0,T], p,q\in M,$
  $$\p^N_\tau(p,q) =  (2\pi \tau)^{-1} e^{-\frac{d(p,q)^2}{2\tau}} \Phi(\tau,p,q).$$
  Let  us set $\q_\tau(r)= \frac{1}{2\pi \tau } e^{-\frac{r^2}{\tau}},$ for any $r,\tau>0.$  We shall apply  \eqref{Proof Lem Taylor along  geo}  to $K_{t,\varepsilon}= \int_{\varepsilon}^t sP^N_s ds,$ for any fixed $\varepsilon>0.$  Up to 
  a constant,  the integrand of  the right-hand-side of (\ref{Proof Lem Taylor along  geo}) is   bounded  by  
  $ \int_\varepsilon ^t\left(d(p,q)\|\nabla_{1,\dot{\gamma}_s}  \q_\tau \circ d (\gamma_s,\gamma_r) \| +\|\nabla_{1,\dot{\gamma}_s} \nabla_{2,\dot{\gamma}_r}  \q_\tau\circ d (\gamma_s ,\gamma_r)\|\right) d\tau. $  Let 
  us  set $R=d(p,q).$  The first term can be bounded by 
  $$  \begin{aligned}
  R^2  \int_{\varepsilon}^t  \tau ^{-1}| s-r| e^{-\frac{ R^2 |s-r|^2}{2\tau}} d\tau &\le R^2 |s-r| \int_{\frac{R^2}{t}|r-s|^2}^{\frac{R^2}{\varepsilon} |r-s|^2}e^{-u/2} \frac{du}{u}\\
  &\le C_T  R^2 |  s-r| \log \frac{T}{|s-r|R^2},
  \end{aligned} $$
  and the second by 
  $$ \begin{aligned}
  R^2 \int_\varepsilon^t  (\tau^{-1}+\frac{R^2(r-s)^2}{\tau^2})e^{-\frac{R^2(r-s)^2}{2\tau}}     d\tau&\le   R^2 \left(2\log \frac{T}{|s-r|R^2}+  \int_{R^2 (r-s)^2/t}^{\infty} e^{-u/2}du\right)\\
  &\le C_T  R^2 \log \frac{T}{|s-r|R^2},
  \end{aligned} ,$$
  for some constant $C_T>0.$  These two bounds, once integrated in (\ref{Proof Lem Taylor along  geo}), imply that for any $\alpha>0,$ the left-hand-side of (\ref{Taylor Kernel second order}) is bounded by $C_T d(p,q)^{2-\alpha},$ uniformly on $p,q\in M$ and $t\in [0,T].$ Using the bound $\min\{a,b\}\le a^\eta b^{1-\eta},$ for $a,b,\eta\in (0,1),$ gives (\ref{Taylor Kernel second order}). \end{proof}

\section{Appendix - Higher order \textquotedblleft polynomials\textquotedblright}
\label{sec:higherOrderPolynomials}

We recall the regularity structure of polynomial functions in flat space $\R^{d}$ given in
\cite{bib:hairer}. It is used to abstractly describe functions in $C^{\gamma
}\left(  \R^{d}\right)  $, $\gamma>0$, and also forms a central ingredient for
general regularity structures associated with singular SPDEs. Let $\gamma>0$
and $n=\lfloor\gamma\rfloor$, that is $n\in\N$ and $\gamma\in(n,n+1]$. For
simplicity of notation let $d=1$. Define
\[
\mcT^{flat}:=\bigoplus_{\ell=0,\dots,n}\operatorname{span}\{X\mathbb{^{\ell}%
}\},
\]
where $\operatorname{span}\{X^{\ell}\}$ denotes the one-dimensional vector
space spanned by the abstract symbol $X^{\ell}$. Hence $\mcT_{flat}%
\simeq\R^{n+1}$.
%This vector space will be used to store, at every point in $\R$

Given $x,y\in\mathbb{R}$ and $\ell\in\N,$ we define the linear maps, $\Pi
_{x}:\mcT_{flat}\rightarrow \subset\mathcal{D}^{\prime}\left(  \mathbb{R}\right)  $ and $\Gamma
_{x\leftarrow y}:\mcT_{flat}\rightarrow\mcT_{flat},$ which are uniquely
determined by
%%
%(\cdot-y) = (\cdot - x + x - y)^\ell
%			= \sum_{i\le \ell} 1/(..) (\cdot - x)^{i} (x-y)^{\ell-i}
%%
\begin{align}
  \Pi^{flat}_{x}X\mathbb{^{\ell}}  &  :=(\cdot-x)^{\ell} \label{eq:polynomialModelFlat}%
\\
\Gamma^{flat}_{x\leftarrow y}X\mathbb{^{\ell}}  &  :=\sum_{i\leq\ell}%
%TCIMACRO{\QATOPD{(}{)}{\ell}{i}}%
%BeginExpansion
\genfrac{(}{)}{0pt}{}{\ell}{i}%
%EndExpansion
(x-y)^{\ell-i}X^{i}.\nonumber
\end{align}
In this case one has $\Pi^{flat}_{x}\Gamma^{flat}_{x\leftarrow y}\tau=\Pi^{flat}_{y}\tau$ for all
$\tau\in\mcT_{flat}.$ One can use this regularity structure to describe
regular functions.

\begin{lemma}
[{\cite[Lemma 2.12]{bib:hairer}}]Let $f:\R\rightarrow\R$. Then $f\in
C^{\gamma}(\R)$ if and only if there exists $\hat{f}:\R\rightarrow \mcT^{flat}$ with
$\hat{f}_{0}(x)=f(x)$ and
\[
  |\hat{f}(x)-\Gamma^{flat}_{x\leftarrow y}\hat{f}(y)|_{\ell}\lesssim|x-y|^{\gamma
-\ell}.
\]
In that case $\hat{f}_{\ell}(x)=f^{(\ell)}(x),\ell=0,\dots,n$.
\footnote{
Here we recall the notation of $\hat{f}_\ell(x)$ as the component of $\hat{f}(x)$
on the $\ell$-th homogeneity, i.e. the coefficient in front of $\symbol{X^\ell}$.}
\end{lemma}

\iffalse Delete this stuff. If $\tau=\sum_{\ell=0}^{n}\tau_{\ell}X^{\ell},$
then
\[
\Pi_{x}\tau=\sum_{\ell=0}^{n}\tau_{\ell}(\cdot-x)^{\ell}%
\]%
\begin{align*}
\Gamma_{x\leftarrow y}\tau &  =\sum_{\ell=0}^{n}\tau_{\ell}\sum_{i=0}^{\ell}%
%TCIMACRO{\QATOPD{(}{)}{\ell}{i}}%
%BeginExpansion
\genfrac{(}{)}{0pt}{}{\ell}{i}%
%EndExpansion
(x-y)^{\ell-i}X^{i}\\
&  =\sum_{0\leq i\leq\ell\leq n}^{n}\tau_{\ell}%
%TCIMACRO{\QATOPD{(}{)}{\ell}{i}}%
%BeginExpansion
\genfrac{(}{)}{0pt}{}{\ell}{i}%
%EndExpansion
(x-y)^{\ell-i}X^{i}\\
&  =\sum_{i=0}^{n}\left[  \sum_{\ell=i}^{n}\tau_{\ell}%
%TCIMACRO{\QATOPD{(}{)}{\ell}{i}}%
%BeginExpansion
\genfrac{(}{)}{0pt}{}{\ell}{i}%
%EndExpansion
(x-y)^{\ell-i}\right]  X^{i}.
\end{align*}
\fi

%Given a function $f \in C^\gamma(\R)$ we can ``lift'' it to a function $\hat f$ taking values in $\mcT_{flat}$ via
%\begin{align*}
%\hat f(x) := 1 f(x) + X f'(x) + \dots + X^n f^{(n)}(x).
%\end{align*}

\subsection{Higher order covariant derivatives}

We want to mirror as best we can the flat space polynomial model described above,
in the general context of a $d$ dimensional
Riemannian manifold. In order to do to this we need to store higher order
derivatives of functions $f:M\rightarrow\R$ in a coordinate independent
fashion. There is a canonical way to do this on a Riemannian manifold by
making use of the associated Levi-Civita connection.

%It turns out that the symmetric part of the higher order covariant derivatives are convenient for this.
We recall the notion of higher order covariant derivatives of functions
$f:M\rightarrow\R$ on a Riemannian manifold with Levi-Civita\footnote{In
general, $\nabla$ can be any affine connection.} connection $\nabla$ (see for
example \cite[Lemma 4.6]{bib:leeRiemannian}).

\begin{definition}
\label{def.dell}Define $\nabla^{\ell}|_{p}f\in\left[  T_{p}^{\ast}M\right]
^{\otimes\ell}\cong\left[  T_{p}M^{\otimes\ell}\right]  ^{\ast}$ by,
\[
\nabla^{0}|_{p}f=f(p),\text{ }\langle\nabla|_{p}f,X_{1}\left(  p\right)
\rangle=\langle d|_{p}f,X_{1}\left(  p\right)  \rangle,
\]
and then inductively by;%
\begin{align*}
&  \left\langle \nabla^{\ell}|_{p}f,X_{1}\left(  p\right)  \otimes\dots\otimes
X_{\ell}\left(  p\right)  \right\rangle \,\\
&  \qquad=\left[  X_{1}\langle\nabla^{\ell-1}f,X_{2}\otimes\dots\otimes
X_{\ell}\rangle\right]  |_{p}\\
&  \qquad\qquad-\sum_{m=2}^{\ell}\langle\nabla^{\ell-1}f,X_{2}\otimes
\dots\otimes X_{m-1}\otimes\nabla_{X_{1}}X_{m}\otimes X_{m+1}\otimes
\dots\otimes X_{\ell}\rangle|_{p},
\end{align*}
where $X_{1},\dots,X_{\ell}$ are arbitrary vector fields on $M.$
\end{definition}

A few remarks are in order.

\begin{enumerate}
\item As the notation suggests, $\nabla^{\ell}|_{p}f$ is indeed tensorial,
i.e. the right side of the previously displayed equation really only depends
on the vector fields, $\left\{  X_{i}\right\}  _{i=1}^{\ell}$, through their
values at $p.$

\item In the literature $\nabla f$ sometimes denotes the gradient of $f$. We
never use the gradient of a function in this work.

\item We shall also sometimes write $\nabla_{W}^{\ell}f=\langle\nabla^{\ell
}|_{p}f,W\rangle$ for any $W\in(T_{p}M)^{\otimes\ell}.$
\end{enumerate}

\begin{lemma}
\label{l.2.4} If $f$ is an $\ell$-times continuously differentiable function
in a neighborhood of $p\in M,$ $v\in T_{p}M,$ and $\gamma_{v}(t):=\exp_{p}(tv),$ then
\begin{equation}
\left.  \frac{d^{\ell}}{dt^{\ell}}\right\vert _{t=0}f\left(  \gamma
_{v}(t)\right)  =\nabla_{v^{\otimes\ell}}^{\ell}f,\text{ for all }v\in T_{p}M.
\label{equ.dl}%
\end{equation}
More generally, if $\ell,n\in\mathbb{N}_{0},$ $f$ is an $(\ell+n+1)$-times
continuously differentiable function in a neighborhood of $p\in M,$
$//_{t}(\gamma_{v}):T_{\gamma_{v}(0)}M\rightarrow T_{\gamma_{v}(t)}M$ is
parallel translation along $\gamma_{v},$ $W_{0}\in T_{p}M^{\otimes\ell},$ and
$W_{t}:=//_{t}(\gamma_{v})^{\otimes\ell}W_{0},$ then
\begin{equation}
\frac{d^{k}}{dt^{k}}\nabla_{W_{t}}^{\ell}f=\nabla_{\dot{\gamma}_{v}\left(
t\right)  ^{\otimes k}\otimes W_{t}}^{\ell+k}f~\forall~0\leq k\leq n+1.
\label{equ.dl2}%
\end{equation}

\end{lemma}

\begin{proof}
Let $\gamma_{v}\left(  t\right)  :=\exp_{p}\left(  tv\right)  $ so that
$\gamma_{v}\left(  t\right)  $ solves the geodesic differential equation,
$\nabla\dot{\gamma}_{v}\left(  t\right)  /dt=0$ with $\dot{\gamma}_{v}\left(
0\right)  =v.$ The proof is completed by showing (by induction) that
\begin{equation}
\frac{d^{k}}{dt^{k}}f\left(  \gamma_{v}\left(  t\right)  \right)
=\nabla_{\dot{\gamma}_{v}\left(  t\right)  ^{\otimes k}}^{k}f\text{ for }1\leq
k\leq\ell. \label{equ.dk}%
\end{equation}
The case $k=1$ amounts to the definition that $\nabla_{v}f=vf=df\left(
v\right)  $ for all $v\in TM.$ For the induction step we have by the product
rule;%
\begin{align*}
\frac{d^{k+1}}{dt^{k+1}}f\left(  \gamma_{v}\left(  t\right)  \right)   &
=\frac{d}{dt}\nabla_{\dot{\gamma}_{v}\left(  t\right)  ^{\otimes k}}%
^{k}f=\nabla_{\dot{\gamma}_{v}\left(  t\right)  ^{\otimes\left(  k+1\right)
}}^{k+1}f+\nabla_{\frac{\nabla}{dt}\left[  \dot{\gamma}_{v}\left(  t\right)
^{\otimes k}\right]  }^{k}f\\
&  =\nabla_{\dot{\gamma}_{v}\left(  t\right)  ^{\otimes\left(  k+1\right)  }%
}^{k+1}f,
\end{align*}
wherein the last equality we have again used the product rule to conclude that
$\frac{\nabla}{dt}\left[  \dot{\gamma}_{v}\left(  t\right)  ^{\otimes
k}\right]  =0.$ The result now follows by evaluating \eqref{equ.dk} at
$k=\ell$ and $t=0.$ The more general assertion in \eqref{equ.dl2} is
proved similarly. One only need observe that $\nabla W_{t}/dt=0$ and hence the
presence of $W_{t}$ in the expressions in no way changes the computations.
\end{proof}

\begin{definition}
[Symmetrizations]\label{def.sym}If $V$ is a real vector space and $\ell
\in\mathbb{N},$ we let $\operatorname{Sym}_{\ell}:V^{\otimes\ell}\rightarrow
V^{\otimes\ell}$ denote the symmetrization projection uniquely determined by
\[
\operatorname{Sym}_{\ell}\left(  v_{1}\otimes\dots\otimes v_{\ell}\right)
=\frac{1}{\ell!}\sum_{\sigma\in S_{\ell}}v_{\sigma(1)}\otimes\dots\otimes
v_{\sigma(\ell)}%
\]
where $S_{\ell}$ is the permutation group on $\left\{  1,2,\dots,\ell\right\}
.$ Often we will simplyt write $\operatorname{Sym}$ for $\operatorname{Sym}%
_{\ell}$ as it will typically be clear what $\ell$ is from the argument put
into the symmetrization function.
\end{definition}

As usual we let $V^{\ast}$ denote the dual space to a vector space $V$ and let
$\left\langle \cdot,\cdot\right\rangle $ denote the pairing between a vector
space and its dual. We will often identify $\left(  V^{\ast}\right)
^{\otimes\ell}$ with $\left[  V^{\otimes\ell}\right]  ^{\ast}$ where the
identification is uniquely determined by%
\[
\left\langle \varepsilon^{1}\otimes\dots\dots\otimes\varepsilon^{\ell}%
,v_{1}\otimes\dots\otimes v_{\ell}\right\rangle =\varepsilon^{1}\left(
v_{1}\right)  \cdot\dots\cdot\varepsilon^{\ell}\left(  v_{\ell}\right)  \text{
}\forall~\left\{  \varepsilon^{i}\right\}  _{i=1}^{\ell}\subset V^{\ast}\text{
and }\left\{  v_{i}\right\}  _{i=1}^{\ell}\subset V.
\]
We also identify $\left(  V^{\ast}\right)  ^{\otimes\ell}$ with the space of
multi-linear maps from $V^{\ell}\rightarrow\mathbb{R}$ using,%
\[
T\left(  v_{1},\dots,v_{\ell}\right)  =\left\langle T,v_{1}\otimes\dots\otimes
v_{\ell}\right\rangle ~\forall~T\in\left(  V^{\ast}\right)  ^{\otimes\ell
}\text{ and }\left\{  v_{i}\right\}  _{i=1}^{\ell}\subset V.
\]
Under these identification we have%
\[
\left\langle \Sym[T],W\right\rangle =\left\langle T,\Sym[W]\right\rangle
\text{ }\forall~T\in\left(  V^{\ast}\right)  ^{\otimes\ell}\text{ and }W\in
V^{\otimes\ell}%
\]
and
\[
\Sym[T]\left(  v_{1},\dots,v_{\ell}\right)  =\left\langle T,\operatorname{Sym}%
\left(  v_{1}\otimes\dots\otimes v_{\ell}\right)  \right\rangle ~\forall
~T\in\left(  V^{\ast}\right)  ^{\otimes\ell}\text{ and }\left\{
v_{i}\right\}  _{i=1}^{\ell}\subset V.
\]

\begin{remark}
\label{r.2.2}If $T\in\left[  V^{\ast}\right]  ^{\otimes\ell}$ and $v_{1}%
,\dots,v_{\ell}\in V,$ then
\[
\frac{\partial}{\partial s_{1}}\dots\frac{\partial}{\partial s_{\ell}}%
|_{s_{1}=\dots=s_{\ell}=0}T\left(  \left(  s_{1}v_{1}+\dots+s_{\ell}v_{\ell
}\right)  ^{\otimes\ell}\right)  =\sum_{\sigma\in S_{\ell}}T\left(
v_{\sigma(1)},\dots,v_{\sigma(\ell)}\right)
\]
and therefore,%
\[
\Sym[T]\left(  v_{1},\dots,v_{\ell}\right)  :=\frac{1}{\ell!}
%\partial_{s_{1}}\dots\partial_{s_{\ell}}
\frac{\partial}{\partial s_{1}}\dots\frac{\partial}{\partial s_{\ell}}%
|_{s_{1}=\dots=s_{\ell}=0}T\left(  \left(  s_{1}v_{1}+\dots+s_{\ell}v_{\ell
}\right)  ^{\otimes\ell}\right)  .
\]
This formula shows that the symmetric part $\Sym\left[  T\right]  $ of $T$ is
completely determined by the knowledge of $T\left(  v,v,\dots,v\right)  $ for
all $v\in V.$
\end{remark}

%\begin{notation}
\begin{definition}
\label{n.2.1}Let $\Sigma^{\ell}T_{p}^{\ast}M$ denote the symmetric tensors in
$\left[  T_{p}^{\ast}M\right]  ^{\otimes\ell}$ and for $T\in\left[
T_{p}^{\ast}M\right]  ^{\otimes\ell},$ let $\Sym[T]\in\Sigma^{\ell}T_{p}%
^{\ast}M$ denote the symmetrization of $T$ as above.
\end{definition}
%\end{notation}

\begin{example}
\label{ex.smooth}
If $U$ is an open subset of $M$ and
$f$ is $\ell$-times continuously differentiable on $U$,
%$f\in C^{\ell}\left( U\right)  ,$
then $\operatorname{Sym}\left[  \nabla^{\ell}f\right]  $ defines
a local section (over $U)$ of $\Sigma^{\ell}T^{\ast}M.$ Moreover since
$v^{\otimes\ell}$ is symmetric for all $v\in T_{p}M$ we may write
\eqref{equ.dl} as
\begin{equation}
\left.  \frac{d^{\ell}}{dt^{\ell}}\right\vert _{t=0}f\left(  \gamma
_{v}(t)\right)  =\left\langle \operatorname{Sym}\left[  \nabla_{p}^{\ell
}f\right]  ,v^{\otimes\ell}\right\rangle ,\text{ for all }v\in T_{p}M.
\label{equ.dla}%
\end{equation}

\end{example}

\begin{theorem}
[Taylor's Theorem on $M$]\label{thm:taylor}Let $\ell,n\in\mathbb{N}_{0},$ $p\in
M,$ $v\in T_{p}M,$ $\gamma_{v}(t):=\exp_{p}(tv),$ $//_{t}(\gamma
_{v}):T_{\gamma_{v}(0)}M\rightarrow T_{\gamma_{v}(t)}M,$ $W_{0}\in
T_{p}M^{\otimes\ell},$ and $W_{t}:=//_{t}(\gamma_{v})^{\otimes\ell}W_{0}.$ If
%$f\in C^{\ell+n+1}\left(  U\right)  $
$f$ is $(\ell+n+1)$-times continuously differentiable on $U$,
where $U$ is an open set containing
$\gamma_{v}\left(  \left[  0,1\right]  \right)  ,$ then
\begin{equation}
\nabla_{W_{1}}^{\ell}f=\sum_{k=0}^{n}\frac{1}{k!}\nabla_{v^{\otimes k}\otimes
W_{0}}^{\ell+k}f+\frac{1}{n!}\int_{0}^{1}\left[  \nabla_{\dot{\gamma}%
_{v}\left(  t\right)  ^{\otimes\left(  n+1\right)  }\otimes W_{t}}^{\ell
+n+1}f\right]  \cdot\left(  1-t\right)  ^{n}dt. \label{equ.tay0}%
\end{equation}
When $\ell=0$ the previous equation is to be interpreted as (also see
\cite[Theorem 6.1]{bib:driverSemko})
\begin{align}
f\left(  \exp_p\left(  v\right)  \right)   &  =\sum_{k=0}^{n}\frac{1}{k!}%
\nabla_{v^{\otimes k}}^{k}f+\frac{1}{n!}\int_{0}^{1}\left[  \nabla
_{\dot{\gamma}_{v}\left(  t\right)  ^{\otimes\left(  n+1\right)  }}%
^{n+1}f\right]  \cdot\left(  1-t\right)  ^{n}dt\label{equ.tay3}\\
&  =\sum_{k=0}^{n}\frac{1}{k!}\left\langle \operatorname{Sym}\left[
\nabla_{p}^{k}f\right]  ,v^{\otimes k}\right\rangle +\frac{1}{n!}\int_{0}%
^{1}\left\langle \operatorname{Sym}\left[  \nabla_{\gamma_{v}\left(  t\right)
}^{n+1}f\right]  ,\dot{\gamma}_{v}\left(  t\right)  ^{\otimes\left(
n+1\right)  }\right\rangle \left(  1-t\right)  ^{n}dt, \label{equ.tay1b}%
\end{align}
where%
\[
\frac{1}{0!}\left\langle \Sym\nabla^{0}|_{p}f,v^{\otimes0}\right\rangle
:=f\left(  p\right)  .
\]

\end{theorem}

\begin{proof}
Let $g\left(  t\right)  :=\nabla_{W_{t}}^{\ell}f$ and recall that the standard
Taylor's theorem with remainder states;%
\[
g\left(  1\right)  =\sum_{k=0}^{n}\frac{1}{n!}g^{\left(  k\right)  }\left(
0\right)  +\frac{1}{n!}\int_{0}^{1}g^{\left(  n+1\right)  }\left(  t\right)
\left(  1-t\right)  ^{n}dt.
\]
The results now follow by using Lemma \ref{l.2.4} in order to compute the
$g^{\left(  k\right)  }\left(  t\right)  $ for $1\leq k\leq n+1.$
\end{proof}

Theorem \ref{thm:taylor} has the following immediate corollaries.

\begin{corollary}
\label{cor.tayc}Moreover,%
\begin{equation}
f\left(  \exp_p\left(  v\right)  \right)  =\sum_{k=0}^{n+1}\frac{1}%
{k!}\left\langle \operatorname{Sym}\left[  \nabla_{p}^{k}f\right]  ,v^{\otimes
k}\right\rangle +o_{f}\left(  \left\vert v\right\vert ^{n+1}\right)
\label{equ.tay1c}%
\end{equation}
where $o_{f}\left(  \left\vert v\right\vert ^{n+1}\right)  \leq\varepsilon
\left(  v\right)  \left\vert v\right\vert ^{n+1}$ and $\varepsilon\left(
v\right)  \rightarrow0$ as $\left\vert v\right\vert \rightarrow0.$
\end{corollary}

\begin{proof}
According to \eqref{equ.tay1b}, \eqref{equ.tay1c} holds with%
\[
o_{f}\left(  \left\vert v\right\vert ^{n+1}\right)  =\frac{1}{n!}\int_{0}%
^{1}\left\langle g\left(  v,t\right)  ,v^{\otimes\left(  n+1\right)
}\right\rangle \left(  1-t\right)  ^{n}dt
\]
where
\[
g\left(  v,t\right)  :=\operatorname{Sym}\left[  \nabla_{\gamma_{v}\left(
t\right)  }^{n+1}f\right]  //_{t}\left(  \gamma_{v}\right)  ^{\otimes\left(
n+1\right)  }-\operatorname{Sym}\left[  \nabla_{\gamma_{v}\left(  0\right)
}^{n+1}f\right]  \rightarrow0\text{ as }v\rightarrow0.
\]
Therefore $o_{f}\left(  \left\vert v\right\vert ^{n+1}\right)  \leq
\varepsilon\left(  v\right)  \left\vert v\right\vert ^{n+1}$ where
\[
\varepsilon\left(  v\right)  =\frac{1}{n!}\int_{0}^{1}\left\Vert g\left(
v,t\right)  \right\Vert \left(  1-t\right)  ^{n}dt\rightarrow0\text{ as
}\left\vert v\right\vert \rightarrow0.
\]

\end{proof}

\begin{remark}
\label{rem.est}Since parallel translation is isometric it follows (continuing
the notation in Theorem \ref{thm:taylor}) that%
\[
\left\vert \nabla_{\dot{\gamma}_{v}\left(  t\right)  ^{\otimes\left(
n+1\right)  }\otimes W_{t}}^{\ell+n+1}f\right\vert \leq\left\Vert \nabla
^{\ell+n+1}f\right\Vert _{\left[  T_{\gamma_{v}\left(  t\right)  }^{\ast
}M\right]  ^{\otimes\left(  \ell+n+1\right)  }}\left\Vert v\right\Vert
^{n+1}\left\Vert W_{0}\right\Vert
\]
and hence%
\begin{align}
&  \left\vert \nabla_{W_{1}}^{\ell}f-\sum_{k=0}^{n}\frac{1}{k!}\nabla
_{v^{\otimes k}\otimes W_{0}}^{\ell+k}f\right\vert \nonumber\\
&  \qquad\leq\frac{1}{\left(  n+1\right)  !}\left\Vert W_{0}\right\Vert
\cdot\max_{0\leq t\leq1}\left\Vert \nabla^{\ell+n+1}f\right\Vert _{\left[
T_{\gamma_{_{v}}\left(  t\right)  }^{\ast}M\right]  ^{\otimes\left(
\ell+n+1\right)  }}\cdot d\left(  p,\exp_p\left(  v\right)  \right)  ^{n+1}.
\label{equ.tayest}%
\end{align}

\end{remark}

Since $M$ is a compact Riemannian manifold it is necessarily complete and
therefore, by the Hopf--Rinow theorem, for each $q\in M$ we may find at least
one $v\in T_{p}M$ such that $q=\exp_p\left(  v\right)  $ and $d\left(
q,p\right)  =\left\vert v\right\vert .$ Using these remarks we can reformulate
\eqref{equ.tay3} as follows.

\begin{corollary}
\label{cor.tay}
If
%$f\in C^{n+1}\left(  M\right)  ,$
$f$ is $(n+1)$-times continuously differentiable on $M$,
$p,q\in M,$ and $v\in
T_{p}M$ is chosen so that $q=\exp_p\left(  v\right)  $ and $d\left(  q,p\right)
=\left\vert v\right\vert ,$ then
\begin{equation}
f\left(  q\right)  =\sum_{k=0}^{n}\frac{1}{k!}\left\langle \Sym\nabla^{k}%
|_{p}f,v^{\otimes k}\right\rangle +O_{f}\left(  d\left(  p,q\right)
^{n+1}\right)  \label{equ.tay}%
\end{equation}
where
\[
\left\vert O_{f}\left(  d\left(  p,q\right)  ^{n+1}\right)  \right\vert
\leq\frac{1}{\left(  n+1\right)  !}\max_{m\in M}\left\Vert \Sym\left[
\nabla^{n+1}|_{m}f\right]  \right\Vert d\left(  p,q\right)  ^{n+1}.
\]
Furthermore if
%$f\in C^{n}\left(  M\right)$
$f$ is $n$-times continuously differentiable on $M$
then
\[
f\left(  q\right)  =\sum_{k=0}^{n}\frac{1}{k!}\left\langle \Sym\nabla^{k}%
|_{p}f,v^{\otimes k}\right\rangle +o_{f}\left(  d\left(  p,q\right)
^{n}\right).
\]

\end{corollary}

\begin{definition}
[Taylor approximations]\label{def.tay}Suppose that $U\subset M$ is an open
subset of $M,$ $p\in U,$
%$f\in C^{n}\left(  U\right)  ,$
$f$ a $n$-times continuously differentiable function on $M$
and $\varepsilon>0$
is sufficiently small so that $B_{p}\left(  \varepsilon\right)  \subset U$ and
$\varepsilon$ is smaller than the injectivity radius of $M.$ We then define,
$\operatorname*{Tay}\nolimits_{p}^{n}f\in C^{\infty}\left(  B_{p}\left(
\varepsilon\right)  \right)  $ by%
\begin{equation}
\left(  \operatorname*{Tay}\nolimits_{p}^{n}f\right)  \left(  q\right)
:=\sum_{k=0}^{n}\frac{1}{k!}\left\langle \Sym\nabla^{k}|_{p}f,\left[  \exp
_{p}^{-1}\left(  q\right)  \right]  ^{\otimes k}\right\rangle
.\label{equ.taypn}%
\end{equation}

\end{definition}

With this notation we have the following local version of Corollary
\ref{cor.tay}.

\begin{corollary}
\label{cor.tay2}If
%$f\in C^{n}\left(  M\right)  ,$
$f$ is a $n$-times continuously differentiable function on $M$,
$p,q\in M$ with $d\left(
p,q\right)  $ smaller than the injectivity radius of $M,$ then
\begin{equation}
f\left(  q\right)  =\left(  \operatorname*{Tay}\nolimits_{p}^{n}f\right)
\left(  q\right)  +o_{f}\left(  d\left(  p,q\right)  ^{n}\right)  .
\label{equ.fq}%
\end{equation}

\end{corollary}

\begin{remark}
\label{rem.flat}In the case $M=\mathbb{R}^{d}$ and $f$ is a polyonmial of
degree at most $n,$ it follows by Taylor's theorem that $f=\operatorname*{Tay}%
\nolimits_{p}^{n}f$ for all $p\in\mathbb{R}^{d}.$ So in the flat case the
error term in \eqref{equ.fq} is no longer present.
\end{remark}

\begin{lemma}
\label{lem.symm}If
%$f\in C^{n}\left(  M\right)  $
$f$ is a $n$-times continuously differentiable function on $M$
and $f\left(  q\right)
=o\left(  d\left(  p,q\right)  ^{n}\right)  ,$ then $\left(  Vf\right)
\left(  p\right)  =0$ for any $n^{\text{th}}$ - order differential operator
$V$
%on $C^{n}\left(  M\right)  $
and in particular, $\nabla^{k}|_{p}f=0$ for
all $0\leq k\leq n.$
\end{lemma}

\begin{proof}
Let $\left(  \Psi,U\right)  $ be a chart on with $p\in U$ and $\Psi\left(
p\right)  =0$ and define $F:=f\circ\Psi^{-1}\in C^{n}\left(  \tilde{U}%
:=\Psi\left(  U\right)  \right)  .$ Then the give assumption implies $F\left(
x\right)  =o\left(  \left\vert x\right\vert ^{n}\right)  $ and therefore for
any $x\in\mathbb{R}^{d}$ and $t\in\mathbb{R}$ small we have $F\left(
tx\right)  =o\left(  t^{n}\right)  $ from which it easily follows that
\[
0=\frac{d^{k}}{dt^{k}}F\left(  tx\right)  |_{t=0}=\left\langle \left(
D^{k}F\right)  \left(  0\right)  ,x^{\otimes k}\right\rangle \text{ for all
}0\leq k\leq n.
\]
As $\left(  D^{k}F\right)  \left(  0\right)  $ is symmetric and $x\in
\mathbb{R}^{d}$ was arbitrary we may conclude that $\left(  D^{k}F\right)
\left(  0\right)  =\mathbf{0}\in\left(  \mathbb{R}^{d}\right)  ^{\ast\otimes
k}$ for $0\leq k\leq n.$ As any $n^{\text{th}}$ - order differential operator
$U$ on $C^{n}\left(  M\right)  $ may be written locally as%
\[
Vf=\sum_{k=0}^{n}\left\langle \left(  D^{k}F\right)  \left(  \Psi\right)
,W_{k}\right\rangle
\]
for some smooth functions, $W_{k}:U\rightarrow\left(  \mathbb{R}^{d}\right)
^{\otimes k}$ for each $0\leq k\leq n,$ it follows that%
\[
\left(  Vf\right)  \left(  p\right)  =\sum_{k=0}^{n}\left\langle \left(
D^{k}F\right)  \left(  \Psi\left(  p\right)  \right)  ,W_{k}\left(  p\right)
\right\rangle =\sum_{k=0}^{n}\left\langle \left(  D^{k}F\right)  \left(
0\right)  ,W_{k}\left(  p\right)  \right\rangle =0.
\]

\end{proof}

\begin{corollary}
\label{cor.derivatives}If
%$f\in C^{n}\left(  M\right)  $
$f$ a $n$-times continuously differentiable function on $M$
and $V$ is an
$n^{\text{th}}$ - order differential operator, %on $C^{n}\left(  M\right)  ,$
then
\[
\left(  Vf\right)  \left(  p\right)  =\left[  V\left(  \operatorname*{Tay}%
\nolimits_{p}^{n}f\right)  \right]  \left(  p\right)
\]
and in particular,
\[
\nabla^{n}|_{p}f=\left[  \nabla^{n}|_{p}\left(  \operatorname*{Tay}%
\nolimits_{p}^{n}f\right)  \right]  \left(  p\right)
\]
from which it follows that $\nabla^{n}|_{p}f$ is a linear combination of
$\left\{  \Sym\nabla^{k}|_{p}f\right\}  _{k=0}^{n}.$
\end{corollary}

We will make the last assertion of Corollary \ref{cor.derivatives} more
explicitly in Corollary \ref{cor.2.8} and Remark \ref{rem.global}. The upshot
is that there is no loss of information in only keeping track of the
symmetrizations of the covariant derivatives.

\begin{corollary}
\label{cor.tay3}If $f\in C^{\infty}\left(  M\right)  ,$ $p,q\in M$ with
$d\left(  p,q\right)  $ then for $0\leq k\leq n$ we have
\[
\left\Vert \nabla^{k}|_{q}\left[  f-\left(  \operatorname*{Tay}\nolimits_{p}%
^{n}f\right)  \right]  \right\Vert \leq\frac{1}{\left(  n+1-k\right)  !}%
\max_{0\leq t\leq1}\left\Vert \nabla^{n+1}|_{\gamma_{v}\left(  t\right)
}\left(  f-\operatorname*{Tay}\nolimits_{p}^{n}f\right)  \right\Vert \cdot
d\left(  p,q\right)  ^{n+1-k},
\]
where $v:=\exp_{p}^{-1}\left(  q\right)  .$
\end{corollary}

\begin{proof}
Let us apply the estimtate in \eqref{equ.tayest} with $f$ replaced by
$g:=f-\operatorname*{Tay}\nolimits_{p}^{n}f$ keeping in mind that $\nabla
^{k}|_{p}\left[  f-\left(  \operatorname*{Tay}\nolimits_{p}^{n}f\right)
\right]  =0$ for $0\leq k\leq n$ by Corollary \ref{cor.derivatives}. This
allows us to conlude for $W_{0}\in T_{p}M^{\otimes k}$ that%
\[
\left\vert \nabla_{W_{1}}^{k}\left[  f-\operatorname*{Tay}\nolimits_{p}%
^{n}f\right]  \right\vert \leq\frac{1}{\left(  n+1-k\right)  !}\left\Vert
W_{0}\right\Vert \cdot\max_{0\leq t\leq1}\left\Vert \nabla^{n+1}g\right\Vert
_{\left[  T_{\gamma_{_{v}}\left(  t\right)  }^{\ast}M\right]  ^{\otimes\left(
n+1\right)  }}\cdot d\left(  p,q\right)  ^{n+1-k}%
\]
where $v:=\exp_{p}^{-1}\left(  q\right)  .$ As the map $W_{0}\rightarrow
W_{1}$ is an isometry it follows that
\[
\left\Vert \nabla^{k}|_{q}\left[  f-\left(  \operatorname*{Tay}\nolimits_{p}%
^{n}f\right)  \right]  \right\Vert \leq\frac{1}{\left(  n+1-k\right)  !}%
\max_{0\leq t\leq1}\left\Vert \nabla^{n+1}g\right\Vert _{\left[
T_{\gamma_{_{v}}\left(  t\right)  }^{\ast}M\right]  ^{\otimes\left(
n+1\right)  }}\cdot d\left(  p,q\right)  ^{n+1-k}.
\]

\end{proof}

%Then: we have to check that we do not loose any information.
%In a quantified way, is this checked in
%Lemma \ref{l.2.7bb},
%Corollary \ref{cor.2.7} and
%Corollary \ref{cor.2.8}.

%from it the unsymmetrized version as well as higher order derivatives in coordinates.

\subsubsection{Symmetric parts of covariant derivatives determine all
derivatives}

%XXX changing notation of $\nabla^\ell$, previous version was at 7:25 pm, 3. August

We will now make Corollary \ref{cor.derivatives} more precise.

%\begin{notation}
\begin{definition}
\label{not.chart}If $\left(  x,U:=\dom( x )  \right)  $ is a
chart on $M,$ let $D^{x}$ denote the flat covariant derivative on $TU$
determined by $D^{x}\frac{\partial}{\partial x^{j}}=0$ for $1\leq j\leq d.$
\end{definition}
%\end{notation}

\begin{remark}
\label{rem:50} If $V=\sum_{j=1}^{d}V_{j}\frac{\partial}{\partial x^{j}}$ is a
vector field on $U$ and $v\in T_{m}M$ then $D_{v}^{x}V=\sum_{j=1}^{d}\left(
vV_{j}\right)  \frac{\partial}{\partial x^{j}}|_{m}.$ Using $D^{x}%
\frac{\partial}{\partial x^{j}}=0$, it easily follows that for all $\ell
\in\mathbb{N}$ and
%$f\in C^{\ell}\left(  U\right)  $
$f$ $\ell$-times continuously differentiable
we have
\[
\left\langle \left(  D^{x}\right)  ^{\ell}|_{m}f,\frac{\partial}{\partial
x^{i_{1}}}\otimes\dots\otimes\frac{\partial}{\partial x^{i_{\ell}}%
}\right\rangle =\left.  \frac{\partial}{\partial x^{i_{1}}}\dots\frac
{\partial}{\partial x^{i_{\ell}}}\right\vert _{m}f
\]
and in particular $\left(  D^{x}\right)  ^{\ell}f\in\Sigma^{\ell}T^{\ast}U.$
\end{remark}

\begin{lemma}
\label{l.2.7bb} Suppose that $\left(  x,U:=\dom(x) \right)  $ is a chart on $M,$ $D=D^{x}$ is the covariant derivative of
Notation \ref{not.chart}. Then, there exists a family of sections $Q_{\ell
,n}\in\Gamma\left[  \operatorname{Hom}\left[  TU^{\otimes n},TU^{\otimes\ell
}\right]  \right]  $ for $1\leq\ell\leq n,$ such that $Q_{n,n}=I$ and for all
$n$-times continuously differentiable functions $f,$
\begin{equation}
\left\langle \nabla^{n}|_{p} f, W \right\rangle = \sum_{\ell=1}^{n}%
\left\langle D^{\ell}|_{p} f,Q_{\ell,n} W\right\rangle \text{ } \forall
\ W\in\left[  T_{p} M \right]  ^{\otimes n}. \label{equ.id1}%
\end{equation}

\end{lemma}

\begin{proof}
Let $D=D^{x}$ and $\Gamma$ be the $\operatorname*{End}\left(  TU\right)  $ --
valued connection one form on $TU$ so that $\nabla=D+\Gamma.$
It is enough to verify that
\eqref{equ.id1} holds on a basis for $T_{p}U^{\otimes n}.$ To this end, let
$i_{j}\in\left\{  1,2\dots,d\right\}  ,$ for $1\leq j\leq n$ and let
$V_{j}=\frac{\partial}{\partial x^{i_{j}}}.$ Then,
\[
%\nabla_{V_{1}}f=V_{1}f=\left\langle Df,V_{1}\right\rangle
\left\langle \nabla f,V_{1}\right\rangle =V_{1}f=\left\langle Df,V_{1}%
\right\rangle ,
\]
which shows that \eqref{equ.id1} holds for $n=1.$ \iffalse Here are some more
hand calculations which may be deleted:
\begin{align*}
\nabla_{V_{2}\otimes V_{1}}^{2}f  &  =V_{2}V_{1}f-\nabla_{\nabla_{V_{2}}V_{1}%
}f=\left\langle D^{2}f,V_{2}\otimes V_{1}\right\rangle -\left\langle
Df,\Gamma\left(  V_{2}\right)  V_{1}\right\rangle \\
\nabla_{V_{3}\otimes V_{2}\otimes V_{1}}^{3}f  &  =V_{3}\left[  \nabla
_{V_{2}\otimes V_{1}}^{2}f\right]  -\nabla_{\nabla_{V_{3}}\left[  V_{2}\otimes
V_{1}\right]  }^{2}f.
\end{align*}
\fi For the sake of completing the proof by induction, let us now assume that
\eqref{equ.id1} holds at level $n-1$ and below. In particular we assume
\[
\left\langle \nabla^{n-1}f,{V_{n-1}\otimes\dots\otimes V_{1}}\right\rangle
=\sum_{\ell=1}^{n-1}\left\langle D^{\ell}f,Q_{\ell,n-1}V_{n-1}\otimes
\dots\otimes V_{1}\right\rangle .
\]
On one hand,
\begin{align*}
V_{n}\left\langle \nabla^{n-1}f,{V_{n-1}\otimes\dots\otimes V_{1}%
}\right\rangle =  &  \left\langle \nabla^{n}f,{V_{n}\otimes\dots\otimes V_{1}%
}\right\rangle +\left\langle \nabla^{n-1}f,{\nabla_{V_{n}}\left[
V_{n-1}\otimes\dots\otimes V_{1}\right]  }\right\rangle \\
=  &  \left\langle \nabla^{n}f,{V_{n}\otimes\dots\otimes V_{1}}\right\rangle
\\
&  +\sum_{k=1}^{n-1}\left\langle \nabla^{n-1}f,{V_{n-1}\otimes\dots
\otimes\left[  \Gamma\left(  V_{n}\right)  V_{k}\right]  \otimes\dots\otimes
V_{1}}\right\rangle ,
\end{align*}
while on the other hand (using the induction hypothesis, the product rule, and
$DV_{k}=0$ for all $k),$
\begin{align*}
V_{n}\left\langle \nabla^{n-1}f,{V_{n-1}\otimes\dots\otimes V_{1}%
}\right\rangle =  &  \sum_{\ell=1}^{n-1}V_{n}\left\langle D^{\ell}%
f,Q_{\ell,n-1}V_{n-1}\otimes\dots\otimes V_{1}\right\rangle \\
=  &  \sum_{\ell=1}^{n-1}\left\langle D^{\ell+1}f,V_{n}\otimes Q_{\ell
,n-1}V_{n-1}\otimes\dots\otimes V_{1}\right\rangle \\
&  +\sum_{\ell=1}^{n-1}\left\langle D^{\ell}f,\left(  D_{V_{n}}Q_{\ell
,n-1}\right)  V_{n-1}\otimes\dots\otimes V_{1}\right\rangle \\
=  &  \left\langle D^{n}f,V_{n}\otimes\dots\otimes V_{1}\right\rangle \\
&  +\sum_{\ell=1}^{n-2}\left\langle D^{\ell}f,V_{n}\otimes Q_{\ell,n-1}%
V_{n-1}\otimes\dots\otimes V_{1}\right\rangle \\
&  +\sum_{\ell=1}^{n-1}\left\langle D^{\ell}f,\left(  D_{V_{n}}Q_{\ell
,n-1}\right)  V_{n-1}\otimes\dots\otimes V_{1}\right\rangle .
\end{align*}
Comparing the last two displayed equations shows,
\begin{align*}
\nabla_{V_{n}\otimes\dots\otimes V_{1}}^{n}f  &  =\left\langle D^{n}%
f,V_{n}\otimes\dots\otimes V_{1}\right\rangle +\left\langle Df,\left[
D_{V_{n}}Q_{1,n-1}\right]  V_{n-1}\otimes\dots\otimes V_{1}\right\rangle \\
&  +\sum_{\ell=2}^{n-1}\left\langle D^{\ell}f,V_{n}\otimes Q_{\ell
-1,n-1}.V_{n-1}\otimes\dots\otimes V_{1}+\left[  D_{V_{n}}Q_{\ell,n-1}\right]
V_{n-1}\otimes\dots\otimes V_{1}\right\rangle \\
&  -\sum_{k=1}^{n-1}\left\langle \nabla^{n-1}f,{V_{n-1}\otimes\dots
\otimes\Gamma\left(  V_{n}\right)  V_{k}\otimes\dots\otimes V_{1}%
}\right\rangle .
\end{align*}
From this expression
%along with the induction hypothesis
it follows that $\nabla_{V_{n}\otimes\dots\otimes V_{1}}^{n}f$ may be
expressed in the form claimed in \eqref{equ.id1}.
\end{proof}

\begin{corollary}
\label{cor.2.7}Let us continue the notation in Lemma \ref{l.2.7bb}. Then,
there exists
\[
\bar{Q}_{\ell,n}\in\Gamma\left[  \operatorname{Hom}\left[  TU^{\otimes
n},TU^{\otimes\ell}\right]  \right]  ,\quad\text{for }1\leq\ell\leq n,
\]
such that $\bar{Q}_{n,n}=I$ and for all $n$-times continuously differentiable
functions $f,$
\begin{equation}
\left\langle D^{n}f,W\right\rangle =\sum_{\ell=1}^{n}\left\langle
\Sym\nabla^{\ell}f,\bar{Q}_{\ell,n}W\right\rangle ,~\forall\ W\in TM^{\otimes
n}. \label{equ.id2}%
\end{equation}

\end{corollary}

\begin{proof}
The proof is again by induction on $n.$ For $n=1,$ we have $D_{W}%
f=Wf=\nabla_{W}f,$ so there is nothing to prove. For the inductive step,
suppose that \eqref{equ.id2} holds at level $n-1$ and below. From
\eqref{equ.id1} with $W$ replaced by $\operatorname{Sym}_{n}W,$ it follows
that,%
\begin{align*}
\left\langle \Sym\left[  \nabla^{n}f\right]  ,W\right\rangle  &  =\left\langle
\nabla^{n}f,{\operatorname{Sym}_{n}W}\right\rangle =\sum_{\ell=1}%
^{n}\left\langle D^{\ell}f,Q_{\ell,n}\operatorname{Sym}_{n}W\right\rangle \\
&  =\left\langle D^{n}f,W\right\rangle +\sum_{\ell=1}^{n-1}\left\langle
D^{\ell}f,Q_{\ell,n}\operatorname{Sym}_{n}W\right\rangle ,
\end{align*}
wherein the last equality we have used that $D^{n}f$ is already symmetric.
From the previous equation along with the inductive hypothesis, we conclude
that $\left\langle D^{n}f,W\right\rangle $ may be expressed as described in \eqref{equ.id2}.
\end{proof}

\begin{corollary}
\label{cor.2.8}If $\nabla$ is a covariant derivative on $TM,$ then there
exists
\[
Q_{\ell,n}^{\nabla}\in\Gamma\left[  \operatorname{Hom}\left[  TM^{\otimes
n},TM^{\otimes\ell}\right]  \right]  , \text{ for }1\leq\ell\leq n,
\]
such that $Q_{n,n}^{\nabla}=I$ and for all $n$-times continuously
differentiable functions $f,$
\begin{equation}
\left\langle \nabla^{n}f, W \right\rangle = \sum_{\ell=1}^{n} \left\langle
\Sym\nabla^{\ell} f,Q_{\ell,n}^{\nabla}W\right\rangle , \text{ }\forall\ W\in
TM^{\otimes n}. \label{equ.id3}%
\end{equation}

\end{corollary}

\begin{proof}
First suppose that $M=U,$ as in Lemma \ref{l.2.7bb}. Then combining the
results of Lemma \ref{l.2.7bb} and Corollary \ref{cor.2.7}, there exists
$Q_{\ell,n}^{x}\in\Gamma\left[  \operatorname{Hom}\left[  TU^{\otimes
n},TU^{\otimes\ell}\right]  \right]  $ such that \eqref{equ.id3} holds for all
$W\in TU^{\otimes M}.$ Let $\left\{  x_{\alpha}\right\}  _{\alpha=1}^{N}$ be a
collection of charts on $M$ such that $\left\{  \dom\left(  x_{\alpha
}\right)  \right\}  _{\alpha=1}^{N}$ is an open cover of $M$ and $\left\{
\psi_{\alpha}\right\}  _{\alpha=1}^{N}$ be a partition of unity relative to
this cover. To complete the proof we define%
\[
Q_{\ell,n}^{\nabla}:=\sum_{\alpha=1}^{N}\psi_{\alpha}Q_{\ell,n}^{x_{\alpha}}.
\]

\end{proof}

We note the following corollary for completeness.
\begin{corollary}
  \label{cor.2.8b}If $\nabla$ is a covariant derivative on $TM$ on $L$ is a
  linear $n^{\text{th}}$ - order differential operator on $C^{\infty}\left(
  M\right)  ,$ then there exists smooth sections, $W_{\ell}\in\Gamma\left(
  \Sigma^{\ell}TM\right)  $ for $0\leq\ell\leq n$ such that
  \begin{equation}
  Lf=\sum_{\ell=0}^{n}\nabla_{W_{\ell}}^{\ell}f\text{ for all }f\in C^{\infty
  }\left(  M\right)  . \label{equ.lf}%
  \end{equation}
\end{corollary}
\begin{proof}
  By definition $Lf$ is locally given by $Lf=\sum_{\ell=0}^{n}\left\langle
  D^{n}f,A_{n}\right\rangle $ for some $A_{n}\in\Gamma\left(  \Sigma^{\ell
  }TU\right)  .$ Using Corollaries \ref{cor.2.7} and \ref{cor.2.8}, we may
  locally express $Lf$ as in \eqref{equ.lf}. The global picture may then be
  constructed using a partition of unity argument.
\end{proof}

\begin{remark}
\label{rem.global}Our proof of Corollary \ref{cor.2.8} was local in nature and
hence does not give much information about how the $Q_{\ell,n}^{\nabla}$
depend on $\nabla.$ It is possible to give a global proof of Corollary
\ref{cor.2.8} which would show that $Q_{\ell,n}^{\nabla}$ may be constructed
from certain combinations of covariant derivatives of the torsion and
curvature tensor of $\nabla.$ Here is a sketch of this argument. In this
sketch we let $v\wedge w:=v\otimes w-w\otimes v$ for any $v,w\in T_{p}M.$

\begin{enumerate}
\item If $v_{1},\dots,v_{n}\in T_{p}M$ and $1\leq i<n,$ then
\begin{multline*}
\nabla_{v_{n}\otimes\dots\otimes v_{i+2}\otimes\left[  v_{i+1}\wedge
v_{i}\right]  \otimes v_{i-1}\otimes\dots\otimes v_{1}}^{n}f\\
=\left\langle \nabla_{v_{n}\otimes\dots\otimes v_{i+2}}^{n-i-1}\left[
R\left(  \cdot,\cdot\right)  \nabla^{i-1}f\right]  ,v_{i+1}\otimes
v_{i}\otimes v_{i-1}\otimes\dots\otimes v_{1}\right\rangle \\
+\left\langle \nabla_{v_{n}\otimes\dots\otimes v_{i+2}}^{n-i-1}\left[
\nabla_{T\left(  \cdot,\cdot\right)  }\nabla^{i-1}f\right]  ,v_{i+1}\otimes
v_{i}\otimes v_{i-1}\otimes\dots\otimes v_{1}\right\rangle
\end{multline*}
where $R\left(  \cdot,\cdot\right)  \nabla^{i-1}f$ is the appropriate action
of the curvature tensor of $\nabla$ on $\nabla^{i-1}f$ and $T$ is the torsion
tensor of $\nabla.$

\item As a consequence of item 1. and the fact that every permutation is a
composition of transpositions, it follows that for any permutation $\sigma\in
S_{n},$
\begin{equation}
\nabla_{v_{\sigma(n)}\otimes\dots\otimes v_{\sigma(1)}}^{n}f=\nabla
_{v_{n}\otimes\dots\otimes v_{1}}^{n}f+\sum_{\ell=1}^{n-1}\left\langle
\nabla^{\ell}f,\mathbf{Q}\left(  \sigma\right)  _{\ell,n}v_{n}\otimes
\dots\otimes v_{1}\right\rangle , \label{equ.id4}%
\end{equation}
where $\mathbf{Q}\left(  \sigma\right)  _{\ell,n}\in\Gamma\left[
\operatorname{Hom}\left[  TM^{\otimes n},TM^{\otimes\ell}\right]  \right]  $
are constructed from certain combinations of covariant derivatives of the
torsion and curvature tensor of $\nabla.$

\item Summing \eqref{equ.id4} on $\sigma$ and then dividing by $n!$ and
setting
\[
\mathbf{Q}_{\ell,n}=\frac{1}{n!}\sum_{\sigma\in S_{n}}\mathbf{Q}\left(
\sigma\right)  _{\ell,n}%
\]
shows
\begin{equation}
\left\langle \Sym\nabla^{n}f,v_{n}\otimes\dots\otimes v_{1}\right\rangle
=\left\langle \nabla^{n}f,v_{n}\otimes\dots\otimes v_{1}\right\rangle
+\sum_{\ell=1}^{n-1}\left\langle \nabla^{\ell}f,\mathbf{Q}_{\ell,n}%
v_{n}\otimes\dots\otimes v_{1}\right\rangle ,\label{equ.id5}%
\end{equation}
where the $\mathbf{Q}_{\ell,n}\in\Gamma\left[  \operatorname{Hom}\left[
TM^{\otimes n},TM^{\otimes\ell}\right]  \right]  $ are constructed from
certain combinations of covariant derivatives of the torsion and curvature
tensor of $\nabla.$

\item Using \eqref{equ.id5} recursively then shows there exists $Q_{\ell
,n}^{\nabla}\in\Gamma\left[  \operatorname{Hom}\left[  TM^{\otimes
n},TM^{\otimes\ell}\right]  \right]  $ such that
\[
\left\langle \nabla^{n}f,v_{n}\otimes\dots\otimes v_{1}\right\rangle
=\left\langle \Sym\nabla^{n}f,v_{n}\otimes\dots\otimes v_{1}\right\rangle
+\sum_{\ell=1}^{n-1}\left\langle \Sym\nabla^{\ell}f,Q_{\ell,n}^{\nabla}%
v_{n}\otimes\dots\otimes v_{1}\right\rangle ,
\]
where each $Q_{\ell,n}^{\nabla}$ is constructed from certain combinations of
covariant derivatives of the torsion and curvature tensor of $\nabla.$
\end{enumerate}
\end{remark}

\subsection{The regularity structure and model\label{sec.trs}}

We are now ready to set up to regularity structure for ``polynomials'' up to order $n$
on a manifold.
%\begin{notation}
\begin{definition}
\label{n.2.3}Fix $n\geq0$ and let $\mathcal{T=}\bigoplus_{\ell=0}^{n}
\Sigma^{\ell}T^{\ast}M$ be the vector bundle over $M$ with fiber at $p\in M$
given by
\begin{equation}
\mathcal{T}|_{p}:=\bigoplus_{\ell=0}^{n}\Sigma^{\ell}T_{p}^{\ast}M.
\label{equ.4}%
\end{equation}
For $p\in M$ and $z$ near $p,$ let
\begin{equation}
G^n_{p}\left(  z\right)  :=\sum_{\ell=0}^{n}\frac{1}{\ell!}\left[  \exp_{p}
^{-1}\left(  z\right)  \right]  ^{\otimes\ell}\in\bigoplus_{\ell=0}^{n}
\Sigma^{\ell}T_{p}M. \label{equ.5}%
\end{equation}
\end{definition}
%\end{notation}

The vector bundle $\mcT$ will be used to store higher order derivatives of
functions. On flat space $\R^{d}$ such ``abstract Taylor expansions'' were
realized as honest functions using polynomials, see
\eqref{eq:polynomialModelFlat}. Polynomials are the simplest function that
have specified derivatives at one point. On the manifold we instead choose
polynomials in exponential coordinates.

\begin{definition}
[Model]\label{def.mod}For $\tau=(\tau_{0},\dots,\tau_{n})\in\mathcal{T}_{p}$
define
\begin{align*}
(\Pi_{p}\tau)\left(  z\right)    & :=\left\langle \tau,G^n_{p}\left(  z\right)
\right\rangle =\left\langle \tau,\sum_{\ell=0}^{n}\frac{1}{\ell!}\left[
\exp_{p}^{-1}\left(  z\right)  \right]  ^{\otimes\ell}\right\rangle \\
& =\sum_{\ell=0}^{n}\frac{1}{\ell!}\tau_{\ell}\left(  \left[  \exp_{p}%
^{-1}\left(  z\right)  \right]  ^{\otimes\ell}\right)  .
\end{align*}
%[The index $p$ on $\Pi_{p}$ is a throwback from the flat space setting since
%in the geometric context we can not longer identify $\mathcal{T}_{p}$ and
%$\mathcal{T}_{q}$ for different $p,q\in M.]$
\end{definition}

These local \textquotedblleft Taylor polynomials\textquotedblright\ are a good
substitute for the usual Taylor polynomials in the flat space theory, as Lemma
\ref{lem:theyRealize1} and Corollary \ref{lem:theyRealize} below demonstrate.

\begin{lemma}
\label{lem:theyRealize1}Let $A=A_{0}+A_{1}+\dots+A_{n}\in\mathcal{T}_{p},$
with $A_{\ell}\in\Sigma^{\ell}T_{p}^{\ast}M,$ for $\ell=0,\ldots,n,$ and
define
\[
\varphi\left(  q\right)  :=\left\langle A,G^n_{p}\left(  q\right)  \right\rangle
=\sum_{\ell=0}^{n}\frac{1}{\ell!}A_{\ell}(\exp_{p}^{-1}(q)^{\otimes\ell}).
\]
Then,
\[
\varphi\left(  p\right)  =A_{0}\text{ and }\Sym[(\nabla)^{\ell
}\varphi|_{p}]=A_{\ell},~\forall~\ell=1,\dots,n.
\]

\end{lemma}

\begin{proof}
Let $\gamma_{v}\left(  t\right)  =\exp_{p}\left(  tv\right)  $.
%as in Lemma \ref{l.2.4}.
Then,
\[
\varphi\left(  \gamma_{v}\left(  t\right)  \right)  =\sum_{\ell=0}^{n}\frac
{1}{\ell!}A_{\ell}(\exp_{p}^{-1}(\gamma_{v}(t)^{\otimes\ell})=\sum_{\ell
=0}^{n}\frac{t^{\ell}}{\ell!}A_{\ell}\left(  v^{\otimes\ell}\right)
\]
and hence by Lemma \ref{l.2.4}
\[
\nabla_{v^{\otimes\ell}}^{\ell}\varphi=\left.  \frac{d^{\ell}}{dt^{\ell}%
}\right\vert _{t=0}\varphi\left(  \gamma_{v}\left(  t\right)  \right)
=A_{\ell}\left(  v^{\otimes\ell}\right)  ,
\]
which suffices to complete the proof by Remark \ref{r.2.2}.
\end{proof}

We now have the immediate corollary of this lemma.

\begin{corollary}
\label{lem:theyRealize} Let $\tau\in\Sigma^{\ell}T_{p}^{\ast}M$. Then, for
$i=0,\dots,n,$
\[
\Sym[ \nabla^i|_p \Pi_p \tau ]=%
\begin{cases}
\tau,\quad & i=\ell\\
0, & \text{else}.
\end{cases}
\]

\end{corollary}

\begin{remark}
\label{rem.sym}Let $x$ be exponential coordinates around $p\in M$, i.e.
suppose that $x=\left(  x^{1},\dots,x^{d}\right)  $ where $\left\{
x^{i}\left(  q\right)  \right\}  _{i=1}^{d}$ are the coordinates of $\exp
_{p}^{-1}\left(  q\right)  $ relative to some basis $\left\{  u_{i}\right\}
_{i=1}^{d}$ of $T_{p}M.$ Then with $v=\sum_{i=1}^{d}v^{i}u_{i}\in T_{p}M,$
\begin{align*}
\sum_{i_{1}\dots i_{\ell}=1}^{d}\partial_{i_{1}\dots i_{\ell}}f(p)v^{i_{1}%
}\dots v^{i_{\ell}}  &  =\frac{d^{\ell}}{dt^{\ell}}|_{0}f\left(  \exp_p \left(  tv\right)  \right)  =\left\langle \nabla_{p}^{\ell}%
f,v^{\otimes\ell}\right\rangle \\
&  =\left\langle \Sym[ \nabla^{\ell}|_p f],v^{\otimes\ell}\right\rangle \\
&  =\sum_{i_{1}\dots i_{\ell}=1}^{d}\left\langle
\Sym[ \nabla^{\ell}|_p f],\left(  \frac{\partial}{\partial x^{i_{1}}}%
,\dots,\frac{\partial}{\partial x^{i_{\ell}}}\right)  \right\rangle v^{i_{1}%
}\dots v^{i_{\ell}}%
\end{align*}
\newline from which it follows that
\[
\partial_{i_{1}\dots i_{\ell}}f(p)=\Sym[ \nabla^{\ell}|_p f]\left(
\frac{\partial}{\partial x^{i_{1}}},\dots,\frac{\partial}{\partial x^{i_{\ell
}}}\right)  ,\forall i_{1},\ldots,i_{\ell}=1\ldots d.
\]

\end{remark}

\begin{definition}
[Transportation]\label{d.2.9} Let $\Gamma_{p\leftarrow q}:\mathcal{T}%
|_{q}\rightarrow\mathcal{T}|_{p}$ be defined by $\Gamma_{p\leftarrow q}%
\tau:=\bar{\tau},$ where
\[
\bar{\tau}_\ell :=\Sym\left[  \nabla^{\ell}|_{p}(\Pi_{q}%
\tau)\right]  =\Sym\left[  \nabla^{\ell}|_{p}\left\langle \tau,G^n_{q}\left(
\cdot\right)  \right\rangle \right]  ,\text{ }\forall\ell=0,\ldots,n,
\]
which makes sense for $d(p,q)<\delta$, the radius of injectivity of $M$.
\end{definition}

\begin{remark}
\label{rem:upwards} For $n\geq2$ this transport will in general also go
\textquotedblleft upwards.\textquotedblright\ That is, if $\tau\in
\mcT_{\alpha}|_{y}$ some $\alpha<n$, then in general $\Gamma_{x\leftarrow
y}\tau$ will have components in homogeneities strictly larger than $\alpha$.
This is not allowed in the original formulation of a regularity structure by
Hairer \cite[Definition 2.1]{bib:hairer}. As we have seen in the main text,
this poses no problem, since our modified definition of a model (Definition \ref{def:modelGeneral})
allows for it.
We moreover believe that any transport that wants to
achieve the following lemma for a \textquotedblleft polynomial
model\textquotedblright\ is forced to do this.
\end{remark}

The definitions have been arranged so that $\Pi_{q}\tau$ and $\Pi_{p}%
\Gamma_{p\leftarrow q}\tau$ agree at $p$ to order $n,$ i.e. $\Pi_{q}\tau$ and
$\Pi_{p}\Gamma_{p\leftarrow q}\tau$ along with all derivatives up to order $n$
agree at $p.$

\begin{lemma}
\label{l.2.10}Let $\tau\in\mathcal{T}|_{q}$ and $p,z\in U$ where $U$ is a
sufficiently small neighborhood of $q.$ If $V$ is a differential operator of
order $k\leq n$ defined on $U,$ then
\[
|V\left[  \Pi_{q}\tau-\Pi_{p}\Gamma_{p\leftarrow q}\tau\right]  \left(
z\right)  |\lesssim_{V}|\tau|d(z,p)^{n+1-k}.
\]

\end{lemma}

\begin{proof}

Let
\[
g\left(  z\right)  =\left(  \Pi_{q}\tau\right)  \left(  z\right)
=\left\langle \tau,G^n_{q}\left(  z\right)  \right\rangle
\]
so that
\[
\left(  \Pi_{p}\Gamma_{p\leftarrow q}\tau\right)  \left(  z\right)  =\left(
\operatorname*{Tay}\nolimits_{p}^{n}g\right)  \left(  z\right)  .
\]
Using Corollary \ref{cor.tay3}, we have the estimate,%
\[
\left\Vert \nabla^{k}|_{z}\left[  g-\left(  \operatorname*{Tay}\nolimits_{p}%
^{n}g\right)  \right]  \right\Vert \leq\frac{1}{\left(  n+1-k\right)  !}%
\max_{0\leq t\leq1}\left\Vert \nabla^{n+1}|_{\gamma_{v}\left(  t\right)
}\left(  g-\operatorname*{Tay}\nolimits_{p}^{n}g\right)  \right\Vert \cdot
d\left(  p,z\right)  ^{n+1-k},
\]
where $v:=\exp_{p}^{-1}\left(  z\right)  .$ For $d\left(  z,p\right)
<\varepsilon$ and $0\leq t\leq1,$ let $\left[  p,z\right]  _{t}:=\exp\left(
t\exp_{p}^{-1}\left(  z\right)  \right)  $ so that $t\rightarrow\left[
p,z\right]  _{t}$ is the geodesic joining $p$ to $z$ parametrized by $\left[
0,1\right]  .$ Then we have
\begin{align*}
\max_{0\leq t\leq1}\left\Vert \nabla^{n+1}|_{\gamma_{v}\left(  t\right)
}\left(  g-\operatorname*{Tay}\nolimits_{p}^{n}g\right)  \right\Vert  &
=\max_{0\leq t\leq1}\left\Vert \nabla^{n+1}|_{\left[  p,z\right]  _{t}}\left(
g-\operatorname*{Tay}\nolimits_{p}^{n}g\right)  \right\Vert \\
&  \leq\max_{w:d\left(  w,p\right)  \leq d\left(  p,z\right)  }\left\Vert
\nabla^{n+1}|_{w}\left(  g-\operatorname*{Tay}\nolimits_{p}^{n}g\right)
\right\Vert
\end{align*}
and so we have
\[
\left\Vert \nabla^{k}|_{z}\left[  g-\left(  \operatorname*{Tay}\nolimits_{p}%
^{n}g\right)  \right]  \right\Vert \leq\frac{1}{\left(  n+1-k\right)  !}%
\max_{w:d\left(  w,p\right)  \leq d\left(  p,z\right)  }\left\Vert
\nabla^{n+1}|_{w}\left(  g-\operatorname*{Tay}\nolimits_{p}^{n}g\right)
\right\Vert \cdot d\left(  p,z\right)  ^{n+1-k}.
\]

\end{proof}

For the proof of the first half of Theorem \ref{thm:cgammaDgamma} below, it is
convenient to introduce the notion of a parallelism on a vector bundle, $E$,
over $M.$

\begin{definition}
[Diagonal domains]\label{d.2.11}Let $\mathcal{U}$ be an open set on $M$. An
open set $\mathcal{D}^{\mathcal{U}}\subseteq M\times M$ is a\textbf{
}$\mathcal{U}$ -- \textbf{diagonal domain}$\mathbb{\ }$if it contains the
diagonal of $\mathcal{U}$, that is $\Delta^{\mathcal{U}}:=\bigcup
_{p\in\mathcal{U}}\left(  p,p\right)  \subseteq\mathcal{D}^{\mathcal{U}}$. A
\textbf{local diagonal domain} is a $\mathcal{V}$ -- diagonal domain for some
nonempty open $\mathcal{V\subseteq}M$.

If $\mathcal{U}=M,$ we write $\mathcal{D}:=\mathcal{D}^{M}$ and refer to
$\mathcal{D}$ simply as a diagonal domain.
\end{definition}

\begin{definition}
[Parallelisms]\label{d.2.12}Let $E$ be a vector bundle over $M$ and
$\operatorname{Hom}\left(  E\right)  \rightarrow M\times M$ be the associated
vector bundle over $M\times M$ with fibers, $\operatorname{Hom}_{\left(
q,p\right)  }\left(  E\right)  :=L\left(  E_{p},E_{q}\right)  $ for $\left(
q,p\right)  \in M\times M,$ where $L\left(  E_{p},E_{q}\right)  $ denote the
set of all linear transformations from $E_{p}$ to $E_{q}.$ A smooth local
section $U$ $\in\Gamma\left(  \operatorname{Hom}\left(  E\right)  \right)  $
with domain $\mathcal{D}$ (i.e. $U\left(  q,p\right)  \in L\left(  E_{p}%
,E_{q}\right)  $ for all $\left(  q,p\right)  \in\mathcal{D}$) is called a
\textbf{parallelism }if $U\left(  p,p\right)  =I_{p}$. If $U$ is only defined
on a local diagonal domain, we refer to $U$ as a \textbf{local parallelism}.
\end{definition}

\begin{example}
[Parallel translation and parallelisms]\label{exa.2.13}One natural example of
a parallelism when $\left(  M,g\right)  $ is a Riemannian manifold and $E$ is
equipped with a covariant derivative, $\nabla^{E},$ is to define
\[
U^{\nabla}\left(  q,p\right)  :=//_{1}^{E}\left(  t\rightarrow\exp_{p}\left(
t\exp_{p}^{-1}\left(  q\right)  \right)  \right)  ,
\]
where $p,q\in M$ are \textquotedblleft close enough\textquotedblright\ so
there is a unique vector $v_{p}$ with minimum length such that $q=\exp
_{p}^{\nabla}\left(  v_{p}\right)  $ and $//_{\left(  \cdot\right)  }^{E}$
denotes the parallel translation operator on $E$ relative to $\nabla^{E}.$ For
our purposes below $E$ will be a bundle associated to $TM$ and $\nabla^{E}$
will be the induced connection on this bundle associated to the Levi-Civita
covariant derivative on $\left(  M,g\right)  .$
\end{example}

\begin{example}
[Charts and parallelisms]\label{exa.2.13a}Each chart $(\Psi,\U)$
induces a local parallelism on
$(T^{\ast}M)^{\otimes\ell}$ for any $\ell\in\mathbb{N}$ as follows. If
$A\in(T_{p}^{\ast}M)^{\otimes\ell}$ is expressed as%
\[
A=\sum_{i_{1},\dots i_{\ell}=1}^{d}A_{i_{1},\dots,i_{\ell}}d\Psi^{i_{1}}%
|_{p} \otimes\dots \otimes d\Psi^{i_{\ell}}|_{p},
\]
then we define $U^{\Psi}(q,p)A\in T_{q}^{\ast}M^{\otimes\ell}$ by
\[
U^{\Psi}(q,p)A=\sum_{i_{1},\dots i_{\ell}=1}^{d}A_{i_{1},\dots,i_{\ell}}%
d\Psi^{i_{1}}|_{q}\otimes\dots \otimes d\Psi^{i_{\ell}}|_{q}.
\]
In other words, $U^{\Psi}(q,p)$ is uniquely determined by requiring
\[
\left\langle U^{\Psi}(q,p)A,\frac{\partial}{\partial\Psi^{i_{1}}}|_{q}%
\otimes\dots\otimes\frac{\partial}{\partial\Psi^{i_{\ell}}}|_{q}\right\rangle
=\left\langle A,\frac{\partial}{\partial\Psi^{i_{1}}}|_{p}\otimes\dots
\otimes\frac{\partial}{\partial\Psi^{i_{\ell}}}|_{p}\right\rangle
\]
for all $q\in \U$ and $1\leq i_{1},i_{2},\dots i_{\ell}\leq d.$ [This example
is basically a special case of Example \ref{exa.2.13} where one takes $\nabla$
to be the flat connection, $D^{\Psi},$ defined in Notation \ref{not.chart}.]
\end{example}

With the aid of a parallelism, we can now define the notion of $\gamma$ --
H\"{o}lder section, $S,$ on $E.$ In what follows we assume that $E$ is
equipped with a smoothly varying inner product, $\left\langle \cdot
,\cdot\right\rangle _{E}.$ We do not necessarily assume that $\nabla^{E}$ is
compatible with $\left\langle \cdot,\cdot\right\rangle _{E}$ or that $U\left(
p,q\right)  $ is unitary for all $\left(  p,q\right)  \in\mathcal{D}.$

\begin{lemma}
  \label{r.2.15}
  Let $S$ be a continuous section of a vector bundle $E$.
  Let $(U, \mathcal{D}), (U', \mathcal{D})$ be parallelisms on $E$.
  Then for every compactum $K \subset \mathcal{D}$
  \begin{align*}
    ||U(q,p) S(p) - S(q)||
    \le C_K 
    \left( ||U'(q,p) S(p) - S(q)|| + d(p,q) \right), \forall p, q \in K.
  \end{align*}
\end{lemma}
\begin{proof}
  We work in a local trivialization.
  Let
  $U, U': \mathbb{R}^{d}\times\mathbb{R}^{d}\rightarrow GL\left(  \mathbb{R}%
  ^{N}\right)  $ be smooth functions such that $U\left(  x,x\right), U'\left(x,x\right)  =I,$ which
  we view to be a parallelism on the trivial bundle, $\mathbb{R}^{d}%
  \times\mathbb{R}^{N}$ over $\mathbb{R}^{d}.$ A continuous section of this
  bundle may be identified with a continuous function, $S:\mathbb{R}%
  ^{d}\rightarrow\mathbb{R}^{N}$
  Then
  \begin{align*}
    ||U(x,y) S(y) - S(x)||
    \le
    ||\left( U(x,y) - U'(x,y) \right) S(y)||
    +
    ||U'(x,y) S(y) - S(x)||.
  \end{align*}
  The statement then follows from smoothness of $U,U'$,
  the fact that they coincide at $x,x$ and local boundedness of $S$.
\end{proof}

\begin{lemma}
\label{lem:UPsi}Let $f\in C\left(  M\right)  ,$ $\gamma>0$ and $n=\lfloor
\gamma\rfloor\in\mathbb{N}_{0}.$ Then $f\in C^{\gamma}(M)$ (as in Definition
\ref{def:cgamma}) iff
$f$ a $n$-times continuously differentiable function on $M$
%$f\in C^{n}\left(  M\right)  $
and for any (local)
parallelisms $\left(  U\right)  $ on the vector bundle $\Sigma^{n}T^{\ast}M,$
$\Sym[ \nabla^n| f ]$ satisfies
\begin{equation}
|U(q,p)\Sym[ \nabla^n|_p f]-\Sym[ \nabla^n|_q f ]|\lesssim d(q,p)^{\gamma-n},
\label{equ.uholder}%
\end{equation}
where $\nabla$ is the Levi-Civita covariant derivative.
\end{lemma}

\begin{proof}
Recall from Definition \ref{def:cgamma}, that $f\in C\left(  M\right)  $ is in $C^\gamma(M)$
iff $f\circ\Psi^{-1}\in C^{\gamma}(\Psi(\U))$ for every coordinate chart
$(\Psi,\U).$ These conditions are equivalent to
%$f\in C^{n}\left(  M\right)  $
$f$ being $n$-times continuously differentiable
and the $n^{\text{th}}$ -- derivatives of $f\circ\Psi^{-1}$ are locally
$\left(  \gamma-n\right)  $-H\"{o}lder on $\Psi\left(  \U\right)  .$ The latter
condition may be expressed as saying%
\begin{equation}
|U^{\Psi}(q,p)D^{n}|_{p}f-D^{n}|_{q}f|\lesssim d(q,p)^{\gamma-n},
\label{equ.uj}%
\end{equation}
where $D=D^{\Psi}$ is the flat connection defined in Notation \ref{not.chart}.
From Lemma \ref{l.2.7bb} and Corollary \ref{cor.2.7} we may express
\begin{equation}
D^{n}f=\Sym[ \nabla^n f ]+Lf \label{equ.lot}%
\end{equation}
where $L$ is a linear differential operator of order at most $n-1.$ As $Lf$ is
%$C^{1}$
continuously differentiable
it follows that
\[
\left(  q,p\right)  \rightarrow U^{\Psi}(q,p)\left(  Lf\right)  _{p}-\left(
Lf\right)  _{q}\text{ }%
\]
is
%$C^{1}$
continuously differentiable
and vanishes at $q=p$ and therefore (by the fundamental theorem of
calculus)%
\begin{equation}
\left\vert U^{\Psi}(q,p)\left(  Lf\right)  _{p}-\left(  Lf\right)
_{q}\right\vert \lesssim d(q,p). \label{equ.lot2}%
\end{equation}
From \eqref{equ.lot} and \eqref{equ.lot2} it follows that
\eqref{equ.uj} is equivalent to%
\begin{equation}
|U^{\Psi}(q,p)\Sym[ \nabla^n|_p f]-\Sym[ \nabla^n|_q f ]|\lesssim
d(q,p)^{\gamma-n}. \label{equ.alt}%
\end{equation}
Lastly using Lemma \ref{r.2.15} we conclude that the estimates in \eqref{equ.alt} and \eqref{equ.uholder} are also equivalent.
\end{proof}

\begin{theorem}
\label{thm:polynomialRS} Fix $n\in\N$ and construct $\mcT$ and $(\Pi,\Gamma)$
as above. Then $\mcT$ is a regularity structure and $(\Pi,\Gamma)$ is a model
of transport precision $n+1$.
\end{theorem}

\begin{proof}
The fact that $\mcT$ is a regularity structure is immediate. Let us now set
$\delta=\delta_{M}$ to be the injectivity radius of $M$ and for $q\in M,$ let
$U_{q}:=\exp_{q}(B_{\delta_{M}}(0)).$

We have to check that
\[
||(\Pi,\Gamma)||_{\beta;M}<\infty.
\]
The homogeneity estimate, $|\langle\Pi_{p}\tau,\varphi_{p}^{\lambda}%
\rangle|\lesssim\lambda^{\ell}$, for $\tau\in\mcT_{\ell}|_{p}$ follows from
the fact that $\Pi_{p}\tau$ is a monomial of order $\ell$ in $\exp_{p}^{-1}%
$-coordinates. Lemma \ref{l.2.10} gives the transport precision, i.e.%
\[
|\langle\Pi_{q}\tau-\Pi_{p}\Gamma_{p\leftarrow q}\tau,\varphi_{p}^{\lambda
}\rangle|\lesssim\lambda^{n+1}\text{ for all }\tau\in\mcT|_{q}.
\]
%%%{\color{red}
%%%Let $p\in U_{q},$ $\ell\in\mathbb{N},$ and $W\in\left[  T_{p}M\right]
%%%^{\otimes\ell},$ then by \eqref{equ.tayest} and Corollary
%%%\ref{lem:theyRealize}
%%%\begin{align*}
%%%\langle\Sym\left[  \nabla^{m}|_{p}\tau_{\ell}\left(  \exp_{q}^{-1}\right)
%%%^{\otimes\ell}\right]  ,W\rangle &  =\sum_{i=1}^{m}\langle D^{i}|_{p}%
%%%\tau_{\ell}\left(  \exp_{q}^{-1}\right)  ^{\otimes\ell},Q_{i,n}\rangle\\
%%%&  \lesssim\sum_{i=1}^{m}d(p,q)^{\ell-i}|W|\\
%%%&  \lesssim d(p,q)^{\ell-m}|W|,
%%%\end{align*}

Let $D$ be the covariant derivative induced by the chart $\exp_{q}^{-1}$.
%Regarding $|\Gamma_{p \leftarrow q} \tau_\ell|_m \lesssim d(p,q)^{\ell-m}$
Using Lemma \ref{l.2.7bb} we get
\begin{align*}
\langle\Sym\left[  \nabla^{m}|_{p}\tau_{\ell}\left(  \exp_{q}^{-1}\right)
^{\otimes\ell}\right]  ,W\rangle &  =\sum_{i=1}^{m}\langle D^{i}|_{p}%
\tau_{\ell}\left(  \exp_{q}^{-1}\right)  ^{\otimes\ell},Q_{i,n}\rangle\\
&  \lesssim\sum_{i=1}^{m}d(p,q)^{\ell-i}|W|\\
&  \lesssim d(p,q)^{\ell-m}|W|,
\end{align*}
and hence $|\Gamma_{p\leftarrow q}\tau_{\ell}|_{m}\lesssim d(p,q)^{\ell-m}$,
which finishes the proof.
\end{proof}

We are finally able to characterize $C^{\gamma}(M)$ in terms of the
``polynomial'' regularity structure.

\begin{theorem}
\label{thm:cgammaDgamma} Let $\gamma\in(0,\infty)\setminus\mathbb{N}$ and $f:
M \rightarrow\R$ continuous. Then, $f\in C^{\gamma}\left(  M\right)  $ if and
only if there is $\hat{f}\in\mathcal{D}^{\gamma}(M,\mathcal{T})$\footnote{The
space of modelled distributions was defined in Definition
\ref{def:modelledDistributions}.} with $\hat{f}_{0}\left(  p\right)  =f\left(
p\right)  $. In that case,
\[
\hat{f}_{\ell}\left(  p\right)  =\Sym[\nabla^{\ell}|_{p}f].
\]

\end{theorem}

\begin{proof}
$\left(  \implies\right)  $ Let $f\in C^{\gamma}\left(  M\right)  $ and
define
\[
\hat{f}\left(  p\right)  := \sum_{\ell=0}^{\lfloor\gamma\rfloor}
\Sym[\nabla^{(\ell)}|_{p}f],
\]
i.e. $\hat{f}_{\ell}\left(  p\right)  :=\Sym[\nabla^{(\ell
)}|_{p}f]$ for $0\leq\ell\leq\lfloor\gamma\rfloor=:n.$ We have to check that
$\hat{f}\in\mathcal{D}^{\gamma}(M, \mathcal{T})$, i.e. for all $\ell
\leq\lfloor\gamma\rfloor$ and $d(p,q) < \delta$
\[
\left\vert (\hat{f}\left(  q\right)  -\Gamma_{q\leftarrow p}\hat{f}\left(
p\right)  )_{\ell}\right\vert \lesssim d\left(  p,q\right)  ^{\gamma-\ell}
\]
or equivalently, using the definition of $\Gamma_{q \leftarrow p}$, if
\[
g\left(  q\right)  :=\left(  \Pi_{p}\hat{f}\left(  p\right)  \right)  \left(
q\right)  = \left\langle \hat{f}\left(  p\right)  ,G^n_{p}\left(  q\right)
\right\rangle ,
\]
we must show
\begin{equation}
\label{equ.6}\left\vert \Sym\left[  \nabla^{\ell}|_{q} \left(  f-g\right)
\right]  \right\vert \lesssim d\left(  p,q\right)  ^{\gamma-\ell}.
\end{equation}

\textbf{$\ell=n$:}\newline Recall $\gamma-n\in(0,1]$. Now the term to bound in
\eqref{equ.6} reads as
\begin{align*}
|\Sym\left[  \nabla^{n}|_{q}f-\nabla^{n}|_{q}\left[  \sum_{i=0}^{n}%
\langle\nabla^{i}|_{p}f,\exp_{p}^{-1}(\cdot)^{\otimes i}\right]  \right]  |
&  \leq|\Sym\left[  \nabla^{n}|_{q}f-\nabla^{n}|_{q}\left[  \langle\nabla
^{n}|_{p}f,\exp_{p}^{-1}(\cdot)^{\otimes i}\right]  \right]  |\\
&  \qquad+\sum_{i=0}^{n-1}|\Sym\left[  \nabla^{n}|_{q}\left[  \langle
\nabla^{i}|_{p}f,\exp_{p}^{-1}(\cdot)^{\otimes i}\right]  \right]  |.
\end{align*}
By Lemma \ref{lem:theyRealize}
\[
\Sym\left[  \nabla^{n}|_{q}\left[  \langle\nabla^{i}|_{p}f,\exp_{p}^{-1}%
(\cdot)^{\otimes i}\right]  \right]  =0,\text{ at }q=p,
\]
and since the expression is smooth in $q$ we can focus on
\[
|\Sym\left[  \nabla^{n}|_{q}f-\nabla^{n}|_{q}\left[  \langle\nabla^{n}%
|_{p}f,\exp_{p}^{-1}(\cdot)^{\otimes i}\right]  \right]  |
\]

Define on the vector bundle $\Sigma^{n}T^{\ast}M$ the parallelism
\begin{align}
  \label{eq:U}
  U(q,p)S:=\nabla^{n}|_{q}\langle S,\exp_{p}^{-1}(\cdot)^{\otimes n}\rangle.
\end{align}

Then by Lemma \ref{lem:UPsi}
\begin{align*}
|\Sym[ \nabla^n|_q f ] - U(q,p) \Sym[ \nabla^n|_p f]| \lesssim d(q,p)^{\gamma-
n},
\end{align*}
so for $\ell= n$ we are done.

%$\ell = 0$:\\
%We need to show
%\begin{align*}
%f(q) - \sum_{i=0}^n \langle \Sym[ \nabla^n|_p f ], \exp_p^{-1}(q)^{\otimes n} \rangle
%\lesssim
%d(q,p)^{\gamma},
%\end{align*}
%but this follows immediately from Taylor's theorem (Theorem \ref{thm:taylor}).

$\ell= 0, \dots, n - 1$:
%This is Bruce's proof.
We need to show \eqref{equ.6}. It is enough to bound for $w \in T_{p} M$, with
$v := \exp^{-1}_{p}(q)$,
\begin{align*}
|\langle\nabla^{\ell}|_{q} \left(  f-g\right)  , ( //_{t}(\gamma_{v}) w
)^{\otimes\ell} \rangle| \lesssim|w|^{\ell} d\left(  p, q \right)
^{\gamma-\ell}.
\end{align*}
Here $//_{t}( \gamma_{v} ) : T_{\gamma_{v}(0)} M \to T_{\gamma_{v}(t)} M$
denotes the parallel transport along $\gamma_{v}(t) := \exp_{p}(tv)$.

For this purpose, define
\[
W_{t} :=//_{t}\left(  \gamma_{v}\right)  w_{1}\otimes\dots\otimes//_{t}\left(
\gamma_{v}\right)  w_{\ell},
\]
and $F:=f-g.$ Since $W_{t}$ and $\dot{\gamma}_{v}\left(  t\right)  $ are
parallel along $\gamma_{v}\left(  t\right)  $ it follows that
\[
\frac{d^{k}}{dt^{k}}\nabla_{W_{t}}^{\ell}F=\nabla_{\dot{\gamma}_{v}\left(
t\right)  ^{\otimes k}\otimes W_{t}}^{\ell+k}F~\forall~0\leq k\leq n-\ell.
\]
Therefore by Taylor's theorem and the fact that $\nabla^{m}|_{p} F_{p}=0$ for
$0\leq m\leq n$\footnote{This follows by the very construction of $g$ along
with Corollary \ref{cor.2.8} and Lemma \ref{lem:theyRealize}.}, we have%
\begin{align}
\nabla_{W_{1}}^{\ell}F  &  =\sum_{k=0}^{n-\ell-1}\frac{1}{k!}\nabla
_{v^{\otimes k}\otimes W_{0}}^{\ell+k}F+\frac{1}{\left(  n-\ell-1\right)
!}\int_{0}^{1}\left[  \nabla_{\dot{\gamma}_{v}\left(  t\right)  ^{\otimes
n-\ell}\otimes W_{t}}^{n}F\right]  \cdot\left(  1-t\right)  ^{n-\ell
-1}dt\nonumber\\
&  =\frac{1}{\left(  n-\ell-1\right)  !}\int_{0}^{1}\left[  \nabla
_{\dot{\gamma}_{v}\left(  t\right)  ^{\otimes n-\ell}\otimes W_{t}}%
^{n}F\right]  \cdot\left(  1-t\right)  ^{n-\ell-1}dt. \label{equ.7}%
\end{align}
Since $g$ is smooth we apply the fundamental theorem of calculus to find%
\begin{align*}
\nabla_{\dot{\gamma}_{v}\left(  t\right)  ^{\otimes\left(  n-\ell\right)
}\otimes W_{t}}^{n}g  &  =\nabla_{v^{\otimes\left(  n-\ell\right)  }\otimes
W_{0}}^{n}g+\int_{0}^{t}\nabla_{\dot{\gamma}_{v}\left(  \tau\right)
^{\otimes\left(  n-\ell\right)  +1}\otimes W_{\tau}}^{n+1}gd\tau\\
&  =\nabla_{v^{\otimes\left(  n-\ell\right)  }\otimes W_{0}}^{n}f+\int_{0}%
^{t}\nabla_{\dot{\gamma}_{v}\left(  \tau\right)  ^{\otimes\left(
n-\ell\right)  +1}\otimes W_{\tau}}^{n+1}gd\tau\\
&  =\nabla_{v^{\otimes\left(  n-\ell\right)  }\otimes W_{0}}^{n}f+O\left(
\left\vert v\right\vert ^{\left(  n-\ell\right)  +1}\left\vert W_{0}%
\right\vert \right)  .
\end{align*}

Using this estimate, it follows that
\[
\nabla_{\dot{\gamma}_{v}\left(  t\right)  ^{\otimes\left(  n-\ell\right)
}\otimes W_{t}}^{n}F=\nabla_{\dot{\gamma}_{v}\left(  t\right)  ^{\otimes
\left(  n-\ell\right)  }\otimes W_{t}}^{n}f-\nabla_{v^{\otimes\left(
n-\ell\right)  }\otimes W_{0}}^{n}f+O\left(  \left\vert v\right\vert ^{\left(
n-\ell\right)  +1}\left\vert W_{0}\right\vert \right)  .
\]

Since
\begin{align*}
\nabla_{\dot{\gamma}_{v}\left(  t\right)  ^{\otimes\left(  n-\ell\right)
}\otimes W_{t}}^{n} f  &  = \nabla^{n}|_{\gamma_{v}(t)} f \left(
//_{t}\left(  \gamma_{v} \right)  v \right)  ^{\otimes(n-\ell)} \otimes\left(
//_{t}\left(  \gamma_{v}\right)  w \right)  ^{\otimes\ell}\\
&  = \langle U(p, \gamma_{v}(t)) \nabla^{n}|_{\gamma_{v}(t)} f, v^{\otimes
(n-\ell)} w^{\otimes\ell} \rangle.
\end{align*}
%and {\color{red} what ??}
%\begin{align}
%\label{eq:what}
%\nabla^{n}|_{\gamma_{v}\left(  t\right)  } f
%//_{t}\left( \gamma_{v}\right)  ^{\otimes n}
%=
%U\left(  p,\gamma_{v}\left(  t\right) \right)
%\nabla^{n}|_{\gamma_{v}\left(  t\right)  } f,
%\end{align}
As shown in the step $\ell= n$, we then get
\[
\left\vert \nabla_{\dot{\gamma}_{v}\left(  t\right)  ^{\otimes\left(
n-\ell\right)  }\otimes W_{t}}^{n}f-\nabla_{v^{\otimes\left(  n-\ell\right)
}\otimes W_{0}}^{n}f\right\vert \leq Cd\left(  \gamma_{v}\left(  t\right)
,p\right)  ^{\gamma-n}\left\vert v\right\vert ^{n-\ell} |w|^{\ell}=
C\left\vert v\right\vert ^{\gamma-\ell} |w|^{\ell}.
\]
and hence%
\[
\left\vert \nabla_{\dot{\gamma}_{v}\left(  t\right)  ^{\otimes\left(
n-\ell\right)  }\otimes W_{t}}^{n}F\right\vert \leq\left[  C\left\vert
v\right\vert ^{\gamma-\ell}+O\left(  \left\vert v\right\vert ^{n-\ell
+1}\right)  \right]  |w|^{\ell}\leq C^{\prime}\left\vert v\right\vert
^{\gamma-\ell} |w|^{\ell}.
\]
Plugging this estimate back into \eqref{equ.7} shows,%
\[
\left\vert \nabla_{W_{1}}^{\ell}F\right\vert \leq C\left\vert v\right\vert
^{\gamma-\ell} |w|^{\ell},
\]
which completes the proof of \eqref{equ.6}.

$\left(  \Longleftarrow\right)  $

Recall that $\gamma\in(n,n+1]$, for some $n \in\N$.

\textbf{Step 1}:
%For $\ell = 0,\dots,n-1$ we will show that $\nabla^\ell f$ is differentiable and $\Sym[\nabla^{\ell+1} f] = \hat f_{\ell+1}$.
We will show that $f$ is $n$-times differentiable and $\Sym[\nabla^{\ell} f] =
\hat f_{\ell}$ for $\ell=0,\dots,n$. This will be done by induction.

So assume for some $\ell= 0,\dots,n-1$ we know that

\begin{itemize}
\item $f$ is $\ell$-times differentiable

\item $\Sym[\nabla^{i} f ] = \hat f_{i}, \quad i=0,\dots,\ell$
\end{itemize}

By Taylor's theorem (Theorem \ref{thm:taylor})
\begin{align}
\label{eq:one}f( \exp_{p}(v) ) = \sum_{j=0}^{\ell-1} \frac{1}{j!}
\Sym[ \nabla ^{i}|_{p} ] f \left(  v^{\otimes j} \right)  + \frac{1}%
{(\ell-1)!} \int_{0}^{1} (1-t)^{\ell-1} \Sym[ \nabla^{\ell}|_{\gamma_{v}(t)} ]
f \left(  \dot{\gamma}_{v}(t)^{\otimes\ell} \right)  .
\end{align}

Now by assumption
\begin{align*}
|\Sym[ \nabla^{\ell}(f - g) ]|_{q}| \lesssim d(q,p)^{\gamma-\ell},
\end{align*}
where
\[
g\left(  q\right)  :=\left(  \Pi_{p}\hat{f}\left(  p\right)  \right)  \left(
q\right)  = \left\langle \hat{f}\left(  p\right)  ,G^n_{p}\left(  q\right)
\right\rangle .
\]
Hence
\begin{align*}
\Sym[ \nabla^{\ell}|_{\gamma_{v}(t)} f ] =
\Sym[ \nabla^{\ell}|_{\gamma_{v}(t)} g ] + O(|tv|^{\gamma-\ell}).
\end{align*}

Plugging this into \eqref{eq:one} and using the fact that $|\dot{\gamma}%
_{v}(t)|=|v|$ we get
\begin{align}
  \label{eq:two}
f(\exp_{p}(v))=\sum_{j=0}^{\ell-1}\frac{1}{j!}\Sym[ \nabla^{i}|_{p} f ]\left(
v^{\otimes j}\right)  +\frac{1}{(\ell-1)!}\int_{0}^{1}(1-t)^{\ell
-1}\Sym[ \nabla^{\ell}|_{\gamma_{v}(t)} g ]\left(  \dot{\gamma}_{v}%
(t)^{\otimes\ell}\right)  +O(|v|^{\gamma}).
\end{align}
Now, since $g$ is smooth and $\frac{\nabla}{dt}\dot{\gamma}_{v}\left(
t\right)  =0,$ we have
\begin{align*}
\frac{d}{dt}\left[  \left[  \nabla^{\left(  \ell\right)  }g\right]
_{\gamma_{v}\left(  t\right)  }\dot{\gamma}_{v}\left(  t\right)  ^{\otimes
\ell}\right]   &  =\left[  \nabla^{\left(  \ell+1\right)  }g\right]
_{\gamma_{v}\left(  t\right)  }\dot{\gamma}_{v}\left(  t\right)
^{\otimes(\ell+1)}\text{ and}\\
\frac{d^{2}}{dt^{2}}\left[  \left[  \nabla^{\left(  \ell\right)  }g\right]
_{\gamma_{v}\left(  t\right)  }\dot{\gamma}_{v}\left(  t\right)  ^{\otimes
\ell}\right]   &  =\left[  \nabla^{\left(  \ell+2\right)  }g\right]
_{\gamma_{v}\left(  t\right)  }\dot{\gamma}_{v}\left(  t\right)
^{\otimes(\ell+2)}%
\end{align*}
and therefore by Taylor's theorem (in one variable) together with Lemma
\ref{l.2.4}
\begin{align}
\Sym\left[  \nabla^{\ell}g\right]  _{\gamma_{v}\left(  t\right)  }\dot{\gamma
}_{v}\left(  t\right)  ^{\otimes\ell}  &  =\left[  \nabla^{\ell}g\right]
_{\gamma_{v}\left(  t\right)  }\dot{\gamma}_{v}\left(  t\right)  ^{\otimes
\ell}\nonumber\\
&  =\left[  \nabla^{\ell}g\right]  _{p}v^{\otimes\ell}+t\left[  \nabla
^{\ell+1}g\right]  _{p}v^{\otimes(\ell+1)}+O\left(  \left\vert tv\right\vert
^{\ell+2}\right) \nonumber\\
&  =\left[  \nabla^{\ell}f\right]  _{p}v^{\otimes\ell}+t\hat{f}_{\ell
+1}(p)v^{\otimes(\ell+1)}+O\left(  \left\vert tv\right\vert ^{\ell+2}\right)
\nonumber\\
&  =\Sym\left[  \nabla^{\ell}f\right]  _{p}v^{\otimes\ell}+t\hat{f}_{\ell
+1}(p)v^{\otimes(\ell+1)}+O\left(  \left\vert tv\right\vert ^{\ell+2}\right)
\label{equ.12.new}%
\end{align}

A simple integration by parts argument shows%
\begin{equation}
\frac{1}{k!}\int_{0}^{1}\left(  1-t\right)  ^{k}tdt=\frac{1}{\left(
k+2\right)  !}. \label{eq:three}%
\end{equation}

\iffalse%
\begin{align*}
\frac{1}{k!}\int_{0}^{1}\left(  1-t\right)  ^{k}tdt  &  =-\frac{1}{\left(
k+1\right)  !}\int_{0}^{1}td\left(  1-t\right)  ^{k+1}dt\\
&  =\frac{1}{\left(  k+1\right)  !}\int_{0}^{1}\left(  1-t\right)
^{k+1}dt=\frac{1}{\left(  k+2\right)  !}.
\end{align*}
\fi

Combining \eqref{eq:two}, \eqref{equ.12.new} and \eqref{eq:three}, we get
\begin{equation}
f(\exp_{p}(v)) = \sum_{j=0}^{\ell-1}\frac{1}{j!}%
\Sym[ \nabla^{i}|_{p}f ]\left(  v^{\otimes j}\right)  + \frac{1}{\ell!}
\Sym[\nabla^{\ell}|_{p}f] \left(  v^{\otimes\ell}\right)  + \frac{1}%
{(\ell+1)!}\hat{f}_{\ell+1}(p)\left(  v^{\otimes(\ell+1)}\right)
+O(|v|^{\gamma}). \label{eq:four}%
\end{equation}

As $v\mapsto\exp_{p}\left(  v\right)  $ is a local diffeomorphism, it now
follows from \eqref{eq:four} that $f$ is $\ell+1$ times differentiable at $p$
and moreover since,
\begin{align*}
f( \exp_{p}(t v) ) = \sum_{j=0}^{\ell} \frac{t^{j}}{j!} \Sym[\nabla^{i}|_{p}]
f \left(  v^{\otimes j} \right)  + \frac{t^{\ell+1}}{(\ell+1)!} \hat
f_{\ell+1}(p) \left(  v^{\otimes(\ell+ 1)} \right)  + O(|t|^{\gamma}),
\end{align*}
we may conclude, using Lemma \ref{l.2.4} that
\begin{align*}
\nabla^{\ell+1}|_{p} v^{\otimes(\ell+1)} = \frac{d^{\ell+1} }{dt^{\ell+1}%
}|_{t=0} f(\exp_{p}(tv)) = \hat{f}_{\ell+1}(p) \left(  v^{\otimes(\ell+1)}
\right)  .
\end{align*}
Then by Remark \ref{r.2.2} it follows that
\begin{align*}
\Sym[\nabla^{\ell+1} f]_{p} = \hat{f}_{\ell+1}(p).
\end{align*}

\textbf{Step 2}:
So far we have shown that $f$ is $n$-times continuously differentiable
and that $\Sym[ \nabla^\ell|_p f ] = \hat f_\ell(p)$ for $\ell=0,\dots,n$.
%To show that $f \in C^\gamma(M)$ we will again apply Lemma \ref{lem:UPsi}.
%
Then with $U$ defined in \eqref{eq:U}
we have
\begin{align*}
  |\Sym[ \nabla^n|_q f] - U(q,p) \Sym[ \nabla^n|_p f] |
  &\le
  %|\Sym[ \nabla^n|_q f] - U(q,p) \Sym[ \nabla^n|_p f] |
  |\hat f(q) - \Gamma_{q \leftarrow p} \hat f(p)|_n
  +
  |\sum_{\ell \le n-1} |\nabla^n|_q \Pi_p \hat f_\ell(p)|.
\end{align*}
The second to last term is of order $d(q,p)^{\gamma-n}$ by assumption.
Moreover, for $\ell \le n-1$, by Corollary \ref{lem:theyRealize},
we have $\nabla^n|_p \Pi_p \hat f_\ell(p) = 0$.
Hence the last term is of order $d(q,p) \lesssim d(q,p)^{\gamma-n}$.
By Lemma \ref{lem:UPsi} we hence get that $f \in C^\gamma(M)$.

%%$\nabla^n f$ is $\gamma-n$-H\"older continuous: \dots
%It remains to show\marginpar{This argument seems messy.} that for any
%coordinate chart $(U,\Psi)$, any $n$-th order derivative of $f\circ\Psi^{-1}$
%is $(\gamma-n)$-H\"{o}lder continous. It is sufficient to confirm this for
%$\Psi=\exp_{p}^{-1}$, $U=\{z:d(p,z)<\delta/2\}$, for every $p$. Now
%\begin{align*}
%\partial_{i_{1}\dots i_{n}}|_{\Psi(q)}f\circ\Psi^{-1}-\partial_{i_{1}\dots
%i_{n}}|_{\Psi(p)}f\circ\Psi^{-1}  &  =\partial_{i_{1}\dots i_{n}}|_{\Psi
%(q)}f\circ\Psi^{-1}-\partial_{i_{1}\dots i_{n}}|_{\Psi(q)}g\circ\Psi^{-1}\\
%&  \qquad+\partial_{i_{1}\dots i_{n}}|_{\Psi(q)}g\circ\Psi^{-1}-\partial
%_{i_{1}\dots i_{n}}|_{\Psi(p)}f\circ\Psi^{-1}.
%\end{align*}
%
%
%Now
%\begin{align*}
%\partial_{i_{1} \dots i_{n}}|_{\Psi(q)} g \circ\Psi^{-1} - \partial_{i_{1}
%\dots i_{n}}|_{\Psi(p)} f \circ\Psi^{-1}  &  = \partial_{i_{1} \dots i_{n}
%}|_{\Psi(q)} g \circ\Psi^{-1} - \partial_{i_{1} \dots i_{n}}|_{\Psi(p)} g
%\circ\Psi^{-1} \lesssim d(q,p),
%\end{align*}
%since $g$ is smooth. Here we replaced $f$ by $g$ using Lemma
%\ref{lem:theyRealize}.
%
%Finally by Corollary \ref{cor.2.7} and Remark \ref{rem:50}
%\[
%\partial_{i_{1}\dots i_{n}}|_{\Psi(q)}(f-g)\circ\Psi^{-1}=\sum_{\ell=1}%
%^{n}\left\langle \Sym\nabla^{\ell}|_{q}(f-g),W_{\ell}\right\rangle .
%\]
%for some $W_{\ell}\in T_{q}^{\ell}M$. By assumption
%\[
%|\Sym[ \nabla^{\ell}|_q (f-g) ]|\lesssim d(q,p)^{\gamma-\ell},\qquad
%\ell=0,\dots,n.
%\]
%Hence
%\[
%\left\vert \partial_{i_{1}\dots i_{n}}|_{\Psi(q)}(f-g)\circ\Psi^{-1}%
%\right\vert \lesssim d(q,p)^{\gamma-n},
%\]
%as desired.
\end{proof}

\textbf{Acknowledgement:}
The first author  was supported in part by RTG 1845 and EPSRC grant EP/I03372X/1. 

The second author has received funding by the DAAD P.R.I.M.E. program.
He would like to thank 
Giuseppe Cannizzaro 
and Konstantin Matetski for discussion on the Schauder estimates.

\end{document}